\documentclass[12pt,reqno]{amsart}
\usepackage[dvipsnames,svgnames,x11names]{xcolor}

\usepackage{mathrsfs}

\usepackage{fullpage}
\usepackage{mathtools}
\makeatletter
\renewcommand\subsubsection{\@secnumfont}{\bfseries}%
\renewcommand\subsubsection{\@startsection{subsubsection}{3}
  \z@{.5\linespacing\@plus.7\linespacing}{-.5em}%
  {\normalfont\bfseries}}
  \makeatother

\usepackage{subfig}

\usepackage{amsfonts}
\usepackage{amsthm,amsmath}
\usepackage{amsbsy}
\usepackage{bm}
\usepackage{amssymb} 
\usepackage{wrapfig}

\usepackage{pgf,tikz}
\usepackage{graphicx}
\usetikzlibrary{trees,shadows}
\usetikzlibrary{arrows}
\usepackage{enumerate,enumitem}
\usepackage{bbm}
\usepackage{mathtools}

\usepackage[hyperindex,breaklinks]{hyperref}
\hypersetup{ colorlinks=true, linkcolor=blue, citecolor=blue,
  filecolor=blue, urlcolor=blue}
  \numberwithin{equation}{section} 
 \makeatletter
\def\paragraph{\@startsection{paragraph}{4}%
  \z@\z@{-\fontdimen2\font}%
  {\bfseries\itshape}}
  \makeatother

\DeclareMathOperator{\one}{\mathbbm{1}} 

\renewcommand{\O}[1]{O\left(#1\right)}
\usepackage[comma, sort&compress]{natbib}

\newtheoremstyle{newplain}%
  {}{}
  {\itshape}{}
  {\bfseries}{.}
  { }{\thmname{#1}\thmnumber{ #2}\thmnote{ (#3)}}

\theoremstyle{newplain}
    \newtheorem{theorem}{Theorem}[section]
    \newtheorem{lemma}[theorem]{Lemma}
    \newtheorem{proposition}[theorem]{Proposition}
    \newtheorem{corollary}[theorem]{Corollary}
    \newtheorem{claim}[theorem]{Claim}

\theoremstyle{definition} 
    \newtheorem{definition}[theorem]{Definition}

    \newtheorem{remark}[theorem]{Remark}

\DeclareMathOperator{\Ber}{\mathrm{Ber}}

\DeclareMathOperator{\Rr}{\mathbb{R}}

\DeclareMathOperator{\Z}{\mathbb{Z}}
\DeclareMathOperator{\N}{\mathbb{N}}

\DeclareMathOperator{\Ba}{\mathcal{B}}

\DeclareMathOperator{\ESD}{ESD}

\newcommand {\ota}{\dot{\iota}}
\DeclareMathOperator{\De}{d}

\DeclareMathOperator{\Tr}{Tr}

\DeclareMathOperator{\tr}{tr}
\DeclareMathOperator{\diag}{diag}

\DeclareMathOperator{\Id}{I}
\DeclareMathOperator{\dd}{d}
\DeclareMathOperator{\Var}{Var}
\DeclareMathOperator{\St}{\mathrm{S}}
\newcommand{\G}{\mathbb{G}}
\newcommand{\mfG}{\mathfrak{G}}
\newcommand{\E}{\mbox{${\mathbb E}$}}

\newcommand{\A}{\mathbf{A}}
\newcommand{\bigO}{\mbox{${\mathrm O}$}}

\newcommand{\xvec}{\mathbf{x}}
\newcommand{\Xvec}{\mathbf{X}}

\newcommand {\prob}{\mathbb{P}}
\newcommand{\ps}{\mathsf{P}}

\newcommand{\B}{\mathbf{B}}

\newcommand{\e}{\mathrm{e}}
\newcommand{\M}{\mathbf{M}}
\newcommand{\C}{\mathbb{C}}

\newcommand{\bA}{\mathbf{A}}
\newcommand{\tilbA}{\tilde{\mathbf{A}}}

\newcommand{\cA}{\overline{\mathbf{A}}} 

\newcommand{\Ver}{\mathbf{V}}

\newcommand{\Tpi}{\mathcal{T}_\pi} 

\global\long\def\ep{\mathbf{E}} 
\global\long\def\pr{\mathbf{P}} 

\begin{document}

\title[KBRG]{The spectrum of dense kernel-based random graphs}

\author[A. Cipriani]{Alessandra Cipriani}
\address{Department of Statistical Science,
University College London,
Gower Street,
London WC1E 6BT,
United Kingdom.}
\email{a.cipriani@ucl.ac.uk}
\author[R.~S.~Hazra]{Rajat Subhra Hazra}
 \address{University of Leiden, Mathematical Institute, Einsteinweg 55,
2333 CC Leiden, The Netherlands.}
\email{r.s.hazra@math.leidenuniv.nl}
\author[N. Malhotra]{Nandan Malhotra}
\address{University of Leiden, Mathematical Institute, Einsteinweg 55,
2333 CC Leiden, The Netherlands.}
\email{n.malhotra@math.leidenuniv.nl}
\author[M.Salvi]{Michele Salvi}
\address{Department of Mathematics,
University of Tor Vergata,
Via della Ricerca Scientifica 1, 00133 Rome, Italy.
}
\email{salvi@mat.uniroma2.it}

\begin{abstract}
   {Kernel-based random graphs (KBRGs) are a broad class of random graph models that account for inhomogeneity among  vertices. We consider KBRGs on a discrete $d-$dimensional torus $\mathbf{V}_N$ of size $N^d$. Conditionally on an i.i.d.~sequence of {Pareto} weights $(W_i)_{i\in \mathbf{V}_N}$ with tail exponent $\tau-1>0$, we connect any two points $i$ and $j$ on the torus with probability
		$$p_{ij}= \frac{\kappa_{\sigma}(W_i,W_j)}{\|i-j\|^{\alpha}} \wedge 1$$ for some parameter $\alpha>0$ and $\kappa_{\sigma}(u,v)= (u\vee v)(u \wedge v)^{\sigma}$ for some $\sigma\in(0,\tau-1)$.
		We focus on the adjacency operator of this random graph and study its empirical spectral distribution. For $\alpha<d$ and $\tau>2$, we show that a non-trivial limiting distribution exists as $N\to\infty$ and that the corresponding measure $\mu_{\sigma,\tau}$ is absolutely continuous with respect to the Lebesgue measure. $\mu_{\sigma,\tau}$ is given by an operator-valued semicircle law, whose Stieltjes transform is characterised by a fixed point equation in an appropriate Banach space. We analyse the moments of $\mu_{\sigma,\tau}$ and prove that the second moment is finite even when the weights have infinite variance. In the case $\sigma=1$, corresponding to the so-called scale-free percolation random graph, we can explicitly describe the limiting measure and study its tail.  }
\end{abstract}
\keywords{Empirical spectral distribution, free probability, kernel-based random graphs, spectrum of random matrices}
\subjclass{05C80, 46L54, 60B10, 60B20}
\date{\today}
\maketitle
\setcounter{tocdepth}{1}
\tableofcontents

\section{Introduction}

Kernel-based spatial random graphs encompass a wide variety of classical random graph models where vertices are embedded in some metric space. In their simplest form (see~\cite{jorritsma2023cluster} for a more complete exposition) they can be defined as follows. Let $V$ be the vertex set of the graph and sample a collection of weights $(W_i)_{i \in V}$, which are independent and identically distributed (i.i.d.), serving as marks on the vertices. 
	Conditionally on the weights, two vertices $i$ and $j$ are connected by an undirected edge with probability
	\begin{equation}\label{eq:connection}
		\mathbb{P}\left(i\leftrightarrow j \mid W_i, W_j \right) = \kappa(W_i, W_j)\|i-j\|^{-\alpha} \wedge 1 \, ,
	\end{equation}
	where $\kappa$ is a symmetric kernel, $\|i-j\|$ denotes the distance between the two vertices in the underlying metric space and $\alpha>0$ is a constant parameter. Common choices for $\kappa$ include:
	\[
	\begin{aligned}
	\kappa_{\text{triv}}(w,\, v) &\equiv 1, \quad & \kappa_{\text{strong}}(w,\,v) &= w \vee v, \\
	\kappa_{\text{prod}}(w, v) &= w\, v , \quad & 
	\kappa_{\text{pa}}(w, v) &= (w \vee v)(w \wedge v)^{\sigma_{\text{pa}}}.
	\end{aligned}
	\]
In the above $\sigma_{\text{pa}}= {\alpha(\tau-1)}/{d}-1$, where $\tau-1$ is the exponent of the tail distribution of the weights, such that the kernel $\kappa_{\text{pa}}$ mimics the form that appears in preferential attachment models~\citep{jorritsma2023cluster}, while the trivial kernel $\kappa_{\text{triv}}$ corresponds to the classical long-range percolation model~\citep{schulman1983long,newman1986one}. The kernel $\kappa_{\text{prod}}$ yields a model which is substantially equivalent to scale-free percolation, introduced in \cite{Deijfen:Remco:Hoogh}, which has connection probabilities of the form 
	\[
	1 - \exp\left( -W_i W_j \|i-j\|^{-\alpha} \right).
	\]
Various percolation properties for kernel-based spatial random graphs are known on $\Z^d$ and beyond (\cite{DHW2015, Hao2023, van2017explosion, gracar2021percolation, jorritsma2024large}, see also \cite{Deprez2019, Dalmau2021} for a version of the same in the continuum) as well as the behaviour of interacting particle systems on them~\citep{berger2002transience, heydenreich2017structures, komjathy2020explosion,cipriani2024scale,gracar2024contact,bansaye2024branching, komjathy2023four}. In contrast, their spectral properties, to the best of the authors' knowledge, have received less attention.

As a branch of random matrix theory, the study of the spectrum of random graphs has wide applications ranging from the study of random Schr\"odinger operators~\citep{carmona2012spectral,geisinger2015convergence} and quantum chaos in physics, to the analysis of community structures~\citep{bordenave2015non} and diffusion processes in network science, to the problems of spectral clustering~\citep{champion2020robust} and graph embeddings~\citep{gallagher2024spectral} in data science.
	 Many challenges remain unsolved in this area, even for the simplest models. As a prominent example, for bond percolation on $\mathbb{Z}^2$ it is known that the expected spectral measure has a continuous component if and only if $p > p_c$, but this result has not yet been established in higher dimensions~\citep{bordenave2017mean}. In this paper we begin the study of spectral properties of spatial inhomogeneous random graphs, which in turn have been proposed as models for several real-world networks (see e.g.~\cite{Dalmau2021}).

	\smallskip
	
	We will work with KBRGs in the typical setting where the weights $(W_i)$ have support in $[1,\infty)$ and the kernel $\kappa$ is an increasing function of the weights. Let us recall that in this case the vertices of KBRG random graphs on $\Z^d$ have almost surely infinite degree as soon as $\alpha<d$. Thus, as it happens in many percolation problems, the regime $\alpha > d$ would be the most appealing (and the toughest to tackle). In the present work we will focus instead on the dense case $\alpha<d$. We consider the discrete torus with $N^d$ vertices  equipped with the torus distance $\|\cdot\|$. The weights are sampled independently from a Pareto distribution with parameter $\tau - 1$ with $\tau > 2$. Conditionally on the weights, vertices $i$ and $j$ are connected independently from other pairs with probability given by \eqref{eq:connection} with a kernel of the form $\kappa_\sigma(w, v) \coloneqq (w\vee v)(w \wedge v)^{\sigma}$.
	It is worth noting a difference between our connection probability and that studied recently in \cite{jorritsma2023cluster, van2023local}, where the connection probabilities are given by
	\[
	\mathbb{P}\left( i\leftrightarrow j \mid W_i, W_j \right) = \left( \kappa_{\sigma}(W_i, W_j)\|i-j\|^{-d} \wedge 1 \right)^{\alpha}.
	\]
	The two forms can be made equivalent through a simple modification of the weights and an appropriate choice of $\alpha$.

We call $\mathbb G_N$ the random graph obtained with this procedure and 
	study the empirical spectral distribution of its adjacency matrix,  appropriately scaled. Note that when $\alpha = 0$ we recover the (inhomogenous) Erd\H{o}s--R\'enyi random graph (modulo a tweak inserting a suitable tuning parameter $\varepsilon_N$) . 
	In recent years, there has been significant research on inhomogeneous Erd\H{o}s–R\'enyi random graphs, which can be equivalently modelled by Wigner matrices with a variance profile. The limiting spectral distribution of the adjacency matrix of such graphs has been studied in \cite{CHHS, zhu2020graphon, Bose}, while local eigenvalue statistics have been analysed in \cite{Zhu:dumitriu, ajanki2019quadratic}. \cite{zhu2024central} studies the fluctuations of the linear eigenvalue statistics for a wide range of such inhomogeneous graphs. Additionally, various properties of the largest eigenvalue have been investigated in \cite{cheliotis2024limit, Husson, LDPChak, ducatez2024large}.  
	One of the most significant properties of the limiting spectral measure for random graphs is its absolute continuity with respect to the Lebesgue measure, which is closely tied to the concept of mean quantum percolation \citep{bordenave2017mean, anantharaman2021absolutely, arras2023existence}. Quantum percolation investigates whether the limiting measure has a non-trivial absolutely continuous spectrum. Recently, it was shown in \cite{arras2023existence} that the adjacency operator of a supercritical Poisson Galton-Watson tree has a non-trivial absolutely continuous part when the average degree is sufficiently large. Additionally, \cite{bordenave2017mean} demonstrated that supercritical bond percolation on $\mathbb{Z}^d$ has a non-trivial absolutely continuous part for $d = 2$. 
	These results motivate similar questions for KBRGs.
	
	\smallskip
	
	\noindent {\bf{Our contributions: results and proofs.}} Here below we showcase our main results and the novelties of our proofs Recall that we work in the regime $\alpha<d$ and $\tau>2$. We also restrict to values of $\sigma$ in $(0,\tau-1)$.  
	\begin{enumerate}[leftmargin=*]
		\itemsep0.3em
		\item In Theorem \ref{theorem:main} we show that, after scaling the adjacency matrix of $\mathbb G_N$ by $c_{0}N^{(d-\alpha)/2}$, the empirical spectral distribution converges weakly in probability to a deterministic measure $\mu_{\sigma,\tau}$. The classical approach to proving the convergence of the empirical distribution is generally through either the method of moments or the Stieltjes transform. However, the limiting measure is expected to be heavy-tailed (see Figure \ref{fig:tailhist}) and so it is not  determined by its moments. As a consequence, we cannot directly apply the method of moments. To overcome this issue, we pass through a truncation argument where we impose a maximal value to the weights, reducing the problem to well-behaved measures. To simplify the method of moments, we further reduce the model by substituting the adjacency matrix of $\mathbb G_N$ with a Gaussian matrix whose entries are centered and have roughly the same variance as before.  This is made possible by a classical result of \cite{chatterjee2005simple}. Once we have shifted our attention to this simpler Gaussianized matrix with bounded weights, we can use the classical method of moments using finding its moments is made possible by a combinatorial argument on partitions and their graphical representation. Finally we remove the truncation effect. 
		\item In Theorem \ref{theorem:tail} we investigate the graph corresponding to $\kappa_{\text{prod}}$, that is, when $\sigma=1$. In this case we can explicitly identify $\mu_{1,\tau}$ as the free multiplicative convolution of the semicircle law and the measure of the weight distribution. In the $\sigma =1$ case the moment expression derived in Theorem~\ref{theorem:main} simplifies, so the challenge is to recover the limiting measure from those moments. This is made possible thanks to extension of the free multiplicative convolution to measures with unbounded support by~\cite{arizmendi2009}. Furthermore, we show that $\mu_{1,\tau}$ has power-law tails with exponent $2(\tau - 1)$. This is based on a Breiman-type argument for free multiplicative convolutions~\citep{Bartosz:Kamil}.
		\item In Theorem \ref{theorem:secondmoment} we explicitly derive the second moment of $\mu_{\sigma,\tau}$ and prove that it is finite and non-degenerate. The proof is based on the ideas of~\citet[Theorem 2.2]{Chakrabarty:Hazra:Sarkar:2016} This result is noteworthy because our weight distribution may exhibit infinite variance in the chosen range of parameters. To show that the second moment is finite, we need to establish the uniform integrability of a sequence of measures converging to the limiting measure. This is achieved through an extension of Skorohod's representation theorem for measures that converge weakly in probability.
		\item In Theorem \ref{theorem:absolutecontinuity} we prove that $\mu_{\sigma,\tau}$ is absolutely continuous. What makes the result possible is that we are able to split the original matrix as a free sum of a standard Wigner matrix and another Wigner matrix with a carefully chosen variance profile (yielding, as a by-product, another characterisation of the limit measure $\mu_{\sigma,\tau}$). We show that $\mu_{\sigma,\tau}$ is the free additive convolution of a semicircle law and another measure. Once this is established, the result is a consequence of~\cite{Biane97}.
		\item In Theorem \ref{theorem:stieltjesfinal} we provide an analytical description of $\mu_{\sigma,\tau}$ when $\tau>3$ and $\sigma<\tau-2$. Removing the truncation in the method of moments proof of Theorem \ref{theorem:main} does not yield an explicit characterization of the limiting measure. On the other hand, certain moment recursions for the truncated Gaussian matrix that appear in the proof can be used to derive properties of $\mu_{\sigma,\tau}$ through the Stieltjes transform.  When the weights are bounded, the limiting measure corresponds to the operator-valued semicircle law (\cite{Speicher2011}). Its transform can be expressed in terms of functions solving an analytic recursive equation (see \cite{avena, zhu2020graphon} for similar results in other random graph ensembles). In our case, when the weights are heavy-tailed, this is no longer possible. We achieve instead convergence of the analytic recursive equation by constructing a suitable Banach space and demonstrating that it forms a contractive mapping.
	\end{enumerate}
	
    \noindent\textbf{Outline of the article.} In Section~\ref{sec:main} we will define the model and state precisely the main results. In Section~\ref{sec:prel_lemmas} we will give some auxiliary results which will be used to prove the main theorems in the rest of the article. More precisely, in Section~\ref{sec:proof_thm_main} we will prove the existence of the limiting ESD and in Section~\ref{sec:tail} we will give estimates on its tail behavior. In Section~\ref{sec:non-deg} we will prove the non-degeneracy of the limiting measure and in Section~\ref{sec:absolutecontinuity} we will show its absolute continuity. Finally, Section~\ref{sec:stieltjes} is devoted to describing the Stieltjes transform of the limiting ESD.

\section{Set-up and main results}\label{sec:main}
\subsection{Random graph models}
To introduce our models, we use $a \wedge b$ to denote the minimum of two real numbers $a$ and $b$, and $a \vee b$ to denote their maximum.


\begin{itemize}[leftmargin=*]
				\itemsep0.3em
    \item[(a)] {\bf Vertex set:} the vertex set is $\Ver_N\coloneqq \{1,\,2,\,\ldots,\,N\}^d$. The vertex set is equipped with torus the distance $\|i - j\| $, where
    \[
    \|i - j\| = \sum_{\ell=1}^d |i_\ell - j_\ell| \wedge (N - |i_\ell - j_\ell|).
    \]

    \item[(b)] {\bf Weights:} the weights $(W_i)_{i \in \Ver_N}$ are i.i.d.\ random variables sampled from a Pareto distribution $W$ (whose law we denote by $\mathbf{P}$) with parameter $\tau - 1$, where $\tau > 1$. That is,
    \begin{equation}\label{eq:paretolaw}
    \mathbf{P}(W > t) = t^{-(\tau - 1)}\one_{\{t \geq 1\}}+\one_{\{t < 1\}}.
    \end{equation}

    \item[(c)] {\bf Kernel:} the kernel function $\kappa_\sigma: [0, \infty) \times [0, \infty) \to [0, \infty)$  determines how the weights interact. In this article, we focus on kernel functions of the form
    \begin{equation}\label{eq:kappa}
    \kappa_\sigma(w, v) \coloneqq (w\vee v)(w \wedge v)^{\sigma},
    \end{equation}
    where $\sigma \geq 0$. 


    \item[(d)] {\bf Long-range parameter:} $\alpha > 0$ tunes the influence of the distance between vertices on their connection probability.

    \item[(e)] {\bf Connectivity function:} conditional on the weights, each pair of distinct vertices $i$ and $j$ is connected independently with probability $P^W(i\leftrightarrow j)$ given by 
    \begin{equation}\label{connection_proba}
    P^W(i\leftrightarrow j) \coloneqq \mathbb{P}(i \leftrightarrow j \mid W_i, W_j) = \frac{\kappa_\sigma(W_i, W_j)}{\|i - j\|^\alpha} \wedge 1.
    \end{equation}
    We will be using the short-hand notation $p_{ij} := \mathbb{P}(i \leftrightarrow j \mid W_i, W_j)$ for convenience. Note that the graph does not have self-loops (see Remark~\ref{rem:selfloop}).
\end{itemize}

The associated graph is connected, as nearest neighbours with respect to the torus distance are always linked. 

\subsection{Spectrum of a random graph} Let us denote the random graph generated by our choice of edge probabilities by $\mathbb{G}_N$. Let $\mathbb{A}_{\mathbb{G}_N}$ denote the adjacency matrix (operator) associated with this random graph, defined as
\[
\mathbb{A}_{\mathbb{G}_N}(i,j) = 
\begin{cases} 
1 & \text{if } i \leftrightarrow j, \\ 
0 & \text{otherwise}.
\end{cases}
\]
Since the graph is finite, the adjacency matrix is always self-adjoint and has real eigenvalues. For $\alpha < d$, the eigenvalues require a scaling, which turns out to be independent of the kernel in our setup. Here we assume $\sigma \in (0,\tau-1)$ and $\tau > 2$, ensuring that the vertex weights $(W_i)_{i \in \Ver_N}$ have finite mean. We define the scaling factor as
\begin{equation}\label{eq:def_c}
c_N = \frac{1}{N^d} \sum_{i \neq j \in \Ver_N}\frac{1}{\|i-j\|^{\alpha}} \sim c_0 N^{d-\alpha},
\end{equation}
where $c_0$ is a constant depending on $\alpha$ and $d$, and for two functions $f(\cdot)$ and $g(\cdot)$ we use $f(t)\sim g(t)$ to indicate that their quotient $f(t)/g(t)$ tends to one as $t$ tends to infinity. The scaled adjacency matrix is then defined as
\begin{equation}\label{eq:scaledadjacency}
\A_N \coloneqq \frac{\mathbb{A}_{\mathbb{G}_N}}{\sqrt{c_N}}.
\end{equation}

The empirical measure that assigns a mass of $1/N^d$ to each eigenvalue of the $N^d \times N^d$ random matrix $\mathbf{A}_N$ is called the Empirical Spectral Distribution (ESD) of $\mathbf{A}_N$, denoted as
\[
\ESD\left(\mathbf{A}_N\right) \coloneqq \frac{1}{N^d} \sum_{i=1}^{N^d} \delta_{\lambda_i},
\]
where $\lambda_1 \leq \lambda_2 \leq \ldots \leq \lambda_{N^d}$ are the eigenvalues of $\mathbf{A}_N$.

\subsection{Main results}
We are now ready to state the main result of this article. Let $\mu_{\sqrt{W}}$ and $\mu_{W}$ denote the laws of $\sqrt{W}$ and $W$, respectively. Here onwards, let $\mathbb{P} = \mathbf{P} \otimes P^W$ represent the joint law of the weights and the edge variables. Note that $\mathbb{P}$ depends on $N$, but we omit this dependence for simplicity. Let $\E, \mathbf{E}$, and $E^W$ denote the expectation with respect to $\prob, \mathbf{P}$, and $P^W$ respectively. 
Furthermore, if $(\mu_N)_{N \geq 0}$ is a sequence of probability measures, we write $\lim_{N \to \infty} \mu_N = \mu_0$ to denote that $\mu_0$ is the weak limit of the measures $\mu_N$. Since the empirical spectral distribution is a random probability measure, we require the notion of convergence in probability in the context of weak convergence.

The Lévy-Prokhorov distance $d_L: \mathcal{P}(\mathbb{R})^2 \to [0, +\infty)$ between two probability measures $\mu$ and $\nu$ on $\mathbb{R}$ is defined as
\[
d_L(\mu, \nu)
    := \inf \big\{\varepsilon > 0 \mid \mu(A) \leq \nu\left(A^\varepsilon\right) + \varepsilon \,\text{ and }\, \nu(A) \leq \mu\left(A^\varepsilon\right) + \varepsilon \quad \forall \,A \in \mathcal{B}(\mathbb{R})\big\},
\]
where $\mathcal{B}(\mathbb{R})$ denotes the Borel $\sigma$-algebra on $\mathbb{R}$, and $A^\varepsilon$ is the $\varepsilon$-neighbourhood of $A$. For a sequence of random probability measures $(\mu_N)_{N \geq 0}$, we say that
\[
\lim_{N \to \infty} \mu_N = \mu_0 \text{ in } \mathbb{P}\text{-probability}
\]
if, for every $\varepsilon > 0$,
\[
\lim_{N\to\infty}\mathbb{P}(d_L(\mu_N, \mu_0) > \varepsilon)= 0.
\]

The first result states the existence of the limiting spectral distribution of the scaled adjacency matrix. 
\begin{theorem}[{Limiting spectral distribution}]\label{theorem:main}
    Consider the random graph $\G_N$ on $\Ver_N$ with connection probabilities given by \eqref{connection_proba} with parameters $\tau>2$, $0<\alpha< d$ and $\sigma\in (0,\tau-1)$.  Let $\ESD(\A_N)$ be the empirical spectral distribution of $\A_N$ defined in \eqref{eq:scaledadjacency}. Then there exists a deterministic measure $\mu_{\sigma,\tau}$ on $\mathbb R$ such that
    $$
    \lim_{N\to \infty}\ESD(\A_N)=\mu_{\sigma,\tau}\qquad\text{ in $\,\prob$--probability}\,.
    $$
    
\end{theorem}

The remaining results focus of the properties of the limiting measure. First we note that when we set $\sigma=1$ we can explicitly identify the limiting measure in terms of free multiplicative convolution. We refer the reader to~\citet[Section 5.2.3]{anderson2010introduction} for an exposition on free multiplicative and additive convolutions. 

For two probability measures $\mu$ and $\nu$ the free multiplicative convolution $\mu\boxtimes \nu$ of the two measures is defined as  the law of the product $ab$ of free, random, non-commutative operators $a$ and $b$, with laws $\mu$ and $\nu$ respectively. 
The free multiplicative convolution for two non-negatively supported measures was introduced in \cite{bercovici1993free}. Note that the semicircle law is not non-negatively supported and hence we use the extended definition of~\cite{arizmendi2009} for the multiplicative convolution. 

\begin{theorem}[{Limiting ESD for $\sigma=1$}]\label{theorem:tail}
  Consider the KBRG for $\sigma=1$, while $\alpha,\,\tau$ are as in the assumptions of Theorem~\ref{theorem:main}. The the limiting spectral distribution $\mu_{1,\tau}$ is given by
    $$\mu_{1,\tau}= \mu_{sc}\boxtimes \mu_{W}\,,$$
    where $\mu_{sc}$ is the semicircle law
\[
\mu_{sc}({\dd x}) =\frac{1}{2\pi} \sqrt{4-x^2}\one_{|x|\leq 2}\dd{x}
    \]
    and $\boxtimes$ is the free multiplicative convolution of the two measures.
    Moreover, the limiting measure $\mu_{1,\tau}$ has a power-law tail, that is, 
  $$
  \mu_{1,\tau}(x,\infty)\sim \frac{1}{2} \big(m_1(\mu_{W})\big)^{\tau-1} x^{-2(\tau-1)}\, \text{ as $x\to \infty$,}
  $$
  where $m_1(\nu)$ denotes the first moment of the probability measure $\nu$.
\end{theorem}

In the general case, it is hard to explicitly identify the limiting measure, so we present some characterisations of it. Since we do not impose that $\tau>3$ and consequently the weights can have infinite variance, it is not immediate if the second moment of the limiting measure is non-degenerate and finite. We prove this in the following result.

\begin{theorem}[{Non-degeneracy of the limiting measure}]\label{theorem:secondmoment}

Under the assumptions of Theorem \ref{theorem:main}, the second moment of the limiting measure $\mu_{\sigma,\tau}$ is given by
$$\int_{\mathbb R} x^2 \mu_{\sigma,\tau}(\De x)= (\tau-1)^2\int_1^\infty\int_1^\infty \frac{1}{(x\wedge y)^{\tau-\sigma}(x\vee y)^{\tau-1}} \De x \De y\in (0,\infty).$$ 
Moreover, for $p\in \mathbb N$ and $p< ({\tau-1})/({\sigma\vee 1})$, we have $\int_{\mathbb{R}} \left|x\right|^{2p} \mu_{\sigma,\tau}(\De \, x)<\infty$. 
\end{theorem}

We state the following result as an independent theorem as the absolute continuity of the KBRG model deserves to be treated separately. 
\begin{theorem}[{Absolute continuity}]\label{theorem:absolutecontinuity}
 Let $\tau>2$ and $\sigma\in (0,\tau-1)$, then $\mu_{\sigma,\tau}$ is symmetric and absolutely continuous with respect to the Lebesgue measure on $\Rr$.
\end{theorem}

We conclude the main results by providing an analytic description of the limiting measure in terms of its Stieltjes transform when we slightly restrict our parameters. Recall that, for $z \in \mathbb{C}^+$, where $\mathbb{C}^+$ denotes the upper half-plane of the complex plane, the Stieltjes transform of a  measure $\mu$ on $\mathbb{R}$  is given by  
\begin{equation}\label{def:ST}
\St_\mu(z) = \int_{\mathbb{R}} \frac{1}{x-z} \mu(\mathrm{d}x)\,.
\end{equation}

\begin{theorem}[Stieltjes transform]\label{theorem:stieltjesfinal}
Let $0<\alpha<d$, $\tau > 3$ and $\sigma < \tau - 2$. Then there exists a unique analytic function $a^\ast$ on $\mathbb{C}^+ \times [1, \infty)$ such that  
\[
\St_{\mu_{\sigma,\tau}}(z) = \int_1^{\infty} a^\ast(z, x) \mu_W(\mathrm{d}x),
\]
where we recall that $\mu_W$ is the law of the random variable $W$. 
\end{theorem}

The function $a^\ast$ in the above theorem turns out to be a fixed point of a contraction mapping on an appropriate Banach space. The equation above shares similarities with the quadratic vector equations introduced and studied in \cite{ajanki2019quadratic}, although in our setting the measures have unbounded support. The properties and the proof of Theorem \ref{theorem:stieltjesfinal} are discussed in Section \ref{sec:stieltjes}.

\begin{remark}[Higher dimensions]\label{rem:high_d}
  While we have presented our results for $0<\alpha <d$, our proofs are worked out in the $d=1$ setup. This is in order to avoid notational complications that would especially affect the clarity of Theorem~\ref{theorem:main}. The limiting spectral distribution and its properties remain unchanged for $d>1$.
\end{remark}
\subsection{Examples, simulations and discussion}

Firstly, in Figure~\ref{fig:twohist} we plot the eigenvalue distribution of the adjacency matrix of two realizations of kernel-based graphs with different parameters, indicated at the top of the image.
\begin{figure}[ht!]
 \centering
\hfill
\subfloat{\includegraphics[width=0.5\linewidth]{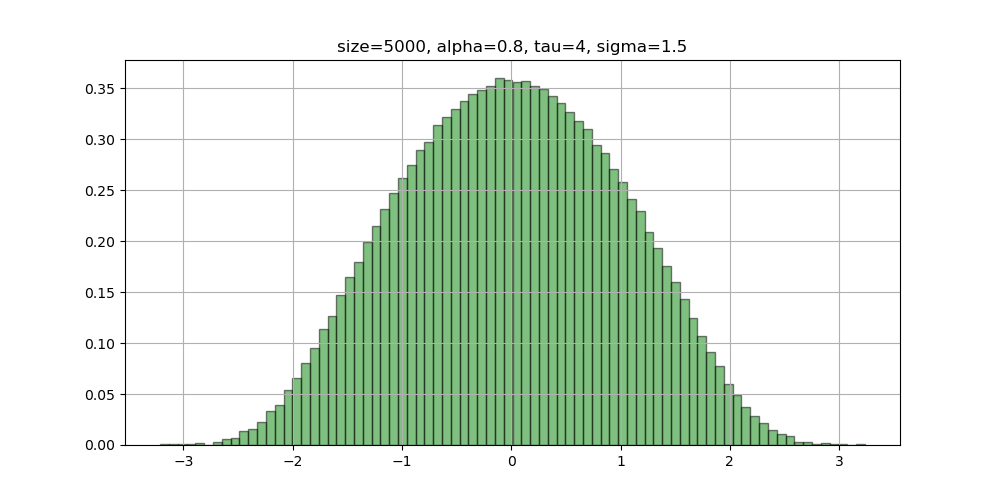}}
    \hfill    \subfloat{\includegraphics[width=0.5\linewidth]{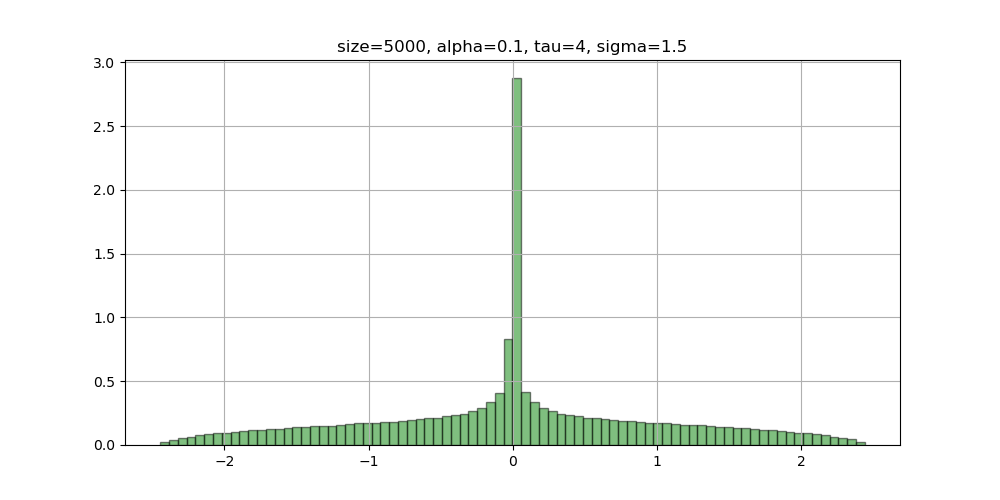}}
    \hfill
    \caption{Eigenvalue distribution for two KBRG realizations}
    \label{fig:twohist}
\end{figure}
\begin{figure}[ht!]
 \centering
\includegraphics[scale=.5]{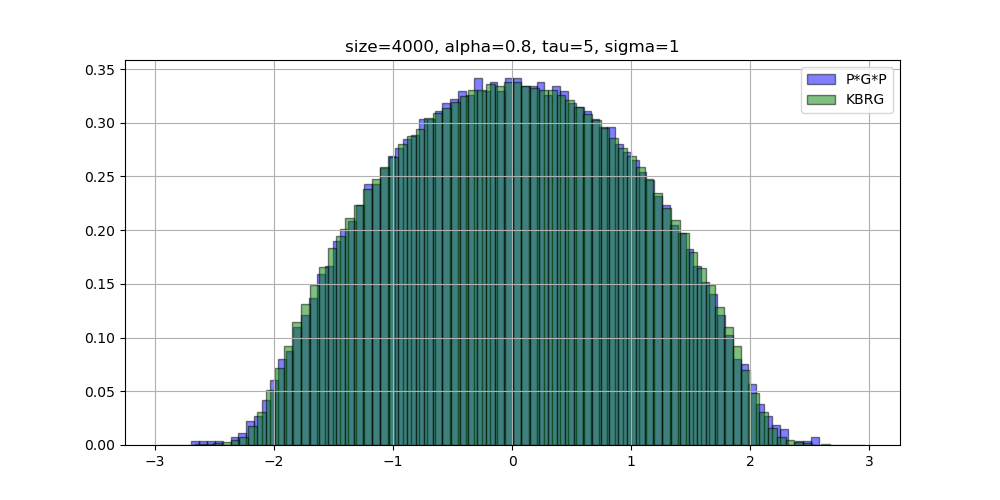}    
    \caption{KBRG eigenvalue distribution and $P_N G_N P_N$ distribution. }
    \label{fig:sigma1}
\end{figure}
Secondly, in Figure~\ref{fig:sigma1} we sample 10 realizations of scale-free percolation adjacency matrices of size $4000\times 4000$ with $\sigma=1$ and plot their eigenvalues (in green). We superpose to them the eigenvalues of the product $P_N G_N P_N$ of a GUE matrix $G_N$ with a diagonal matrix $P_N$ with i.i.d.~entries distributed as $\sqrt{Pareto(\tau)}$ (in blue). Note that by~\citet[Remark 14.2]{Nica:Speicher},~\citet[Remark 4.3]{chakrabarty2018note}, the a.s.~limiting ESD of $P_N G_N P_N$ is $\mu_{sc}\boxtimes \mu_W$. All matrices are centred and rescaled by the sample second moment.
Thirdly, to elucidate the tail behaviour of the limiting ESD when $\sigma=1$ (Theorem~\ref{theorem:tail}) we draw in Figure~\ref{fig:tailhist} the empirical survival function of the eigenvalues of a matrix of size $7000\times 7000$ in $x\geq 1.5.$ 

Finally, we provide in Figure~\ref{fig:normalesd} a simulation of the eigenvalues of the Gaussian matrix $\tilbA_{N,m,g}$ (see~\eqref{eq:Gaussianmatrix}) when $\alpha=0$ and $N=6000$. We compare this picture with the right-hand side of Figure~\ref{fig:twohist}, which has a small $\alpha$. We conjecture that the atom appearing in the latter is due to high connectivity of the kernel-based realization (if $\alpha=0$, for all $i,
\,j$ we have that $p_{ij}$ is identically one in~\eqref{connection_proba}), whilst in the Gaussian setup this trivialization does not arise.
\begin{figure}[ht!]
 \centering  \includegraphics[scale=0.6]{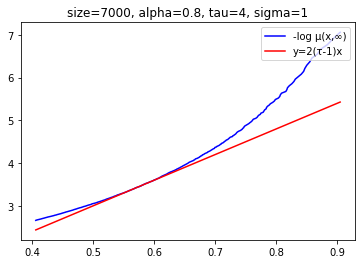}\caption{Negative of the log-empirical survival function and tails of Theorem~\ref{theorem:tail} for $x\geq 1.5$.}
    \label{fig:tailhist}
\end{figure}
\begin{figure}[ht!]
 \centering  \includegraphics[scale=0.6]{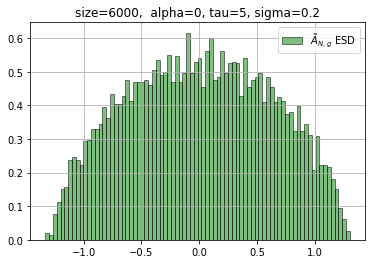}\caption{ESD for $\tilbA_{N,m,g}$.}
    \label{fig:normalesd}
\end{figure}

\begin{remark}[{Sparse case}] We expect the case $\alpha>d$ to be very different due to the sparse nature of the graph. There has been a significant development in the area of spectral properties of sparse random graphs using the techniques of local weak convergence \citep{bordenave2010resolvent, bordenave2017mean, bordenave:lelarge:salez}. However, it is not immediately clear whether these techniques can be employed in our framework in order to determine the properties of the limiting measure: the underlying random graph generated in our model will not be tree-like to begin with. We plan to address this case in a future work.
\end{remark}

\section{Notation and preliminary lemmas}\label{sec:prel_lemmas}

In this section we fix some notation and collect some technical lemmas that will be used in the proofs of our main results.

\subsection{Notation} 
We will use the Landau notation $o_N,\,\bigO_N$ indicating in the subscript the variable under which we take the asymptotic (typically this variable will grow to infinity unless otherwise specified). Universal positive constants are denoted as $c,\,c_1,\,\ldots,$ and their value may change with each occurrence. For an $N\times N$ matrix $A=(a_{ij})_{i,\,j=1}^N$ we use $\Tr(A)\coloneqq\sum_{i=1}^N a_{ii} $ 
for the trace and $\mathrm{tr}(A)\coloneqq N^{-1}\Tr(A)$ for the normalised trace. When $n\in\N$ we write $[n]\coloneqq\{1,\,2,\,\ldots,\,n\}.$ We denote the cardinality of a set $A$ as $\#A$, and, with a slight abuse of notation, $\#\sigma$ also denotes the number of cycles in a permutation $\sigma$.

\subsection{Technical lemmas}\label{sec: technical lemmas SFP} 
For bounding the $d_L$ distance between the ESDs of two matrices, we will need the following inequality, due to Hoffman and Wielandt.

\begin{proposition}[Hoffman-Wielandt inequality{~\citep[Corollary A.41]{Bai-silverstein}}]\label{prop:hoffman-wielandt}
Let $\mathbf{A}$ and $\mathbf{B}$ be two $N \times N$ normal matrices and let  $\ESD(\A)$ and $\ESD(\mathbf{B})$  be their ESDs, respectively. Then,
\begin{equation}\label{HWinequality}
d_L\left(\ESD(\A), \ESD(\mathbf{B})\right)^3 \leq \frac{1}{N} \Tr\left[(\mathbf{A}-\mathbf{B})(\mathbf{A}-\mathbf{B})^*\right] .
\end{equation}
Here $\A^*$ denotes the conjugate transpose of $\A.$ 
Moreover, if $\A$ and $\mathbf{B}$ are two Hermitian matrices of size $N\times N$, then
\begin{equation}\label{HW2}
\sum_{i=1}^N\left(\lambda_i(\A)-\lambda_i(\mathbf{B})\right)^2 \leq \Tr[(\A-\mathbf{B})^2].
\end{equation}
\end{proposition}

The next two straightforward lemmas control the tail of the product of two Pareto random variables and the expectation of a truncated Pareto.
\begin{lemma}\label{lemma:twotails}
Let $X$ and $Y$ be two independent Pareto r.v.'s with parameters $\beta_1$ and $\beta_2$ respectively, with $\beta_1\leq \beta_2$. There exist constants $c_1=c_1(\beta_1,\beta_2)>0$ and $c_2=c_2(\beta_1)>0$ such that 
$$
\mathbf P(XY>t) 
    = \begin{cases} c_1t^{-\beta_1} & \text{ if }\beta_1< \beta_2\\
c_2t^{-\beta_1}\log t & \text{ if } \beta_1= \beta_2.
\end{cases}
$$
\end{lemma}
\begin{lemma}\label{lem:Par_trunc}
    Let $X$ be a  Pareto random variable with law $\mathbf{P}$ and parameter $\beta>1$. For any $m>0$ it holds
    \[
    \mathbf E\left[X\one_{X\geq m}\right]=\frac{\beta}{(\beta-1)} m^{1-\beta}. 
    \]
\end{lemma}
We state one final auxiliary lemma related to the approximation of sums by integrals.
\begin{lemma}\label{lem:sumtoint}
    Let $\beta\in (0,\,1]$. Then there exists a constant $c_1=c_1(\beta)>0$ such that 
    \begin{equation}\label{eq:alpha_sum_bound}
    \frac{1}{N}\sum_{i\neq j\in \Ver_N}\frac{1}{\|i-j\|^\beta}\sim c_1 \max\{N^{1-\beta},\,\log N\}.
    \end{equation}
    If instead $\beta>1$, there exists a constant $c_2>0$ such that 
    \[
    \frac{1}{N}\sum_{i\neq j\in \Ver_N}\frac{1}{\|i-j\|^\beta}\sim c_2\, .
    \]
\end{lemma}

We end this section by quoting, for the reader's convenience, the following lemma from \citet[ Fact 4.3]{Chakrabarty:Hazra:Sarkar:2016}.
\begin{lemma}\label{lemma:slutsky} Let $(\Sigma, d)$ be a complete metric space, and let $(\Omega, \mathcal{A}, P)$ be a probability space. Suppose that $\left(X_{m n}:(m, n) \in\{1,2, \ldots, \infty\}^2 \backslash\{\infty, \infty\}\right)$ is a family of random elements in $\Sigma$, that is, measurable maps from $\Omega$ to $\Sigma$, the latter being equipped with the Borel $\sigma$-field induced by $d$. Assume that
\begin{enumerate}[label=(\arabic*),ref=(\arabic*)]
\item\label{item:1ar}for all fixed $1 \leq m<\infty$
$$
 \lim_{n\to \infty}d\left(X_{m n}, X_{m \infty}\right) =0 \text{ in $P$-probability}.$$
\item\label{item:2ar} For all $\varepsilon>0$,
$$
\lim _{m \rightarrow \infty} \limsup _{n \rightarrow \infty} P\left(d\left(X_{m n}, X_{\infty n}\right)>\varepsilon\right)=0 .
$$
\end{enumerate}

Then, there exists a random element $X_{\infty \infty}$ of $\Sigma$ such that
\begin{equation}\label{eq:450}
\lim_{m\to\infty}d\left(X_{m \infty}, X_{\infty \infty}\right) = 0 \text{ in $P$-probability}
\end{equation}
and
$$
 \lim_{n\to \infty}d\left(X_{\infty n}, X_{\infty \infty}\right) = 0\text{ in $P$-probability}.
$$
Furthermore, if $X_{m \infty}$ is deterministic for all $m$, then so is $X_{\infty \infty}$, and  \eqref{eq:450} simplifies to
\begin{equation}\label{eq:451}
\lim _{m \rightarrow \infty} d\left(X_{m \infty}, X_{\infty \infty}\right)=0 .
\end{equation}

\end{lemma}

\section{Limiting spectral distribution: proof of Theorem \ref{theorem:main}}\label{sec:proof_thm_main}

The proof of Theorem \ref{theorem:main} is split into several parts and we will now briefly sketch them.
\begin{itemize}[leftmargin=*]
    \item[(1)]  {\bf Truncation}. The first part of the proof is a truncation argument on the unbounded weights $(W_i)_{i\in \Ver_N}$. We construct a new sequence $(W_i^m)_{i\in \Ver_N} $ that is obtained by truncating the original weights at a value $m>1$. We construct another scaled adjacency matrix $\A_{N,m}$, with entries $\A_{N,m}(i,j)$ distributed as Bernoulli random variables with parameter $p_{ij}^m$ given by \eqref{connection_proba} with the weights substituted by the truncated ones. We then show (see Lemma \ref{lemma:truncation}) that the empirical measure $\ESD(\bA_N)$ is well approximated by $\ESD(\A_{N,m})$, that is, their L\'evy distance  vanishes in probability in the limit $m\to\infty$.
    \item[(2)] {\bf Gaussianisation}. In the second part, we aim to Gaussianise $\A_{N,m}$ using the ideas of \cite{chatterjee2005simple}. We begin with the construction of a centred matrix $\cA_{N,m}$, that is obtained by subtracting out the expectation from each entry of $\A_{N,m}$. We then Gaussianise $\cA_{N,m}$, that is, we pass to another matrix $\bA_{N,g}$ with each entry $\bA_{N,g}(i,j)$ being a normal random variable with mean $0$ and the same variance $p_{ij}^m(1-p_{ij}^m)$ as the corresponding entry of $\cA_{N,m}$. Lastly, we tweak the variances of $\bA_{N,g}$ to obtain a Gaussian random matrix $\tilbA_{N,m,g} $ with entries $\tilbA_{N,m,g} (i,j)$ having mean $0$ and variance equal to $r_{ij}^m$, the ``unbounded version'' of $p_{ij}^m$ (see \eqref{eq:rm}). Thanks to  \eqref{HWinequality}, we can show (Lemma \ref{lemma:centring}, Lemma \ref{lemma:gaussianisation} and Lemma \ref{lemma:gaussian_removing1}) that in this whole process we did not lose too much: the L\'evy distance between the empirical measures $\ESD(\A_{N,m})$ and $\ESD(\tilbA_{N,m,g} )$ is small in probability. We remark here that the order of the errors in Lemmas \ref{lemma:centring} and \ref{lemma:gaussian_removing1} is $N^{-\alpha}$, and these steps fail for $\alpha=0$.
    \item[(3)] {\bf Identification of the limit}. We then proceed to analyse the limit of the measure $\ESD(\tilbA_{N,m,g})$ as $N$ goes to infinity. We use Wick's formula to compute its expected moments  and use a concentration argument to show the existence of a unique limiting measure 
    \[
    \mu_{\sigma,\tau,m} :=\lim_{N\to\infty}\ESD(\tilbA_{N,m,g} )\]
    using Proposition \ref{prop:GaussianESD}. We conclude the proof of Theorem \ref{theorem:main} by letting the truncation $m$ go to infinity: using Lemma \ref{lemma:slutsky} we can show that there is a unique limiting measure $\mu_{\sigma,\tau}$ such that $\mu_{\sigma,\tau}:=\lim_{m\to\infty}\mu_{\sigma,\tau,m}$. In the case $\sigma=1$ calculations become explicit.
\end{itemize}
\begin{remark}[Self-loops]\label{rem:selfloop}
    We can use Proposition \ref{prop:hoffman-wielandt} to show that having self-loops in the model will not affect the limiting spectral distribution. 
    Let $\bA_N$ be the scaled adjacency matrix of the model as defined in \eqref{eq:scaledadjacency}. Now, consider $$D_N=c_N^{-1/2}\diag(1,\ldots,1)$$ to be the $N\times N$ diagonal matrix with all diagonal entries ``1'', scaled by a factor of $\sqrt{c_N}$, and $\bA_{N,SL}=\bA_N+D_N$. If we extend the definition of $p_{ij}$ for the case $i=j$ as $p_{ii}=1$, then $\bA_{N,SL}$ will be the scaled adjacency of the random graph with self-loops. Using \eqref{HWinequality}, we get 
    \[
    d_L^3(\mu_{\bA_N},\mu_{\bA_{N,SL}}) \leq \frac{1}{N}\Tr[(\bA_N - \bA_{N,SL})^2] = \frac{1}{N}\Tr[D_N^2] = \frac{N}{Nc_N} = \bigO(c_N^{-1}).
    \]
\end{remark}
\subsection{Truncation}\label{subsec:trunc}
Now we show that for our analysis the weights can be truncated. More precisely, let $m>1$ be a truncation threshold and define $W_i^m= W_i\one_{W_i\le m}$ for any $i\in\Ver_N$.
For all $N\in\N$, we define a new random graph with vertex set $\Ver_N$ and connection probability as follows: conditional on the weights $(W_i^m)_{i\in \Ver_N}$ we connect $i,j\in \Ver_N$ with probability 
\begin{equation}\label{eq:rm}
p_{ij}^m= r_{ij}^m\wedge 1 \qquad\text{ with }\quad
r_{ij}^m
    = \frac{(W_i^m\vee W_j^m)(W_i^m\wedge W_j^m)^\sigma}{\|i-j\|^\alpha}\qquad i\neq j\in \Ver_N\,.
\end{equation}

Let $\A_{N,m}$ be the corresponding adjacency matrix scaled by $\sqrt{c_N}$ and $\ESD(\A_{N,m})$ be its ESD.

It will be useful later to have the two following easy bounds (following from Lemma \ref{lem:sumtoint}):
\begin{align}\label{eq:r_bounds}
    \sum_{i\neq j\in \Ver_N}r_{ij}^m\leq m^{1+\sigma}Nc_N\, ,
    \qquad\qquad \sum_{i\neq j\in \Ver_N}(r_{ij}^m)^t\leq c\,m^{2+2\sigma}\max\{N^{1-t\alpha},\log N\}\,,
\end{align}
for some constant $c>0$ and $t>1$ a real number. The second bound is not optimal, since for some $t>1$ such that $t\alpha>1$, the upper bound will just be a constant depending on $t$ and $\alpha$. However, for our computations, this bound suffices. 

\begin{lemma}[{Truncation}]\label{lemma:truncation}
   For every $\delta>0$ one has 
   $$
   \limsup_{m\to \infty}\lim_{N\to\infty}\prob\left(d_L(\ESD(\A_{N}), \ESD(\A_{N,m}))> \delta\right)=0\,.
   $$
\end{lemma}
\begin{proof}
By~\eqref{HWinequality} we have that 
\begin{align}
    \E\left[d_L^3\left( \ESD(\A_N), \, \ESD(\A_{N,m})\right)\right]&\le \frac{1}{N c_N}\E\left[\Tr\left( (\A_N-\A_{N,m})^2\right)\right]\nonumber\\
    &=\frac{1}{N c_N}\sum_{i\neq j\in \Ver_N}\E\left[(\A_N(i,j)-\A_{N,m}(i,j))^2\one_{\A_N(i,j)\neq\A_{N,m}(i,j)}\right]\nonumber\\
    &\leq \frac{1}{N c_N}\sum_{i\neq j\in \Ver_N}\prob\left(\A_N(i,j)\neq\A_{N,m}(i,j)\right)\label{eq:here}.
    \end{align}
    For fixed $i,\,j$ we will analyse $\prob\left(\A_N(i,j)\neq\A_{N,m}(i,j)\right)$ as follows. We notice that $\A_N(i,j)\neq\A_{N,m}(i,j)$ can occur only if one between $W_i$ and $W_j$ exceeds $m$. Calling
    \begin{equation}\label{eq:setsAnB}
    A=\{W_i\geq m > W_j\} \;\text{ and }\;
    B=\{W_i\geq W_j\geq m\}
    \end{equation}
     we have, by symmetry of $W_i$ and $W_j$, that $\prob\left(\A_N(i,j)\neq\A_{N,m}(i,j)\right)$ equals
    $$
    2\prob\left(\{\A_N(i,j)\neq\A_{N,m}(i,j)\}\cap A\right)+ 2\prob\left(\{\A_N(i,j)\neq\A_{N,m}(i,j)\}\cap B\right)\,.
    $$
    Notice that on the events $A$ and $B$ the variable $\A_{N,m}(i,j)$ is always $0$. So we can bound
    \begin{align*}
    \prob\left(\{\A_N(i,j)\neq\A_{N,m}(i,j)\}\cap A\right)
        &=\prob\left(\{\A_N(i,j)=1\}\cap A\right)\\
        &\leq \ep\Big[\frac{\kappa_{\sigma}(W_i, W_j)}{\|i-j\|^{\alpha}}\one_{A}\Big]
        \leq \frac{\ep[W_i\one_{W_i\geq m}]\ep[W_j^\sigma]}{\|i-j\|^{\alpha}}
        \leq c\frac{m^{2-\tau}}{\|i-j\|^{\alpha}}
    \end{align*}
    for some constant $c>0$, where we have used Lemma~\ref{lem:Par_trunc} and the fact that $\E[W_j^\sigma]<\infty$.
    Analogously we can bound the second summand by
    $$
    \prob\left(\{\A_N(i,j)\neq\A_{N,m}(i,j)\}\cap B\right)
        \leq \ep\left[\frac{W_i W_j^\sigma}{\|i-j\|^{\alpha}}\one_{B}\right]
        \leq \frac{\ep[W_i \one_{W_i\geq m}]\ep[W_j^\sigma]}{\|i-j\|^{\alpha}}
        \leq c\frac{m^{2-\tau}}{\|i-j\|^{\alpha}}\,.
    $$
    Plugging these estimates back into \eqref{eq:here} we obtain
    \begin{align*}
        \E\left[d_L^3\left( \ESD(\A_N), \, \ESD(\A_{N,m})\right)\right]
        &\leq \frac{4c}{N c_N}\sum_{i\neq j\in \Ver_N}\frac{m^{2-\tau}}{\|i-j\|^{\alpha}}
        =4cm^{2-\tau}\,.
    \end{align*}
    We can then conclude by applying Markov's inequality:
    \begin{align*}
        \limsup_{m\to \infty}\lim_{N\to\infty}\prob\left(d_L\left( \ESD(\A_N), \, \ESD(\A_{N,m})\right)> \delta\right)&\leq  \limsup_{m\to \infty}\lim_{N\to\infty}\frac{\E\left[d_L^3\left( \ESD(\A_N), \, \ESD(\A_{N,m})\right)\right]}{\delta^3}\\
        &=0
    \end{align*}
    since $\tau>2$.
\end{proof}

\subsection{Centring}
Let $1< m\le \infty$ and $\cA_{N,m}$ be the centred and rescaled truncated adjacency matrix, i.e. the matrix defined as \begin{equation}\label{eq:cA_def}
  \cA_{N,m}(i,j)=  \A_{N,m}(i,j)- E^W[ \A_{N,m}(i,j)],\quad i\neq j\in \Ver_N.  
\end{equation}
Note that here $m=\infty$ corresponds to the matrix with non-truncated weights. The following lemma says that the centring does not affect the limiting spectral distribution.
\begin{lemma}[{Centring}]\label{lemma:centring}
For any $m\in (1, \infty]$, under the conditions in Theorem \ref{theorem:main}, we have, for all $\delta>0$,
$$
\lim_{N\to\infty}\prob\left(d_L\big( \ESD(\A_{N,m}), \, \ESD(\cA_{N,m})\big)>\delta\right)= 0\,,
$$
where $\ESD(\cA_{N,m})$ is the empirical spectral distribution of $\cA_{N,m}$. 
\end{lemma}

\begin{proof}
By \eqref{HWinequality} we have
    \begin{align}
    \E\left[d_L^3\big( \ESD( \A_{N,m}), \, \ESD(\cA_{N,m})\big)\right]
    &\le \frac{1}{N}\E\left[\Tr( E^W[\A_{N,m}]^2)\right]\nonumber\\
    &= \frac{1}{N c_N} \sum_{i\neq  j\in \Ver_N}  \ep[p_{ij}^m ]^2\nonumber\\
    &\leq \frac{1}{N c_N}\sum_{i\neq j\in \Ver_N} \frac{\ep\left[ (W_i\vee W_j)(W_i\wedge W_j)^{\sigma} \right]^2}{\|i-j\|^{2\alpha}} \nonumber \\
    &\le \frac{c}{Nc_N} \max\{N^{1-2\alpha},\log N\}.\label{eq:ord_d}
    \end{align}  
   Here $c$ is some constant as for $\tau>2$ and $\sigma<\tau -1$ we have 
    \[
    \ep\left[(W_i \vee W_j)(W_i\wedge W_j)^{\sigma}\right] = 2\ep\left[W_iW_j^{\sigma}\one_{W_i>W_j}\right] \leq 2\ep[W_i]\ep[W_j^{\sigma}] <\infty.
    \]  
  In the last inequality we used Lemma~\ref{lem:sumtoint}. The result follows by applying Markov's inequality.
\end{proof}

\subsection{Gaussianisation}\label{subsec:gauss}
Let $\left\{G_{i, j}, 1 \leq i \leq j\right\}$ be a family of i.~i.~d.\,standard Gaussian random variables, independent of the weights and the graph. Define a symmetric $N \times N$ matrix $\A_{N,m,g}$ by
\begin{equation}\label{gaussianized_matrix}
\A_{N,m,g}(i, j)=\begin{cases}
    \frac{\sqrt{p_{i,j}^m(1-p_{i,j}^m)}}{\sqrt{c_N}} G_{i \wedge j, i \vee j} &\text{ for $1 \leq i\neq  j \leq N$} \,\\
    0& \text{ for $i=j$}. 
\end{cases}
\end{equation}
Notice that the entries of $\A_{N,m,g}$ have the same mean and variance of the corresponding entries of $\cA_{N,m}$.
Consider a three-times continuously differentiable function $h: \mathbb{R} \rightarrow \mathbb{R}$ such that
$$
\max _{0 \leq k \leq 3} \sup _{x \in \mathbb{R}}\left|h^{(k)}(x)\right|<\infty
$$
where $h^{(k)}$ denotes the $k$-th derivative. For an $N \times N$ real symmetric matrix $\M_N$ define the \emph{resolvent} of $\M_N$ as
$$
R_{M_N}(z)=\left(\M_N-z \Id_N\right)^{-1},\qquad z \in \mathbb{C}^+,
$$
where $\Id_N$ is the $N\times N$ identity matrix. 
In particular, if $\mu\coloneqq \mu_{\M_N}$ is the $\ESD$ of $\M_N$, the relation between the Stieltjes transform $\St_{{\M_N}}$ of $\mu_{\M_N}$ and resolvent can be expressed as
\begin{equation}\label{Hdef}
H(\M_N):=\St_{{\M_N}}(z) = \tr (R_{M_N}(z)), \, z\in \C^+
\end{equation}
\cite[Section 1.3.2]{Bai-silverstein}.
The next result shows that the real and imaginary parts of the Stieltjes transform of $\mu_{\cA_{N,m}}$ are close to those of $\mu_{\A_{N,m,g}}$. Since one knows that the convergence of the ESD is equivalent to showing the convergence of the corresponding Stieltjes transform, one can shift the problem to the Gaussianised setup and work with the matrix $\A_{N,m,g}$.

\begin{lemma}[Gaussianisation]\label{lemma:gaussianisation} Consider the matrix $\cA_{N,m}$ defined in Subsection~\ref{subsec:trunc} and the matrix $\A_{N,m,g}$ defined in~\eqref{gaussianized_matrix}. For any three-times continuously differentiable function $h: \mathbb{R} \rightarrow \mathbb{R}$ such that
$$
\max _{0 \leq k \leq 3} \sup _{x \in \mathbb{R}}\left|h^{(k)}(x)\right|<\infty
$$ 
we have
$$
\begin{gathered} 
\lim _{N \rightarrow \infty} \Big|\mathbb{E}\left[h\left(\Re H\left(\A_{N,m,g}\right)\right)\right]-\mathbb{E}\left[h\left(\Re H\left(\cA_{N,m}\right)\right)\right]\Big|=0, \\
\lim _{N \rightarrow \infty} \Big|\mathbb{E}\left[h\left(\Im H\left(\A_{N,m,g}\right)\right)\right]-\mathbb{E}\left[h\left(\Im H\left(\cA_{N,m}\right)\right)\right]\Big|=0 \,,
\end{gathered}
$$
where $\Re$ and $\Im$ denote the real and imaginary parts respectively and $h^{(k)}$ denotes the $k$-th derivative of $h$.
\end{lemma}
To prove the above lemma, we will need the following result from \cite{chatterjee2005simple}. 
\begin{theorem}[{\citet[Theorem 1.1]{chatterjee2005simple}}]\label{Chatterjee1.1}
Let $\mathbf{X}=\left(X_1, \ldots, X_n\right)$ and $\mathbf{Y}=\left(Y_1, \ldots, Y_n\right)$ be two vectors of independent random variables with finite second moments, taking values in some open interval $I$ and satisfying, for each $i, \mathbb{E} X_i=\mathbb{E} Y_i$ and $\mathbb{E} X_i^2=\mathbb{E} Y_i^2$.
Let $f: I^n \rightarrow \mathbb{R}$ be three-times differentiable in each argument. If we set $U=f(\mathbf{X})$ and $V=f(\mathbf{Y})$, then for any thrice differentiable $h: \mathbb{R} \rightarrow \mathbb{R}$ and any $K>0$,
$$
\begin{aligned}
&|\mathbb{E} h(U)-\mathbb{E} h(V)| \leq C_1(h) \lambda_2(f) \sum_{i=1}^n\left[\mathbb{E}\left[X_i^2 \one_{\left|X_i\right|>K}\right]+\mathbb{E}\left[Y_i^2 \one_{\left|Y_i\right|>K}\right]\right] \\
&+C_2(h) \lambda_3(f) \sum_{i=1}^n\left[\mathbb{E}\left[\left|X_i\right|^3 \one_{\left|X_i\right| \leq K}\right]+\mathbb{E}\left[\left|Y_i\right|^3 \one_{\left|Y_i\right| \leq K}\right]\right]
\end{aligned}
$$

where $C_1(h)=\left\|h^{\prime}\right\|_{\infty}+\left\|h^{\prime \prime}\right\|_{\infty}, C_2(h)=\frac{1}{6}\left\|h^{\prime}\right\|_{\infty}+\frac{1}{2}\left\|h^{\prime \prime}\right\|_{\infty}+\frac{1}{6}\left\|h^{\prime \prime \prime}\right\|_{\infty}$ and

$$
\lambda_s(f) \coloneqq \sup \left\{\left|\partial_i^q f(x)\right|^{\frac{s}{q}}: 1 \leq i \leq n, 1 \leq q \leq s, x \in I^n\right\},
$$

where $\partial_i^q$ denotes $q$-fold differentiation with respect to the $i$-th coordinate.
\end{theorem}

\begin{proof}[Proof of Lemma \ref{lemma:gaussianisation}]
    We prove this for the real part of the Stieltjes transform. The bounds for the imaginary part remain the same. We fix a complex number $z \in \C^+$, given by $z=\Re(z) + \dot\iota\eta$ with $\eta>0$. 
    
    Let $n=N(N-1)/2$ and $\xvec = (x_{ij})_{1\leq i<j\leq N} \in \Rr^{n}$. Define $\mathrm{R}(\xvec)$ to be the matrix-valued differentiable function given by 
    \[
    \mathrm{R}(\xvec) \coloneqq (\M_N(\xvec)-z\Id_N)^{-1}, 
    \]
    where $\M_N(\cdot)$ is the matrix-valued differentiable function that maps a vector in $\Rr^{n}$ to the space of $N\times N$ Hermitian matrices, given by 
    \[
    \M_N(\xvec)_{ij} = \begin{cases} 
        c_N^{-1/2} x_{ij} &\text{if $i< j$},\\
        c_N^{-1/2} x_{ji} &\text{if $i>j$},\\
        0 & \text{if $i=j$}.
    \end{cases}
    \] 
    Since $\M_N$ is symmetric, it has all real eigenvalues. The function $H(\M_N(\xvec))$ admits partial derivatives of all orders. In particular, we denote for any $\mathbf{u} \in\{(i, j)\}_{1 \leq j < i \leq n}$ the partial derivative as $\partial H / \partial x_{\mathbf{u}}$. 
    For any $\mathbf{u} \in\{(i, j)\}_{1 \leq j < i \leq n}$, using the identity $(\M_N(\xvec)-z\Id)\mathrm R(\xvec)=\Id_N$ we have
    \[
    \frac{\partial \mathrm{R}(\xvec)}{\partial x_{\mathbf{u}}} = -\mathrm{R}(\xvec)(\partial_{\mathbf{u}}\M_N )\mathrm{R}(\xvec).
    \]
By iterative application of derivatives, three identities were derived in \cite{chatterjee2005simple}:
$$\begin{aligned} \frac{\partial H}{\partial x_{\mathbf{u}}} & =-\frac{1}{N} \operatorname{Tr}\left(\frac{\partial \M_N(\xvec)}{\partial x_{\mathbf{u}}} \mathrm R(\xvec)^2\right), \\ 
\frac{\partial^2 H}{\partial x_{\mathbf{u}}^2} & =\frac{2}{N} \operatorname{Tr}\left(\frac{\partial \M_N(\xvec)}{\partial x_{\mathbf{u}}} \mathrm R(\xvec) \frac{\partial \M_N(\xvec)}{\partial x_{\mathbf{u}}} \mathrm R(\xvec)^2\right), \\ 
\frac{\partial^3 H}{\partial x_{\mathbf{u}}^3} & =-\frac{6}{N} \operatorname{Tr}\left(\frac{\partial \M_N(\xvec)}{\partial x_{\mathbf{u}}} \mathrm R(\xvec)\frac{\partial \M_N(\xvec)}{\partial x_{\mathbf{u}}} \mathrm R(\xvec) \frac{\partial \M_N(\xvec)}{\partial x_{\mathbf{u}}} \mathrm R(\xvec)^2\right) .\end{aligned}$$

Note that $\partial_{ij}\M_N(\xvec)$ is a matrix with $c_N^{-1/2}$ at the $(i,j)^{\text{th}}$ and $(j,i)^{\text{th}}$ entry, and 0 everywhere else.  Using the bounds on Hilbert Schmidt norms and following the exact argument regarding the bounds in equations (4), (5) and (6) in \cite{chatterjee2005simple} we get that
\[
\left\| \frac{\partial H}{\partial x_{\mathbf{u}}}\right\|_{\infty}\le \frac{2}{\eta N\sqrt{c_N}},\, \left\| \frac{\partial^2 H}{\partial x_{\mathbf{u}}^2}\right\|_{\infty}\le \frac{4}{\eta^3 Nc_N}, \,\left\|\frac{\partial^3 H}{\partial x_{\mathbf{u}}^3}\right\|_{\infty}\le \frac{12}{\eta^4 N c_N^{3/2}}.
\]

Hence 
$$\lambda_2(H) \le 4\max\left\{\frac{1}{\eta^4}, \frac{1}{\eta^3}\right\} \frac{1}{Nc_N}$$
and
$$\lambda_3(H)\le 12 \max\left\{\frac{1}{\eta^6},\frac{1}{\eta^{9/2}},\frac{1}{\eta^4}\right\}\frac{1}{Nc_N^{3/2}}.$$

    Conditional on the weights $(W_i)_{i\ge 1}$, consider the following sequence of independent random variables. Let $\Xvec_b = (X_{ij}^b)_{1\le i<j\le N}$ be a vector with  $X_{ij}^{b} \sim \Ber(p_{ij}^m) - p_{ij}^m$. Similarly, take another vector $\Xvec_g = (X_{ij}^g)_{1\le i<j\le N}$ with $X_{ij}^g \sim \mathcal{N}\left(0,\, p_{ij}^m(1-p_{ij}^m)\right)$.
    Then, 
    \[
    \cA_{N,m} = \M_N(\Xvec_b) \;\text{ and } \;\bA_{N,g} = \M_N(\Xvec_g)
    \]
    in law. 
    We have that 
    \[
    \big|\mathbb{E}\left[h\left(\Re H_z\left(\A_{N,m,g}\right)\right)-h\left(\Re H_z\left(\cA_{N,m}\right)\right)\right]\big| = \big|\ep\left[ E^W\left[   h\left(\Re H_z\left(\A_{N,m,g}\right)\right)-h\left(\Re H_z\left(\cA_{N,m}\right)\right)    \right]      \right]\big|\,.
    \] 
Conditionally on the weights, the sequences $\Xvec_g$ and $\Xvec_b$ form two vectors of independent random variables, with $E^W[X_{ij}^b] = E^W[X_{ij}^g]$ and $E^W[(X_{ij}^b)^2] = E^W[(X_{ij}^g)^2]$. Then, using Theorem \ref{Chatterjee1.1} on $E^W[ h\left(\Re H_z\left(\A_{N,m,g}\right)\right)-h\left(\Re H_z\left(\cA_{N,m}\right)\right) ]$, we have that 
    \begin{align}
    &\big|\ep\left[E^W[ h\left(\Re H_z\left(\A_{N,m,g}\right)\right)-h\left(\Re H_z\left(\cA_{N,m}\right)\right) ]\right] \big|\nonumber\\
    &\leq C_1(h)\lambda_2(H)\sum_{1\le i<j\le N} \E[(X_{ij}^b)^2\one_{|X_{ij}^b|>K_N}] +  \E[(X_{ij}^g)^2\one_{|X_{ij}^g|>K_N}] \label{eq:Chatterjee<K}\\
    &+ C_2(h)\lambda_3(H)\sum_{1\le i<j\le N} \E[(X_{ij}^b)^3\one_{|X_{ij}^b|\leq K_N}] +  \E[(X_{ij}^g)^3\one_{|X_{ij}^g|\leq K_N}] \label{eq:Chatterjee>K}\,,
    \end{align}
    where $K_N$ is a (possibly) $N-$dependent truncation and where we have used that $\left|\partial_{\mathbf{u}}^p \Re H\right|=\left|\Re \partial_{\mathbf{u}}^\rho H\right| \leq\left|\partial_{\mathbf{u}}^p H\right|$.
Now using the fact that ${r}/{p}>0$ we have $\left|\partial_{\mathbf{u}}^p \Re H\right|^{\frac{r}{p}} \leq\left|\partial_{\mathbf{u}}^p H\right|^{\frac{r}{p}}$, and therefore
$$
\lambda_r(\Re H) \leq \lambda_r(H).
$$

    We begin by evaluating \eqref{eq:Chatterjee<K}. To compute the Bernoulli term, notice that $X_{ij}^b$ are uniformly bounded by 1, so, for any $K_N>1$, we automatically have that
    \begin{align*}
        \sum_{1\le i<j\le N} \E[(X_{ij}^b)^2\one_{|X_{ij}^b|>K_N}] =0\,.
    \end{align*}
    For the Gaussian term, we apply the Cauchy-Schwarz inequality (with respect to $\E$). Using also the trivial bound $p_{ij}^m\leq r_{ij}^m$ and Markov's inequality, we obtain
    \begin{align*}
        \sum_{1\le i<j\le N} &\E[(X_{ij}^g)^2\one_{|X_{ij}^g|>K_N}] \leq \sum_{1\le i<j\le N} \E[(X_{ij}^g)^4]^{1/2}\prob(|X_{ij}^g|>K_N)^{1/2} \\
        &\leq 3\sum_{1\le i<j\le N} \ep[(r_{ij}^m)^2]^{1/2}\,\,\frac{\E[(X_{ij}^g)^2]^{1/2}}{K_N}\leq 3\sum_{1\le i<j\le N} \ep[(r_{ij}^m)^2]^{1/2}\,\,\frac{\ep[r_{ij}^m]^{1/2}}{K_N} \\
        &\stackrel{\eqref{eq:r_bounds}}{=} \bigO_N(N\cdot K_N^{-1}\max\{N^{1-3\alpha/2},\,\log{N}\}).
    \end{align*}
    We thus conclude that \eqref{eq:Chatterjee<K} is of order $$\eqref{eq:Chatterjee<K} = \bigO_N(c_N^{-1}K_N^{-1}\max\{N^{1-3\alpha/2},\,\log{N}\}).$$   
    For \eqref{eq:Chatterjee>K}, we use that for any random variable $X$ we have the bound $\E[|X|^3\one_{|X|\leq K}] \leq K\E[X^2]$. Hence we can bound 
    \begin{align*}
        \sum_{1\le i<j\le N} \E[(X_{ij}^b)^3\one_{|X_{ij}^b|\leq K_N} + (X_{ij}^g)^3\one_{|X_{ij}^g|\leq K_N} ] &\leq K_N\sum_{1\le i<j\le N} \E[(X_{ij}^b)^2 + (X_{ij}^g)^2] \\
        &\leq 2K_N\sum_{1\le i<j\le N} \ep[r_{ij}^m] 
        \stackrel{\eqref{eq:r_bounds}}{=} \bigO_N(K_NNc_N)\,. 
    \end{align*}
    This yields that \eqref{eq:Chatterjee>K} is of order $\bigO_N(K_Nc_N^{-1/2}).$ Choosing  $K_N=\O{1}$ gives us that 
    \begin{equation}\label{eq:GaussianisationL1convergence}
        \left|\E\left[h\left(\Re H\left(\A_{N,m,g}\right)\right)\right]-\E\left[h\left(\Re H\left(\cA_{N,m}\right)\right)\right]\right| 
        = o_N(1)\,.
    \end{equation}
    A similar argument holds for the imaginary part $\Im(H)$ and this completes the proof. 
    \end{proof}

\subsubsection{Simplification of the variance structure.}
To conclude Gaussianisation, we would like to construct a final matrix $\tilbA_{N,m,g} $ with a simpler variance structure than that of $\A_{N,m,g}$. We let its entries be 
\begin{equation}\label{eq:Gaussianmatrix}
\tilbA_{N,m,g} (i, j)=\frac{\sqrt{r_{ij}^m}}{\sqrt{c_N}} G_{i \wedge j, i \vee j} \quad 1 \leq i, j \leq N
\end{equation}
where $r_{ij}^m$ is as in~\eqref{eq:rm} and the $\{G_{i,j}:i\ge j\}$ are the i.~i.~d. collection of Gaussian variables used in \eqref{gaussianized_matrix}. We need to prove that the ESD of this matrix gives asymptotically a good approximation of the ESD of $\A_{N,m,g}$.
\begin{lemma}[Simplification of variance]\label{lemma:gaussian_removing1}
    For any $\delta>0$ 
    $$ 
    \lim_{N\to\infty}\prob\left(d_L(\ESD(\A_{N,m,g}), \ESD(\tilbA_{N,m,g} ))> \delta\right)=0\,.
    $$
\end{lemma} 
\begin{proof}
    Construct a matrix $L_{N,g}$ with entries 
    $$
    L_{N,g}(i,j) = \begin{cases}
        \frac{\sqrt{p_{ij}^m}}{\sqrt{c_N}} G_{i \wedge j, i \vee j} &\quad  1 \leq i\neq j \leq N \\
        0 &\quad 1\le i=j \le N
    \end{cases}
    $$
    where $p_{ij}^m = r_{ij}^m\wedge 1$. By \eqref{HWinequality}, we have that 
    \begin{align*}
    \E[d_L^3(\ESD(\A_{N,m,g}), \ESD(L_{N,g}))] &\leq \frac{1}{N c_N}\sum_{i\neq j\in \Ver_N} \E\left[G_{i,j}^2 p_{i,j}^m\left(\sqrt{1-p_{ij}^m} - 1\right)^2\right]  \\ 
        &\leq \frac{1}{Nc_N}\sum_{i\neq j\in \Ver_N} \ep[p_{ij}^m|(1 - p_{ij}^m) - 1|] \\
        &\leq \frac{1}{Nc_N}\sum_{i\neq j\in\Ver_N}  \ep[(r_{ij}^m)^2]
        \stackrel{\eqref{eq:r_bounds}}{=}  o_N(1).
    \end{align*}
    For $i\neq j\in \Ver_N$ define the events $\mathcal A_{ij}=\{r_{ij}^m\leq 1\}$. 
    Construct yet another matrix $\tilde{L}_{N,g}$ as 
    \[
    \tilde{L}_{N,g}(i,j) 
        = {L}_{N,g}(i,j)\one_{\mathcal A_{ij}} + \frac{X_{ij}}{\sqrt{c_N}}\one_{\mathcal A_{ij}^c}
    \]
where, conditional on the weights, $X_{ij}\sim\mathcal N\left(0,r_{ij}^m \right)$ are mutually independent and independent of the $\{G_{i,j}\}_{i> j}$.
    It is easy to see that $\tilde{L}_{N,g} = \tilbA_{N,m,g} $ in distribution. So, comparing $L_{N,g}$ with $\tilde{L}_{N,g}$, using \eqref{HWinequality} we get 
    \begin{align*}
        \E&[d_L^3(\ESD(\tilde{L}_{N,g}), \ESD(L_{N,g}))] \leq \frac{1}{N}\sum_{i\neq j\in\Ver_N} \E[(L_{N,g}(i,j) - \tilde{L}_{N,g}(i,j))^2] \\
        &=\frac{1}{N}\sum_{i\neq j\in\Ver_N}^N \E[(L_{N,g}(i,j) - \tilde{L}_{N,g}(i,j))^2\one_{\mathcal A_{ij}^c}]\\
       &= \frac{1}{N}\sum_{i\neq j\in\Ver_N}^N \E\left[\left(\frac{\sqrt{p_{ij}^m}}{\sqrt{c_N}} G_{i \wedge j, i \vee j} - \frac{X_{ij}}{\sqrt{c_N}}\right)^2\one_{\mathcal A_{ij}^c}\right].
        \end{align*}
        Using that the $G_{i,\,j}$ are centred and independent of the weights, and the Cauchy-Schwarz inequality, we can develop the square to obtain a further upper bound of the form
        \begin{align*}
        & \frac{1}{Nc_N}\sum_{i\neq j\in\Ver_N}^N \ep[G_{i\wedge j,i\vee j}^2\one_{\mathcal A_{ij}^c}]+\ep[X_{ij}^2\one_{\mathcal A_{ij}^c}]  \\
        &\leq \frac{1}{Nc_N}\sum_{i\neq j\in\Ver_N}^N \pr(\mathcal A_{ij}^c)+ \ep[X_{ij}^4]^{1/2}\pr(\mathcal A_{ij}^c)^{1/2} \\
        &\leq \frac{1}{Nc_N}\sum_{i\neq j\in\Ver_N}^N \pr(\mathcal A_{ij}^c)+\frac{3\ep[(W_i^m\vee W_j^m)^2(W_i^m\wedge W_j^m)^{2\sigma}]^{1/2}}{\|i-j\|^{\alpha}}\,
        \pr(\mathcal A_{ij}^c)^{1/2}\\
        &= o_N(1)
   \end{align*}
    since 
    $$
    \pr(\mathcal A_{ij}^c) \leq \pr\left(W_iW_j^\sigma\geq \|i-j\|^\alpha\right) 
    \leq  \frac{c}{\|i-j\|^{\alpha\left((\tau-1)\wedge\frac{\tau-1}{\sigma}\right)}}.
    $$

    Using the triangle inequality, we get 
    \[
    \E[d_L^3(\ESD(\A_{N,m,g}), \ESD(\tilbA_{N,m,g} ))] = o_N(1)\,.
    \]
    We conclude the proof using Markov's inequality.
\end{proof}

\subsection{Moment method}
\subsubsection{Preliminary results: combinatorial setup}
We will recall here the combinatorics features of partitions we need in the paper, and refer the reader for a detailed exposition to~\citet[Chapter 9]{Nica:Speicher}.

For $k \geq 1$, denote by $\mathcal{P}(2k)$ the set of partitions of $[2k]$, and by $NC(2k)\coloneqq NC([2k])$ the set of non-crossing partitions of $\{1,2, \ldots, 2 k\}$. When we write a partition, we order its blocks in such a way that the first block always contains $1$, and the $(i+1)$th block contains the smallest element not belonging to any of the previous $i$ blocks. 

In what follows, we shall use Wick's formula. Let $(X_1, \ldots, X_n)$ be a real Gaussian vector, then
\begin{equation}\label{eq:Wick}
\E[X_{i_1}\cdots X_{i_k}]= \sum_{\pi\in \mathcal P_2(2k)} \prod_{(r,s)\in \pi} \E[ X_{i_r} X_{i_s}],
\end{equation}
where $\mathcal P_2(2k)$ denotes the pair partitions of $[2k]$.

Any partition $\pi\in \mathcal P(k)$ can be realised as a \emph{permutation} of $[k]$, that is, a bijective mapping $[k]\to [k]$. Let $S_k$ denote the set of permutations on $k$ elements. Let $\gamma=(1, 2, \ldots, k)\in S_k$ be the shift by $1$ modulo $k$. We will be interested in the composition of two permutations $\gamma$ and $\pi$, denoted by $\gamma\pi$, which will be seen below as a partition.

As an example, consider $\pi = \{\{1,2\},\{3,4\}\}$ and $\gamma=(1,2,3,4)$. To compute $\gamma\pi$, we read $\pi$ as $(1,2)(3,4)$, and compute $\gamma\pi = (1,3)(2)(4)$. We finally read $\gamma\pi$ as $\{\{1,3\},\{2\},\{4\}\}$.

  \begin{definition}[Graph associated to a partition,~{\citet[Definition 2.3]{avena}}]\label{pap1-partitiongraph}  For a fixed $k\geq 1$, let $\gamma$ denote the cyclic permutation $(1,2,\ldots,k)$. For a partition $\pi$, we define $G_{\gamma\pi}=(\Ver_{\gamma\pi}, E_{\gamma\pi})$ as a rooted, labelled directed graph associated with any partition $\pi$ of $[k]$, constructed as follows.
    \begin{itemize}
        \item Initially consider the vertex set $\Ver_{\gamma\pi}=[k]$ and perform a closed walk on $[k]$ as $1\to 2\to 3\to \cdots \to k\to 1$ and with each step of the walk, add an edge. 
        \item Evaluate $\gamma\pi$, which will be of the form $\gamma\pi = \{V_1,V_2,\ldots,V_m\}$ for some $m\geq 1$ where $\{V_i\}_{1\leq i\leq m}$ are disjoint blocks. Then, collapse vertices in $\Ver_{\gamma\pi}$ to a single vertex if they belong to the same block in $\gamma\pi$, and collapse the corresponding edges. Thus, $\Ver_{\gamma\pi} = \{V_1,\ldots,V_m\}$. 
        \item Finally root and label the graph as follows. 
        \begin{itemize}
            \item \emph{Root:} we always assume that the first element of the closed walk (in this case `1') is in $V_1$, and we fix the block $V_1$ as the root. 
            \item \emph{Label:} each vertex $V_i$ gets labelled with the elements belonging to the corresponding block in $\gamma\pi$.
        \end{itemize}
        \end{itemize} 
\end{definition}
For the partitions $\pi=\{\{1,\,2\},\,\{3,\,4\}\}$, $\gamma\pi=\{\{1,\,3\},\{2\},\{4\}\}$, Figure~\ref{fig:closedwalk} illustrates this procedure.
\begin{figure}[ht!]
    \centering
    \includegraphics[width=1\linewidth]{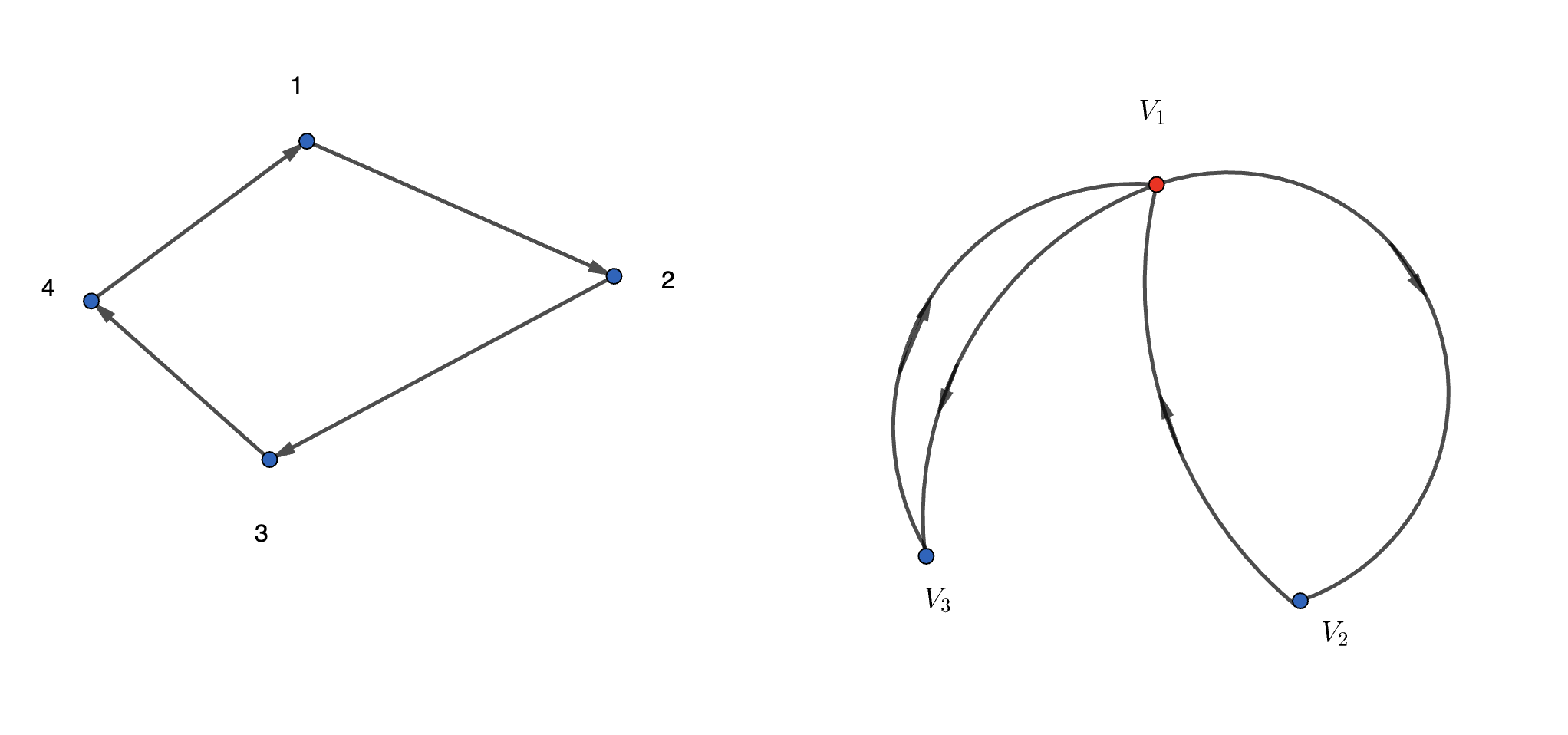}
    \caption{Left: closed walk on $[4]$. Right: graph associated to $\gamma\pi=\{\{1,3\},\{2\},\{4\}\}$. The root is in red.}
    \label{fig:closedwalk}
\end{figure}

The following lemma is an exercise in \citet[Exercise 22.15]{Nica:Speicher} and explains also why non-crossing  pair partitions will have the dominant role in the computations that follow. We will denote as $NC_2(2k)$ the set of non-crossing pair partitions of $[2k]$. For a partition $\pi$ we let $\#\pi$ the number of its blocks.
\begin{lemma} \label{lemma:gammapi}
Given $\pi\in \mathcal P_2(2k)$, one has $\#\gamma\pi\le k+1$ and the equality holds if and only $\pi\in NC_2(2k)$. If $\pi\in NC_2(2k)$, the graph $G_{\gamma\pi}$ is a rooted tree. 
\end{lemma}

Finally, given $\pi\in NC_2(2k)$, we define the map $\mathcal{T}=\Tpi:[2k]\to [k+1]$ as follows. By Lemma \ref{lemma:gammapi}, we know that $\#\gamma\pi=k+1$ and let $\gamma\pi=\{V_1, V_2, \ldots, V_{k+1}\}$. Define
\begin{equation}\label{def:Tpi}
\Tpi(i)= j \quad  \text{ if } \;\, i \in V_{j}.
\end{equation}
\subsubsection{Moment characterisation}
We are now ready to give the proofs on Gaussianisation leading to the main result of this subsection, the proof of Theorem~\ref{theorem:main}.

\begin{proposition}\label{prop:GaussianESD}
    Let $\tilbA_{N,m,g} $ be defined as in \eqref{eq:Gaussianmatrix}. Let $\ESD(\tilbA_{N,m,g})$ be its empirical spectral distribution. Then, for $k\in\N$, one has
    \begin{equation}\label{eq:truncatedmoments}
    \lim_{N\to\infty}\E\left[ \int_{\Rr} x^{2k} \ESD(\tilbA_{N,m,g} )(\De x)\right]= M_{2k}
    \end{equation}
    and odd moments are zero. Moreover, 
    \begin{equation}\label{eq:truncatedvar}
    \lim_{N\to\infty}\mathrm{Var}\left(\int_{\Rr} x^{2k} \ESD(\tilbA_{N,m,g})(\De x)\right)= 0,
    \end{equation}
    where
    \begin{equation}\label{eq:limitingmoments}
        M_{2k}= \sum_{\pi \in NC_2(2k)} \ep\left[\prod_{(u,v)\in E(G_{\gamma\pi})}\kappa_{\sigma}(W_u^m, W_v^m)\right]<\infty,
    \end{equation}
    where $\kappa_\sigma$ is as in~\eqref{eq:kappa} and $E(G_{\gamma\pi})$ is the edge set of the tree $G_{\gamma\pi}$. Moreover, there exists a unique compactly supported symmetric and deterministic measure $\mu_{\sigma,\tau, m}$ characterised by the moment sequence $\{M_{2k}\}_{k\in\N}$ such that
    \begin{equation}
       \label{eq:lim_gauss} \lim_{N\to\infty}\ESD(\tilbA_{N,m,g})=\mu_{\sigma,\tau,m}\quad \text{in }\prob\text{-probability}.\end{equation}
\end{proposition}
\begin{proof}
    Let $\{G_{i,j}: 1\le i<j\le N \}$ be a sequence of standard independent centred Gaussian random variables as in \eqref{eq:Gaussianmatrix} which is also independent of $(W_i)_{i\in [N]}$. Let $\mathcal G$ be the matrix 
    \begin{align}\label{eq:mathcalG}\mathcal G(i,j)= \begin{cases}
        \|i-j\|^{-\alpha/2}G_{i\wedge j, i\vee j}& i\neq j\\
        0 & i=j
    \end{cases}\end{align}
    
    Observe that
    $$\tilbA_{N,m,g} \overset{d}= \Upsilon_{\sigma,m}\circ \mathcal G,$$
    where $\Upsilon_{\sigma,m}$ is the matrix with elements
    $$\Upsilon_{\sigma, m}(i,j)= \sqrt{ \frac{\kappa_{\sigma}(W_i^m, W_j^m)}{c_N}}$$ 
    and $\circ$ denotes the Hadamard product. 
    Using Wick's formula \eqref{eq:Wick} we have
    \begin{align}
    \E&\left[\tr\left(\tilbA_{N,m,g} ^{2k}\right)\right]
    =\frac{1}{Nc_N^k} \mathop{\sum\nolimits^{'}}_{1\le i_1,\ldots, i_{2k}\le N} \E\left[\prod_{\ell=1}^{2k}\Upsilon_{\sigma, m}(i_\ell, i_{\ell+1 })\prod_{\ell=1}^{2k} \mathcal G(i_\ell, i_{\ell+1 })\right]\nonumber\\
    &=\frac{1}{Nc_N^k} \mathop{\sum\nolimits^{'}}_{1\le i_1,\ldots, i_{2k}\le N} \E\left[ \prod_{\ell=1}^{2k} \kappa_{\sigma}^{1/2}(W_{i_\ell}^m, W_{i_{\ell+1}}^m)\right]\sum_{\pi\in \mathcal P_2(2k)} \prod_{(r,s)\in \pi}\E\left[ \mathcal G(i_r, i_{r+1})\mathcal G(i_s, i_{s+1})\right]\nonumber\\
    &=\frac{1}{Nc_N^k} \mathop{\sum\nolimits^{'}}_{1\le i_1,\ldots, i_{2k}\le N} \E\left[ \prod_{\ell=1}^{2k} \kappa_{\sigma}^{1/2}(W_{i_\ell}^m, W_{i_{\ell+1}}^m)\right]\sum_{\pi\in \mathcal P_2(2k)} \prod_{(r,s)\in \pi}\frac{1}{\|i_r- i_{r+1}\|^{\alpha}} \one_{\{i_r,i_{r+1}\}=\{i_s,i_{s+1}\}},\label{eq:break1}
    \end{align}
    where we set $i_{2k+1}=i_1$ to ease notation, and $(r,s)\in\pi$ means $\pi(r)=s$ and  $\pi(s)=r$. Here the $\sum^\prime$ indicates the sum over all the indices $(i_1, \ldots, i_{2k})$ such that $i_\ell\neq i_{\ell+1}$ for $\ell\in [2k]$. The condition $\{i_r,i_{r+1}\}=\{i_s,i_{s+1}\}$ is satisfied in two cases:
    \begin{enumerate}[label=\textbf{Case~\arabic*}),ref= Case~\arabic*)]\label{item:matching cases}
        \item\label{item:matching case 1}$i_r= i_{s+1}$ and $i_s= i_{r+1}$, that is, $i_r= i_{\gamma \pi(r)}$ and $i_s= i_{\gamma\pi(s)}$, or 
        \item\label{item:matching case 2} $i_r= i_s$ and $i_{r+1}= i_{s+1}$, that is, $i_r= i_{\pi(r)}$ and $i_{r+1}= i_{\pi(r)+1}$. 
    \end{enumerate}
   As we are going to show, the limit of~\eqref{eq:break1} will be supported on permutations $\pi\in NC_2(2k)$ and such that \ref{item:matching case 1} is true for all $(r,s)\in \pi$. To prove this,  let  us define
    $$\mathrm{Cat}_{\pi, k}=\{ \mathbf{i}=(i_1,\,\ldots,\,i_{2k})\in[N]^{2k}:\, i_r\neq i_{r+1},\, i_r= i_{\gamma\pi(r)} \, \, \forall \, r\in [2k]\}.$$
When the condition $i_r = i_{\gamma\pi(r)}$ holds for all $r$, we see that $\mathbf{i}$ is constant on the blocks of $\gamma\pi$. We construct a graph $G(\mathbf{i})$ associated to $\mathbf{i}\in\mathrm{Cat}_{\pi,k}$ by performing a closed walk $i_1\to i_2 \to \ldots i_{2k}\to i_1$, and then collapsing elements $i_r,i_s$ into the same vertex if $r,s$ belong to the same block in $\gamma\pi$. We then collapse multiple edges. After this, we see that $G(\mathbf{i})=G_{\gamma\pi}$. Thus, when we sum over $\mathbf{i}\in\mathrm{Cat}_{\pi,k}$, the count is over $\#\gamma\pi$ many indices. 

We split the summation in \eqref{eq:break1} into two parts: a first sum over the non-crossing pairings and $\mathbf{i}\in \mathrm{Cat}_{\pi, k}$, and a second part with all the other terms, that we call $\mathcal R_1$. Since we take $\mathbf{i}\in\mathrm{Cat}_{\pi,k}$, $\mathbf i$ is constant on the blocks of $\gamma\pi$. Using this property, we obtain
\begin{align*}
  \E&\left[\tr\left(\tilbA_{N,m,g}^{2k}\right)\right]=\sum_{\pi\in NC_2(2k)} \frac{1}{Nc_N^k}\sum_{\mathbf{i}\in \mathrm{Cat}_{\pi, k}} \E\left[ \prod_{j=1}^{2k} \kappa_{\sigma}^{1/2}(W_{i_j}^m, W_{i_{j+1}}^m)\right]\prod_{(r,s)\in \pi}\frac{1}{\|i_r- i_{r+1}\|^{\alpha}}+ \mathcal R_1\\
  &=\sum_{\pi\in NC_2(2k)} \frac{1}{Nc_N^k}\sum_{\mathbf{i}\in \mathrm{Cat}_{\pi, k}} \E\left[ \prod_{(u,v)\in E(G_{\gamma\pi})} \kappa_{\sigma}(W_{u}^m, W_v^m)\right]\prod_{(r,s)\in \pi}\frac{1}{\|i_r- i_{r+1}\|^{\alpha}}+\mathcal R_1\\
\end{align*}
where in the last line we have used that 
$\mathbf{i}$ is constant on the blocks of $\gamma\pi$. Since the inner expectation no longer depends on $\mathbf{i}$, we get that
\begin{align*}
    \E\left[\tr\left(\tilbA_{N,m,g} ^{2k}\right)\right]=&\sum_{\pi\in NC_2(2k)} \E\left[ \prod_{(u,v)\in E(G_{\gamma\pi})} \kappa_{\sigma}(W_{u}^m, W_v^m)\right] \frac{1}{Nc_N^k}\sum_{\mathbf{i}\in \mathrm{Cat}_{\pi, k}} \prod_{(r,s)\in \pi}\frac{1}{\|i_r- i_{r+1}\|^{\alpha}}\\
&+\mathcal R_1.\end{align*}

Now we make the following two claims which will finish the proof. 
\begin{claim}
The following hold.
\label{claim:count}
\begin{enumerate}[leftmargin=*,label=\alph*),ref=\alph*)]
    \item\label{item:a} For any  $\pi\in NC_2(2k)$, 
\begin{equation*}\label{claim:a2}
    \lim_{N\to\infty}\frac{1}{N c_N^k}\sum_{\mathbf{i}\in \mathrm{Cat}_{\pi, k}} \prod_{(r,s)\in \pi}\frac{1}{\|i_r- i_{r+1}\|^{\alpha}}=1.
\end{equation*}
\item\label{item:b}We have that $\lim_{N\to\infty} \mathcal R_1= 0.$
\end{enumerate}
\end{claim}

With the above claim, whose proof is deferred to page~\pageref{pg:claim}, we have that \eqref{eq:truncatedmoments} holds. Moreover, the odd moments are identically 0, since there are no non-crossing pair partitions for tuples of the form $\{1,2,\ldots,2k+1\}, k\in\mathbb{N}$. We now need to now show that \eqref{eq:truncatedvar} holds. 

We introduce some new notation to prove \eqref{eq:truncatedvar}. Let $\mathbf{j}=(j_1,\ldots,j_{2k})$. Let $P(\mathbf{i})$ denote the expectation 
\[
P(\mathbf{i}) \stackrel{\eqref{eq:mathcalG}}{\coloneqq}\E\left[\prod_{\ell=1}^{2k}{\kappa_\sigma^{1/2}(W_{i_\ell}^m,W_{i_{\ell+1}}^m)}\mathcal G (i_\ell,i_{\ell+1})\right],
\]
and $P(\mathbf{i},\mathbf{j})$ be 
\[
P(\mathbf{i},\mathbf{j}) \coloneqq\E\left[\prod_{\ell=1}^{2k}{\kappa_\sigma^{1/2}(W_{i_\ell}^m,W_{i_{\ell+1}}^m)}\mathcal G (i_\ell,i_{\ell+1})\prod_{p=1}^{2k}\kappa_\sigma^{1/2}(W_{i_p}^m,W_{i_{p+1}}^m)\mathcal G (i_p,i_{p+1})\right]
\]
(with the usual cyclic convention that ${2k+1}$ equals $1$ for subscripts of indices).
We can then see that 
\begin{equation}\label{eq:variancestep1}
    \Var\left(\int_{\mathbb{R}}x^{2k}\ESD(\tilbA_{N,m,g})(\dd{x})\right) = \frac{1}{N^2c_N^{2k}}\sum_{\mathbf{i},\mathbf{j}:[2k]\to [N]} \left[P(\mathbf{i},\mathbf{j}) - P(\mathbf{i})P(\mathbf{j})\right].
\end{equation}
Note that if the terms involving $\mathbf{i}$ and $\mathbf{j}$ are completely different, that is, if the product of the terms $\mathcal G (i_1,i_2)\cdots\mathcal G (i_{2k},i_1)$ is independent of $\mathcal G (j_1,j_2)\cdots\mathcal G (j_{2k},j_1)$, then $P(\mathbf{i},\mathbf{j})=P(\mathbf{i})P(\mathbf{j})$, and \eqref{eq:variancestep1} becomes identically 0. Hence, we have 
\begin{equation}\label{eq:variancestep1.2}
\Var\left(\int_{\mathbb{R}}x^{2k}\mu_{\tilbA_{N,m,g} }(\dd{x})\right) = \frac{1}{N^2c_N^{2k}}\mathop{\sum\nolimits^{(\geq 1)}}_{\mathbf{i},\mathbf{j}:[2k]\to [N]}P(\mathbf{i},\mathbf{j}),
\end{equation}
where $\sum^{(\geq 1)}$ is over $\mathbf{i},\mathbf{j}$ such that there is \emph{at least} one matching of the form $\tilde{\A}_{N,m,g} (i_r,i_{r+1})=\tilde{\A}_{N,m,g} (j_s,j_{s+1})$ for some $1\leq r,\,s\leq 2k-1$. 
If there is only one entry of $\mathbf i$, say $i_1$, equal to only one entry of $\mathbf j$, say $j_1$, then we still have
\[E^W\left[\prod_{\ell=1}^{2k}\mathcal G (i_\ell,i_{\ell+1})\mathcal G (j_\ell,j_{\ell+1})\right]=0\]
since all entries $\mathcal G(i_\ell,i_{\ell+1})$ are independent (even if $i_1=j_1$) and centred. All the more, $P(\mathbf{i},\mathbf{j})=0$, so let us pass to having two equal indices, that is, a matching.

Let us consider the case when there is \emph{exactly one} matching. Since both indices in $\mathbf{i}$ and $\mathbf{j}$ can be reordered without affecting the variance,without loss of generality we can assume that the matching is $(i_1,i_2)=(j_1,j_2)$, and the rest of the indices of $\mathbf{i}$ are different from the ones in $\mathbf{j}$. One now has $\mathbf{i'}= (i_3,\,\ldots,\,i_{2k})$ and $\mathbf{j'}=(j_3,\,\ldots,\,j_{2k})$ with $2k-2$ indices each, and so we can construct partitions $\pi,\pi'$ for each of them independently. 

For the ease of notation, let 
\[a_{i,j} \coloneqq {\kappa_\sigma^{1/2}(W_{i}^m,\,W_{j}^m)}\mathcal G (i,\,j)\] and let $\sum^{(1)}$ be the sum over $\mathbf{i}, \mathbf{j}$ such that there is exactly one matching between $\mathbf{i}$ and $\mathbf{j}$. Using Wick's formula in the second equality, we have 
\begin{align}
    \frac{1}{N^2c_N^{2k}}&\mathop{\sum\nolimits^{(1)}}_{\mathbf{i},\mathbf{j}:[2k]\to[N]}P(\mathbf{i},\mathbf{j})=\frac{1}{N^2c_N^{2k}}\mathop{\sum\nolimits^{(1)}}_{\mathbf{i},\mathbf{j}:[2k]\to[N]}\ep\left[E^W\left[\prod_{\ell=1}^{2k}a_{i_\ell,\,i_{\ell+1}}a_{j_\ell,\,j_{\ell+1}}\right]\right]\nonumber \\
    &=\frac{1}{N^2c_N^{2k}}\sum_{\mathbf{i},\mathbf{j}:[2k]\to[N]}\ep\left[E^W[a_{i_1,i_2}^2]\sum_{\pi,\pi'\in\mathcal{P}_2(\{3,\,\ldots,\,2k\})}\prod_{(r,s)\in\pi}E^W[a_{i_r,\,i_{r+1}}a_{i_s,\,i_{s+1}}]\times\right.\nonumber\\
    &\left.\times\prod_{(r',s')\in\pi'}E^W[a_{j_{r'},\,j_{r'+1}}a_{j_{s'},\,j_{s'+1}}]\right] .\label{eq:aux_4}
\end{align}
Following the idea of the proof for \eqref{eq:truncatedmoments}, we assume Claim \ref{claim:count} to be true to obtain the optimal order. We will consider $\mathbf{i'},\mathbf{j'} \in \mathrm{Cat}_{\pi,k-1}$, and notice that 
\begin{equation}\label{eq:boundEWa}
E^W[a_{\ell,\ell'}^2]
\leq \frac{m^{1+\sigma}}{\|\ell-\ell'\|^\alpha}.\end{equation}
Interchanging summands, we obtain 
\begin{align}
\eqref{eq:aux_4}
&= \frac{1}{N^2c_N^{2k}}\ep\left[\sum_{\pi,\pi'\in\mathcal{P}_2(\{3,\,\ldots,\,2k\})}\sum_{\substack{\mathbf{i'},\mathbf{j'}\in\mathrm{Cat}_{\pi,k-1},\\i_1\neq i_2\in[N]}}E^W\left[a_{i_1,i_2}^2\right]\prod_{(r,s)\in\pi}E^W\left[a_{i_ri_{\gamma\pi(r)}}^2\right]\right.\times\nonumber\\
&\times\left.\prod_{(r',s')\in\pi'}E^W\left[a_{j_{r'}j_{\gamma\pi(r')}}^2\right]\right] + \mathcal{R}_1'\nonumber \\
&\stackrel{\eqref{eq:boundEWa}}{\leq} \frac{1}{N^2c_N^{2k}}\sum_{\pi,\pi'\in\mathcal{P}_2(\{3,\,\ldots,\,2k\})}\sum_{\substack{\mathbf{i'},\mathbf{j'}\in\mathrm{Cat}_{\pi,k-1},\\i_1\neq i_2\in[N]}} \frac{m^{1+\sigma}}{\|i_1-i_2\|^{\alpha}}\prod_{(r,s)\in\pi}\frac{m^{1+\sigma}}{\|i_r-i_{\gamma\pi(r)}\|^{\alpha}}\times\nonumber\\
&\times\prod_{(r',s')\in\pi'}\frac{m^{1+\sigma
}}{\|j_{r'}-j_{\gamma\pi(r')}\|^{\alpha}} + \mathcal{R}_1'\label{eq:variancestep2},
\end{align}
where $\mathcal{R}_1'$ is an error term such that $\lim_{N\to\infty}\mathcal{R}_1'=0$, which follows from Claim \ref{claim:count}. The contributing terms of the right-hand side of~\eqref{eq:variancestep2} can be upper-bounded by 
\begin{align*}
   \frac{1}{N^2c_N^{2k}}&\sum_{\pi,\pi'\in\mathcal{P}_2(\{3,\,\ldots,\,2k\})}\sum_{\substack{\mathbf{i}:
\mathbf{i'}\in\mathrm{Cat}_{\pi,k-1},\\ i_1\neq i_2}}\frac{m^{1+\sigma}}{\|i_1-i_2\|^{\alpha}}\prod_{(r,s)\in\pi}\frac{m^{1+\sigma}}{\|i_r-i_{\gamma\pi(r)}\|^{\alpha}}\nonumber\\
&\times\sum_{\mathbf{j'}\in\mathrm{Cat}_{\pi',k-1}}\prod_{(r',s')\in\pi'}\frac{m^{1+\sigma
}}{\|j_{r'}-j_{\gamma\pi(r')}\|^{\alpha}}\\
&= \frac{1}{N^2c_N^{2k}}\sum_{\pi,\pi'\in\mathcal{P}_2(\{3,\,\ldots,\,2k\})}\sum_{\substack{\mathbf{i}:
\mathbf{i'}\in\mathrm{Cat}_{\pi,k-1},\\ i_1\neq i_2}}\frac{m^{1+\sigma}}{\|i_1-i_2\|^{\alpha}}\prod_{(r,s)\in\pi}\frac{m^{1+\sigma}}{\|i_r-i_{\gamma\pi(r)}\|^{\alpha}}O_N(N c_k^{k-1}).
\end{align*}
Analogously, the sum over $\mathbf{i}$ conditioned on $\mathbf{i'}\in \mathrm{Cat}_{\pi,k-1}$ will be at most of order $N c_N^k$. Since the sum over partitions is finite and independent of $N$, we obtain 
\begin{equation*}
    \frac{1}{N^2c_N^{2k}}\mathop{\sum\nolimits^{(1)}}_{\mathbf{i},\mathbf{j}:[2k]\to[N]}P(\mathbf{i},\mathbf{j}) = \bigO_N(c_N^{-1}).
\end{equation*}
More generally, if one has $t$ pairings of the form $(i_1,i_2)=(j_1,j_2),\ldots,(i_{t-1},i_{t})=(j_{t-1},j_{t})$, one can use the same argument and instead obtain a faster error of the order of $c_N^{-t+1}$, simply due to the set $(j_{t+1},j_2,\ldots,j_{2k})$ now having only $2k-t$ independent indices from $\mathbf{i}$. Thus, we conclude \begin{equation}\label{eq:varianceorder}
    \Var\left(\int_{\mathbb{R}}x^{2k}\mu_{\tilbA_{N,m,g} }(\dd{x})\right) = \bigO_N(c_N^{-1}).
\end{equation}This proves \eqref{eq:truncatedvar}.

To conclude, one can see that 
\begin{equation}\label{eq:bound_mom_Cat}M_{2k}\leq( m^{1+\sigma})^kC_k,
\end{equation}
where $C_k$ is the $k^{\text{th}}$ Catalan number. 
Since $\sum_{k\geq1}C_k^{-1/2k}=\infty$, so Carleman's condition implies that $\{M_{2k}\}_{k\geq1}$ uniquely determine the limiting measure. 
Therefore we can find $C,\,R>0$ such that for all $k\geq 1$ we have $M_{2k}\leq C R^{2k}$. In turn, it is a straightforward exercise to show that this implies that $\mu_{\tau,\,\sigma,\,m}$ is compactly supported, and since it has odd moments equal to zero it is symmetric. To conclude the proof of Proposition \ref{prop:GaussianESD} we use for example~\citet[pg. 134]{tao2012topics}.

\end{proof} 

\begin{proof}[Proof of Claim~\ref{claim:count}]\label{pg:claim}
We first show~\ref{item:a}. Fix $\pi\in NC_2(2k)$.
Recall that $\mathbf{i}\in\mathrm{Cat}_{\pi,k}$ is constant on the blocks of $\gamma\pi$. Therefore 
the number of free indices over which we can construct $\mathbf i$ is $\#\gamma\pi= k+1$ (Lemma~\ref{lemma:gammapi}). 

For any $\pi\in NC_2(2k)$, there exists \emph{at least} one block of the form $(r,r+1)\in\pi$, where $1\leq r\leq 2k$, and $2k+1$ is identified with ``1''. Then, $\{r+1\}\in \gamma\pi$ is a singleton, and consequently, $i_{r+1}$ is a free index under $\gamma\pi$, that is, under the summation over indices $i_1,\ldots,i_{2k}$, $i_{r+1}$ runs from 1 to $N$ independent of other indices. Moreover, as $\mathbf{i}\in\mathrm{Cat}_{\pi,k}$, we have $i_r=i_{r+2}$. If we remove the block $(r,r+1)$ from $\pi$, we obtain $\pi'\in NC_2(2k-2)$ as a new partition on $\{1,2,\ldots,r-1,r+2,\ldots 2k\}$. Let $\mathbf{i}'$ be the tuple $(i_1,i_2,\ldots,i_{r-1},i_{r+2},\ldots,i_{2k})$. We then have $\mathbf{i}'\in \mathrm{Cat}_{\pi',k-1}$. So, we can write 
\begin{align}\label{eq:induction step}
\frac{1}{N c_N^k}\sum_{\mathbf{i}\in \mathrm{Cat}_{\pi, k}} \prod_{(r,s)\in \pi}\frac{1}{\|i_{r}- i_{s}\|^{\alpha}}= \frac{1}{N c_N^k}\sum_{\mathbf{i}'\in \mathrm{Cat}_{\pi', k-1}}\left(\prod_{(r,s)\in \pi'}\frac{1}{\|i_{r}- i_{s}\|^{\alpha}}\right)\left(\sum_{i_{r+1}=1}^N \frac{1}{\|i_{r+1}-i_{r+2}\|^{\alpha}}\right).
\end{align}
We now proceed inductively. For $k=1$ the result is given by~\eqref{eq:def_c}. Assume now that we have shown, for some $k-1\geq 0$ and any $\pi'\in NC_2(2(k-1))$, that
\begin{equation}\label{eq:induction hypothesis}
    \lim_{N\to\infty}\frac{1}{Nc_N^{k-1}} \sum_{\mathbf{i}'\in \mathrm{Cat}_{\pi',k-1}} \prod_{(r,s)\in\pi'}\frac{1}{\|i_r-i_s\|^{\alpha}}=1.
\end{equation}
We need to show the same statement holds for $k$, which is precisely Claim \ref{claim:count}\ref{item:a}. Now, we have that 
\begin{align}
    \eqref{eq:induction step} = \frac{1}{Nc_N^{k-1}}\sum_{\mathbf{i}'\in \mathrm{Cat}_{\pi',k-1}} \left(\prod_{(r,s)\in\pi'}\frac{1}{\|i_r-i_s\|^{\alpha}} \right)\left( \frac{1}{c_N}\sum_{i_{r+1}=1}^N \frac{1}{\|i_{r+2}-i_{r+1}\|^{\alpha}}\right). 
\end{align}
Taking the limit $N\to\infty$, we have that the second factor in brackets above by \eqref{eq:def_c}, and then the remaining expression equals 1 by the induction hypothesis \eqref{eq:induction hypothesis}. This proves ~\ref{item:a}.

To show~\ref{item:b}, we now analyse $\mathcal R_1$ explicitly. We have to deal with two cases: 

\begin{enumerate}[label=\textbf{b.\arabic*}),ref=b.\arabic*)]
\item\label{item:b1} $\pi\in \mathcal P_2(2k)$ and $\mathbf{i}\notin \mathrm{Cat}_{\pi, k}$.
\item\label{item:b2} $\pi\in \mathcal P_2(2k)\setminus NC_2(k)$ and $\mathbf{i}\in \mathrm{Cat}_{\pi,k}.$
\end{enumerate}

Note that for both cases the following factor involving the weights will not play any role:

$$\E\left[ \prod_{j=1}^{2k} \kappa_{\sigma}^{1/2}(W_{i_j}^m, W_{i_{j+1}}^m)\right] \le m^{k(1+\sigma)}.$$

We first deal with Case~\ref{item:b2}. From Lemma \ref{lemma:gammapi} we have $\#\gamma\pi\le k$ and hence
\begin{align}
\sum_{\pi\in \mathcal P_2(2k)\setminus NC_2(2k)}\frac{1}{Nc_N^k}&\sum_{\mathbf{i}\in \mathrm{Cat}_{\pi, k}} \E\left[ \prod_{j=1}^{2k} \kappa_{\sigma}^{1/2}(W_{i_j}^m, W_{i_{j+1}}^m)\right]\prod_{(r,s)\in \pi}\frac{1}{\|i_r- i_{r+1}\|^{\alpha}}\nonumber\\
&\le m^{k(1+\sigma)}\sum_{\pi\in \mathcal P_2(2k)\setminus NC_2(2k)}\frac{1}{Nc_N^k}\sum_{i_1\in[N]}\sum_{i_2, \ldots, i_{k}\in[N]}\frac{1}{\|i_2\|^\alpha\ldots \|i_k\|^\alpha}\label{eq:i not in Cat_k},
\end{align}
where $\eqref{eq:i not in Cat_k}$ follows from $\mathbf{i}$ being constant on the cycles of $\gamma\pi$. 
Thus, we get that the terms involved in Case \ref{item:b2} give a contribution of the order 
\begin{align}
  \eqref{eq:i not in Cat_k}  \leq cm^{k(1+\sigma)} \sum_{\pi\in \mathcal P_2(2k)\setminus NC_2(2k)} \frac{1}{N^{1+k(1-\alpha)}} N^{1+(k-1)(1-\alpha)}=\O{\frac{1}{N^{1-\alpha}}}=o_N(1)\label{eq:contr_b2}\,.
\end{align}

We now show that the contribution from~\ref{item:b1} is also negligible. Begin by fixing a partition $\pi$. For any tuple $\mathbf{i}$, we construct a corresponding graph $G(\mathbf{i})$  (recall that when $\mathbf{i}\in \mathrm{Cat}_{\pi, k}$ we ended up with $G(\mathbf{i})=G_{\gamma\pi}$).
For $\mathbf{i}\not\in\mathrm{Cat}_{\pi,k}$, $G(\mathbf{i})$ is constructed by a closed walk $i_1\to i_2\to \ldots i_{2k}\to i_1$, thereby adding the edges $(i_p, i_{p+1})_{p=1}^{2k}$ with $i_{2k+1}=i_1$. We then collapse indices $i_r,i_s$ into the same vertex when $\{i_r,i_{r+1}\}=\{i_s,i_{s+1}\}$, which can be justified by~\eqref{eq:break1}.
We then proceed by collapsing the multiple edges and looking at the skeleton graph $G(\mathbf{i})$, with vertex set $V(\mathbf{i})$. Hence, we see that 
\begin{align}
&\sum_{\pi\in \mathcal P_2(2k)}\frac{1}{Nc_N^k}\mathop{\sum\nolimits^{'}}_{\mathbf{i}:[2k]\to [N]} \E\left[ \prod_{j=1}^{2k} \kappa_{\sigma}^{1/2}(W_{i_j}^m, W_{i_{j+1}}^m)\right]\prod_{(r,s)\in \pi}\frac{1}{\|i_r- i_{r+1}\|^{\alpha}}\nonumber\\
&\le m^{k(1+\sigma)} \sum_{\pi\in \mathcal P_2(2k)}\frac{1}{Nc_N^k} N^{1+(\#V(\mathbf{i})-1)(1-\alpha)}\nonumber\\
&\le \O{ N^{(\#V(\mathbf{i})-k-1)(1-\alpha)}}.\label{eq:contri}
\end{align}   
since $m>1$ is fixed and the sum over the set $\mathcal{P}_2(2k)$ is finite. We see that the only non-trivial contribution comes when $\#V(\mathbf{i})= k+1$, which signifies that $G(\mathbf{i})$ is a tree. Now we claim that for any $\pi\in\mathcal{P}_2(2k)$ and $\mathbf{i}\notin \mathrm{Cat}_{\pi,k}$ we have $\#V(\mathbf{i})<k+1$. 

When $\mathbf{i}\notin \mathrm{Cat}_{\pi, k}$, it implies that there exists at least one $(r,s)\in \pi$, such that $i_{r}= i_s$ and $i_{r+1}= i_{s+1}$. Let us begin by assuming that there exists \emph{exactly one} such pair. Observe that due to the restrictions in $\sum^\prime$, no pair-wise indices are same, hence $s$ can neither be $r+1$, nor $r-1$. Now consider the reduced partition $\pi^{\prime}= \pi\setminus(r,s)$. Observe that $\pi^\prime\in \mathcal{P}_2(2k)(\{1,\ldots, r-1, r+1,\ldots, s-1, s+1,\ldots, 2k\})$. Note that now $\mathbf{i}^\prime\in \mathrm{Cat}_{\pi^\prime,\,k-1}$, so its contribution to~\eqref{eq:variancestep2} is of the order of 
$N^{1+(k-1)(1-\alpha)}$, which comes from the tree $G(\mathbf{i}')$ on $k$ vertices, and where $\mathbf{i}^\prime$ are the $(2k-2)$ indices which are obtained by removal of $(i_r, i_{r+1})$. So, all we are left to show is that due to \ref{item:matching case 2}, $i_r$ and $i_s$ will not give rise to a new vertex in $G(\mathbf{i})$. 

Now, there exists an $r< e< s-1$ such that $(e, s-1)\in \pi$. Due to \ref{item:matching case 2}, we have that  $i_r=i_s$ contribute to the same vertex in $G(\mathbf{i})$. Also $i_e= i_s$ and $i_{e+1}= i_{s-1}$ due to \ref{item:matching case 1}. This implies that $i_r= i_s= i_e$, where $i_e$ is already a contributing index in $G(\mathbf{i}^\prime)$. This implies that $G(\mathbf{i})$ is a tree on at most $k$ vertices, and hence $\#V(\mathbf{i})\le k$. This shows that the contribution in \eqref{eq:contri} goes to 0. 

The case for which there is more than one pair breaking the constraint in $\mathrm{Cat}_{\pi,\,k}$ leads to an even smaller order. When none of the pairs satisfy the constraint then $i_r= i_{\pi(r)}$ for all $r$ and hence $i$ is constant on the blocks of $\pi$. So $\#V(\mathbf{i})\le k$ and again the contribution in \eqref{eq:contri} goes to 0, thus proving the claim. 
\end{proof}
We wish to highlight that Proposition~\ref{prop:GaussianESD} is in fact more general, and works beyond the kernels $\kappa_\sigma$ defined in~\eqref{eq:kappa}.
\begin{remark}\label{cor:general_kappa}
 The statement of Proposition~\ref{prop:GaussianESD} holds when we replace the entries of $\tilbA_{N,m,g} $ in~\eqref{eq:Gaussianmatrix} by
 \begin{equation*}
\sqrt{\frac{{\kappa(W_i,\,W_j)}}{{c_N}\, \|i-j\|^{\alpha}} }G_{i \wedge j, i \vee j} \quad 1 \leq i, j \leq N
\end{equation*}
for any function $\kappa:[1,\,\infty)^2\to[0,\infty)$ which is symmetric and such that, for all $k\in\N$, 
 \begin{equation}\label{eq:new_cond_kappa}
\E\left[ \prod_{j=1}^{2k} \sqrt{\kappa(X_j, X_{j+1})}\right]<\infty
\end{equation}
where $X_1,\,\ldots,\,X_{2k}$ are i.i.d. random variables in $[1,\infty)$. 
\end{remark}
In our case the kernels $\kappa(x,\,y)\coloneqq\kappa_\sigma(x,\,y)\one_{x,y\leq m}$ satisfy~\eqref{eq:new_cond_kappa}.

\begin{proof}[Proof of Theorem~\ref{theorem:main}] To prove the final result, we shall use Lemma \ref{lemma:slutsky} with the complete metric space $\Sigma= \mathcal P(\Rr)$ and metric $d_L$. Recall also the definition of $\tilbA_{N,m,g}$ resp. $\cA_{N,m}$ of~\eqref{eq:Gaussianmatrix} resp.~\eqref{eq:cA_def}. In Proposition \ref{prop:GaussianESD} we have shown that there exists a (deterministic) measure $\mu_{\sigma,\tau, m}$ such that, for every $m>0$,
$$\lim_{N\to\infty} \ESD(\tilbA_{N,m,g})= \mu_{\sigma,\tau,m} \text{ in $\prob$--probability}.$$
Hence for any $h$ satisfying the assumptions of Lemma \ref{lemma:gaussianisation} and $H$ as in~\eqref{Hdef} it follows that
\[
\lim_{N \rightarrow \infty} \mathbb{E}\left[h\left(\Re H\left(\tilbA_{N,m,g}\right)\right)\right]= h\left(\Re \St_{\mu_{\sigma,\tau,m}}(z)\right).
\]
and thus, by means of Lemma~\ref{lemma:gaussianisation} and Lemma~\ref{lemma:gaussian_removing1},
\[
\lim_{N \rightarrow \infty} \mathbb{E}\left[h\left(\Re H\left(\cA_{N,m}\right)\right)\right]= h\left(\Re \St_{\mu_{\sigma,\tau,m}}(z)\right).
\]
Since the above holds true for any $h$ satisfying the assumptions of Lemma \ref{lemma:gaussianisation} and $\mu_{\sigma,\tau,m}$ is deterministic, it follows that
\[
\lim_{N \rightarrow \infty} \Re H\left(\cA_{N,m}\right)=\Re \St_{\mu_{\sigma,\tau,m}}(z) \text{ in $\prob$--probability}.
\]
A similar argument for the imaginary part shows that
\[
\lim_{N \rightarrow \infty} \Im H\left(\cA_{N,m}\right)= \Im  \St_{\mu_{\sigma,\tau,m}}(z) \text{ in $\prob$--probability}.
\]
Combining the real and imaginary parts, we have, for any $z\in \C^+$,
\[
\lim_{N \rightarrow \infty} \St_{\ESD(\cA_{N,m})}(z) = \St_{\mu_{\sigma,\tau, m}}(z) \;\text{ in $\prob$--probability}.
\]
Since the convergence of the Stieltjes transform characterises weak convergence, we have
$$\lim_{N\to \infty} \ESD(\cA_{N,m})= \mu_{\sigma,\tau, m} \;\text{ in $\prob$--probability}.$$
From Lemma~\ref{lemma:gaussian_removing1} and Lemma~\ref{lemma:centring}, it also follows that, for every $\delta>0$ and $m>0$,
$$\limsup_{N\to \infty} \prob( d_L(\mu_{\A_{N,m}}, \mu_{\sigma,\tau, m})>\delta)=0.$$
This shows condition~\ref{item:1ar} of Lemma \ref{lemma:slutsky}. Condition~\ref{item:2ar} follows from Lemma \ref{lemma:truncation} where we have proved that $$\limsup_{m\to \infty}\lim_{N\to\infty}\prob\left(d_L(\mu_{\A_{N,m}}, \mu_{\A_N})> \delta\right)=0.$$
Thus, it follows from Lemma \ref{lemma:slutsky} that there exists a deterministic measure $\mu_{\sigma,\tau}$ such that 
\begin{equation}\label{eq:important_lim}
  \lim_{m\to \infty}d_L( \mu_{\sigma,\tau, m}, \mu_{\sigma,\tau})= 0,  
\end{equation}
and hence using the triangle inequality the result follows. 

\end{proof}

\section{The case $\sigma=1$: proof of Theorem \ref{theorem:tail}}\label{sec:tail}

\begin{proof}[Proof of Theorem \ref{theorem:tail}]

{\bf Step 1: identification.} We are now dealing with the special case of $\sigma=1$. We go back to the moments of $\mu_{\sigma,\tau, m}$. Let $\gamma\pi= (V_1,\ldots, V_{k+1})$ and let $\ell_i= \#V_i$ (with a slight abuse of notation, we are viewing here $V_i$ as a set rather than a cycle). Since $\sigma=1$, $\kappa_{\sigma}(W_u^m, W_v^m)= W_u^m W_v^m$. It follows that
\begin{align*}
M_{2k}&=\sum_{\pi\in NC_2(2k)} \E\left[ \prod_{(u,v)\in E(G_{\gamma\pi})}W_u^m W_v^m\right]\\
&= \sum_{\pi\in NC_2(2k)}\prod_{i=1}^{k+1}\E[ (W^m_1)^{\ell_i}] \\
&= \int_{\Rr}x^{2k}\mu_{sc}\boxtimes\mu_{W,m}(\De x).
\end{align*}
 
The last equality follows from the combinatorial expression of the moments of the free multiplicative convolution of the semicircle element with an element whose law is given by $\mu_{W,m}$ (see \citet[Theorem 14.4]{Nica:Speicher}). 
Consider the map $x\mapsto x^2$ from $\Rr\to [0,\infty)$ and let $\mu^2$ be the push-forward of a probability measure $\mu$ under this mapping, so that $\mu_{sc}$ is pushed forward to $\mu_{sc}^2$. Then by \citet[Corollary 6.7]{bercovici1993free} it follows that
$$\lim_{m\to\infty}\mu_{W,m}\boxtimes \mu_{sc}^2\boxtimes \mu_{W,m}= \mu_{W}\boxtimes \mu_{sc}^2\boxtimes \mu_{W} .$$

A consequence of \citet[Lemma 8]{arizmendi2009} is that $$\mu_{W,m}\boxtimes \mu_{sc}^2\boxtimes \mu_{W,m}= (\mu_{sc}\boxtimes \mu_{W,m})^2$$
and 
\begin{equation}\label{eq:square_prod}\mu_{W}\boxtimes \mu_{sc}^2\boxtimes \mu_{W}=(\mu_{sc}\boxtimes \mu_{W})^2.
\end{equation}
Thus
$$\lim_{m\to\infty}(\mu_{sc}\boxtimes \mu_{W,m})^2=(\mu_{sc}\boxtimes \mu_{W})^2 .$$
Observe that $\mu_{sc}\boxtimes \mu_{W,m}$ and $\mu_{sc}\boxtimes \mu_{W}$ are symmetric around the origin~\citep[Theorem 7]{arizmendi2009}, hence we have that 
$$\lim_{m\to\infty}d_L(\mu_{sc}\boxtimes \mu_{W}, \mu_{sc}\boxtimes \mu_{W,m})=\lim_{m\to\infty}d_L(\mu_{\sigma,\tau, m}, \mu_{sc}\boxtimes \mu_{W})=0.$$
Theorem \ref{theorem:main} then implies that the $\ESD(\bA_N)$ converges to $\mu_{sc}\boxtimes \mu_{W}$ weakly in probability. 

{\bf Step 2: tail asymptotics.}  In the following we use the recent results of \citet[Lemma 7.2]{Bartosz:Kamil} from which we also borrow the notation. The free probability analogue of the classical Breiman's lemma is as follows: let $\mu, \,\nu$ be probability measures and 
\begin{equation}\label{eq:first_req}
\mu(x,\infty)\sim x^{-\beta} L(x)\end{equation}
with $L(\cdot)$ a slowly varying function~\cite[Definition 1.1]{Bartosz:Kamil}. Assume furthermore that the $\lfloor\beta+1\rfloor$-th moment of $\nu$ exists:
$$m_{\lfloor\beta+1\rfloor}(\nu)<\infty.$$ 
Then 
$$\mu\boxtimes \nu(x,\infty)\sim m_1^{\beta} (\nu)\mu(x, \infty)$$
with $m_1 (\nu)$ the first moment of $\nu$. 

Since $\mu_W\boxtimes \mu_{sc}$ is a symmetric measure we have, using \citet[equation (7.3)]{Bartosz:Kamil} and~\eqref{eq:square_prod},
 \begin{equation}\label{eq:square_meas}     
 \mu_W\boxtimes \mu_{sc}(x,\infty)=\frac{1}{2} (\mu_W\boxtimes \mu_{sc})^2(x^2,\infty)
= \frac{1}{2} \mu_W\boxtimes \mu_{sc}^2\boxtimes \mu_W(x^2,\infty).\end{equation}
By the commutativity and associativity of the free multiplicative convolution~\cite[Remark 14.2]{Nica:Speicher} we have $\mu_W\boxtimes \mu_{sc}^2\boxtimes \mu_W= \mu_{sc}^2\boxtimes \mu_W\boxtimes \mu_W$. 
Let $\nu_W\coloneqq \mu_W\boxtimes \mu_W$. Then a consequence of \citet[Theorem 1.3(iv)]{Bartosz:Kamil} is that
\begin{equation}\label{eq:tail_nuW}
    \nu_W(x,\infty)\sim \left(m_1(\mu_{W})\right)^{\tau-1} \mu_{W}(x,\infty).\end{equation}
Therefore $\nu_W$ satisfies~\eqref{eq:first_req} with $\beta\coloneqq \tau-1$, and clearly $m_{\lfloor \tau \rfloor}(\mu^2_{sc})<\infty$. Thus, applying~\citet[Lemma 7.2]{Bartosz:Kamil},
\begin{align*}
\left(\mu_{sc}\boxtimes\nu_W\right)(x,\infty)&\stackrel{\eqref{eq:square_meas}}{=}\frac12\mu_W\boxtimes \mu_{sc}^2\boxtimes \mu_W(x^2,\infty)\\
&\sim\frac12 \left(m_1(\mu_{sc}^2)\right)^{\tau-1} \nu_W(x^2,\infty)\\
&\stackrel{\eqref{eq:tail_nuW}}{\sim} \frac12\left(m_1(\mu_{sc}^2)\right)^{\tau-1} \left(m_1(\mu_{W})\right)^{\tau-1} \mu_{W}(x^2,\infty)\\
&\sim \frac12\left(m_1(\mu_{sc}^2)\right)^{\tau-1} \left(m_1(\mu_{W})\right)^{\tau-1} x^{-2(\tau-1)}.
\end{align*}
We can conclude noting that $m_1(\mu_{W})$ is finite since $\tau>2$ and $m_1(\mu_{sc}^2)= m_2(\mu_{sc})=1$~\cite[Proposition 5 a)]{arizmendi2009}.
\end{proof}

\section{Non-degeneracy of the limit: proof of Theorem \ref{theorem:secondmoment}}\label{sec:non-deg}

	The proof of Theorem \ref{theorem:secondmoment} follows the arguments in \citet[Theorem 2.2]{Chakrabarty:Hazra:Sarkar:2016}. 
	A key observation is that the limiting measure $\mu_{\sigma,\tau}$ does not depend on the parameter $\alpha$. This will allow us to deal with an easier model, formally corresponding to the case $\alpha=0$, that does not feel the influence of the torus' geometry. The lack of geometry also allows us to work on a unique probability space. More precisely, let $(G_{i,j})_{i,j \ge 1}$ be an i.i.d.~sequence of $ \mathcal{N}(0,1)$ random variables, and let $(W_i)_{i \ge 1}$ be an i.i.d.~sequence of Pareto-distributed random variables with parameter $\tau - 1$. Assume they are defined on the same probability space $(\Omega, \mathcal{F}, \pr)$. Define the $N \times N$ matrix
	$$
	B_{N,m} =N^{-1/2}\sqrt{\kappa_\sigma(W_i^m, W_j^m)} \, G_{i\wedge j,i\vee j}\,.
	$$
	Let $B_{N,\infty}$ denote the matrix with non-truncated weights. The following result can be proven exactly as in Proposition \ref{prop:GaussianESD}.
	\begin{proposition}\label{prop:withoutalpha}
		Let $\ESD(B_{N,m})$ be the empirical spectral distribution of $B_{N,m}$. Then for all $m \ge 1$,
		$$\lim_{N \to \infty} \ESD(B_{N,m}) = \mu_{\sigma,\tau, m} \quad\text{ in $\pr$-probability}.$$
		Moreover,
		$$\lim_{N \to \infty} \ESD(B_{N,\infty}) = \mu_{\sigma,\tau} \quad \;\;\;\text{ in $\pr$-probability}.$$
\end{proposition}

We use this result to prove Theorem \ref{theorem:secondmoment}. Recall that, for a distribution function $F$, the generalised inverse is given by
$$
F^{\leftarrow}(y) := \inf \{x \in \mathbb{R} : F(x) \geq y\}, \quad 0 < y < 1.
$$
\begin{proof}[Proof of Theorem \ref{theorem:secondmoment}]
From Proposition \ref{prop:withoutalpha}, it follows that there exists a subsequence $(N_k)_{k \ge 1}$ such that $\mu_{N_k,m}$ converges weakly almost surely to $\mu_{\sigma,\tau,m}$; that is,
\begin{equation}\label{eq:almostsure}
\lim_{k\to\infty}d_L(\ESD(B_{N_k,m}), \mu_{\sigma,\tau,m})=0 \quad \text{$\pr$-almost surely}.
\end{equation}
For a $n\times n$ matrix $A$, let us denote by $\lambda_1(A) \le \lambda_2(A) \le \dots \le \lambda_n(A)$ its eigenvalues. For fixed integers $1 \le k< \infty$, $1<m<\infty$, define the following random variables on the probability space $(\Omega \times (0,1), \mathcal{F} \otimes \mathcal{B}(0,1), \ps = \pr \times \mathrm{Leb})$:
\[
Z_{k,m}(\omega, x) =  \lambda_{\left\lceil N_k x \right\rceil} \big(B_{N_k,m}(\omega)\big), \quad \omega \in \Omega,\, x \in (0,1),
\]
and
$$
Z_{k, \infty}(\omega, x) :=  \lambda_{\left\lceil N_k x \right\rceil} \big(B_{N_k,\infty}(\omega)\big), \quad \omega \in \Omega, \,x \in (0,1).
$$
Let $F_{ m}$ be the distribution function of $\mu_{\sigma,\,\tau,\,m}$ (we suppress the dependence on $\sigma$ and $\tau$ in $F_m$ for ease of notation), and define
$$Z_{\infty, m}(\omega,x) \coloneqq F_{ m}^{\leftarrow}(x), \quad \omega \in \Omega, \,x \in (0,1).$$

Now consider $L^2(\Omega \times (0,1))$ with the $\ps$ measure. This is a complete metric space, with $d(X,Y) = \mathsf{E}[(X - Y)^2]$. Our aim is to use Lemma \ref{lemma:slutsky} applied to the sequence of random variables $Z_{k,m}$. We proceed therefore to check assumptions~\ref{item:1ar} and~\ref{item:2ar} of the lemma. These will directly follow if we  prove that  
	\begin{equation}\label{eq:ass1}
	\lim_{k \to \infty} \mathsf{E}\left[\left(Z_{k,m} - Z_{\infty,m}\right)^2\right] = 0
	\end{equation}
and
	\begin{equation}\label{eq:conv2}
	\lim_{m \to \infty} \lim_{k \to \infty} \mathsf{E}\left[\left(Z_{k,m} - Z_{k,\infty}\right)^2\right] = 0.
	\end{equation}
We start by \eqref{eq:ass1}. First of all we show that
\begin{equation}\label{eq:theorem2mom:step1}
\lim_{k \to \infty}Z_{k,m} = Z_{\infty, m} \quad \ps\text{-almost surely} .
\end{equation}
Define
$$A\coloneqq A'\times (0,\,1)\coloneqq\left\{ \omega \in \Omega : \lim_{k \to \infty} d_L(\ESD(B_{{N_k},m}), \mu_{\sigma,\tau, m}) = 0, \, \forall m > 1 \right\}\times (0,1).$$
Observe that $\ps(A) = 1$ due to \eqref{eq:almostsure} and $\mathrm{Leb}(0,1) = 1$. To prove \eqref{eq:theorem2mom:step1}, it suffices to show that, for all $\omega \in A'$,
\begin{equation}\label{eq:quantile}
\lim_{k \to \infty}Z_{k,m}(\omega, x) = Z_{\infty, m}(\omega,x), \quad x \in (0,1).
\end{equation}

Let $F_{k,m}(\omega,\cdot)$ be the distribution function of $\ESD(B_{N_k,m}(\omega))$. On $A$, we have $F_{k,m}(\omega, x) \to F_{m}(x)$ for all $x$ at which $F_{m}$ is continuous. Note that
$$Z_{k,m}(\omega, x) = F^{\leftarrow}_{k,m}(\omega,x).$$

It then follows from \citet[Proposition 0.1]{resnick2008extreme} that for all $x \in (0,1)$
\[
\lim_{k \to \infty}F_{k,m}^{\leftarrow}(x) = F_{m}^{\leftarrow}(x).
\]
Thus, we have proved \eqref{eq:theorem2mom:step1}.

Next, we show that for all $m \ge 1$,
\begin{equation}\label{eq:UIz}
\{Z_{k,m}^2 : 1 \le k < \infty\} \text{ is uniformly integrable}.
\end{equation}
It suffices to show that $\sup_{k \ge 1} \mathsf{E}[ Z_{k,m}^4] < \infty$. Since $\lceil N_k x \rceil$ is constant on intervals of length $1/N_k$, it easily follows that
\begin{align*}
    \lim_{k\to\infty}\mathsf{E}[Z_{k,m}^4] &=\lim_{k\to\infty} \frac{1}{N_k} \ep\left[\sum_{i=1}^{N_k} \lambda_i(B_{N_k,m})^4\right] \\
    &=\lim_{k\to\infty} \frac{1}{N_k } \ep\Tr(B_{N_k,m}^4) = \int_{\mathbb{R}} x^4 \, \mu_{\sigma,\tau,m}(\dd{x}) < \infty
\end{align*}
using \eqref{eq:truncatedmoments} and \eqref{eq:limitingmoments}, hence \eqref{eq:UIz} is proven. Using this and \eqref{eq:theorem2mom:step1}, we obtain \eqref{eq:ass1}.

We move to \eqref{eq:conv2}.
To prove this note that 
\begin{align*}
\mathsf{E}\left[\left(Z_{k,m} - Z_{k,\infty}\right)^2\right] &= \frac{1}{N_k}\ep\left[ \sum_{j=1}^{N_k}\big(\lambda_j(B_{N_k,m})- \lambda_{j}(B_{N_k,\infty})\big)^2\right]\\
&\stackrel{\eqref{HW2}}{\le} \frac{1}{N_k}\ep\left[\Tr\big(\left( B_{N_k,m}-B_{N_k,\infty}\right)^2\big)\right]\\
&=\frac{1}{N_k}\ep\left[\sum_{i,j=1}^{N_k}\big(B_{N_k,m}(i,j)- B_{N_k,\infty}(i,j)\big)^2\right].
\end{align*}
Reasoning as in the proof of Lemma \ref{lemma:truncation},  it follows that
\begin{align*}
\frac{1}{N_k}\ep\left[\sum_{i,j=1}^{N_k}(B_{N_k,m}(i,j)- B_{N_k,\infty}(i,j))^2\right]&= \frac{1}{N_k^2}\sum_{i,j=1}^{N_k}\ep\left[\left(\sqrt{\kappa_{\sigma}(W_i^m, W_j^m)}-\sqrt{\kappa_{\sigma}(W_i, W_j)}\right)^2\right]\\
&\le \frac{2}{N_k^2}\sum_{i,j=1}^{N_k}\ep\left[\kappa_{\sigma}(W_i, W_j)\one_{W_j<m<W_i}\right]\\
&+ \frac{2}{N_k^2}\sum_{i,j=1}^{N_k}\ep\left[\kappa_{\sigma}(W_i, W_j)\one_{W_i\ge W_j>m}\right].
\end{align*}
We can use similar bounds as for Lemma \ref{lemma:truncation}, which yield that both summands have order at most $m^{2-\tau}$. Hence \eqref{eq:conv2} follows, since $\tau>2$.

Since we have now checked assumptions ~\ref{item:1ar} and~\ref{item:2ar} of Lemma \ref{lemma:slutsky}, it follows that
there exists $Z_{\infty} \in L^2(\Omega \times (0,1))$ such that
$$\lim_{m \to \infty} \mathsf{E}\left[\left( Z_{\infty, m} - Z_{\infty} \right)^2\right] = 0.$$

Let $U$ be a uniform random variable on $(0,1)$. Then $F_{m}^{\leftarrow}(U)$ has the same distribution as $\mu_{\sigma,\tau,m}$. Since $\mu_{\sigma,\tau, m} $ converges weakly to $ \mu_{\sigma,\tau}$ by~\eqref{eq:important_lim}, $Z_{\infty}$ has law $\mu_{\sigma,\tau}$. Hence
$$\lim_{m\to\infty}\mathsf{E}[Z_{\infty,m}^2] = \lim_{m\to\infty}\int_{\mathbb{R}} x^2 \, \mu_{\sigma,\tau,m}(\dd{x}) = \int_{\mathbb{R}} x^2 \, \mu_{\sigma,\tau}(\dd{x}),$$
and
$$\lim_{m\to\infty}\int_{\mathbb{R}} x^2 \, \mu_{\sigma,\tau,m}(\dd{x})= (\tau-1)^2 \int_1^\infty \int_1^\infty \frac{1}{(x \wedge y)^{\tau - \sigma} (x \vee y)^{\tau - 1}} \, \dd{x} \, \dd{y} $$
which can be easily obtained from \eqref{eq:limitingmoments} with $k=1$.
This completes the proof of the first part. 

\smallskip

Since $\lim_{m\to \infty}\mu_{\sigma,\tau,m} = \mu_{\sigma,\tau}$ weakly, we apply Fatou's lemma to obtain
\begin{align*}
\int x^{2p} \, \mu_{\sigma,\tau}(\dd{x}) &\le \liminf_{m \to \infty} \int x^{2p} \, \mu_{\sigma, \tau, m}(\dd{x}) = \lim_{m \to \infty} M_{2p},
\end{align*}
where, recalling~\eqref{eq:limitingmoments},
$$M_{2p} = \sum_{\pi \in NC_2(2p)} \ep\left[\prod_{(u,v) \in E(G_{\gamma\pi})} \kappa_{\sigma}(W_u^m, W_v^m)\right].$$
For $\sigma > 0$, we observe that $(x\wedge y)^\sigma (x \vee y) \le (xy)^{\sigma \vee 1}$. Thus,
\begin{equation}\label{eq:momentbounds}
M_{2p} \le \sum_{\pi \in NC_2(2p)} \prod_{i=1}^{p+1} \E\left[ (W_i^m)^{(\sigma \vee 1)\#V_i}\right],
\end{equation}
where $\{V_1, \ldots, V_{p+1}\}$ are the blocks of $\gamma\pi$. Due to Lemma \ref{lemma:gammapi}, it follows that $\max_{1 \le i \le p+1} \#V_i \le p$, typically achieved by partitions $\pi$ such that $$\gamma\pi= \{(1, 3, \ldots, 2p-1), (2), (4), \ldots, (2p)\}.$$ 
This shows that the maximum moment bound required for the right-hand side of \eqref{eq:momentbounds} to remain finite is $\E[(W_i)^{p(\sigma \vee 1)}]$. Since $W_i$ has a tail index of $\tau - 1$, if $p(\sigma \vee 1) < \tau - 1$, then $\E[(W_i)^{p(\sigma \vee 1)}] < \infty$. Therefore, $M_{2p} $ is uniformly bounded in $m $, completing the proof of the theorem.
\end{proof}

\section{Absolute continuity and symmetry: proof of Theorem \ref{theorem:absolutecontinuity}}\label{sec:absolutecontinuity}

We begin by showing absolute continuity. We shall use the following fact from \citet[Fact 2.1]{chakrabarty2016remarks}, which follows from \citet[Proposition 22.32]{Nica:Speicher}.
\begin{lemma}\label{lemma:Fact2.1} Assume that, for each $N, A_N$ is a $N \times N$ Gaussian Wigner matrix scaled by $\sqrt{N}$, that is, $\left(A_N(i, j): 1 \leq i \leq j \leq N\right)$ are i.i.d.~normal random variables with mean zero and variance $1 / N$, and $A_N(j, i)=A_N(i, j)$. Suppose that $B_N$ is a $N \times N$ random matrix, such that for all $k\geq 1$
$$
\lim_{N\to\infty}\frac{1}{N} \operatorname{Tr}\left(B_N^k\right)= \int_{\mathbb{R}} x^k \mu(\dd x)
$$
in probability,
for some compactly supported (deterministic) probability measure $\mu$. Furthermore, let the families $\left(A_N: N \geq 1\right)$ and $\left(B_N: N \geq 1\right)$ be independent. Then for all $k\geq 1$

$$
\lim_{N\to\infty}\frac{1}{N} \mathrm{E}_{\mathcal{F}} \operatorname{Tr}\left[\left(A_N+B_N\right)^k\right] =\int_{\mathbb{R}} x^k \mu \boxplus \mu_{sc}(\dd x) 
$$
in probability,
where $\mathcal{F}:=\sigma\left(B_N: N \geq 1\right)$ and $\mathrm{E}_{\mathcal{F}}$ denotes the conditional expectation with respect to $\mathcal{F}$.
\end{lemma}

\begin{proof}[Proof of Theorem \ref{theorem:absolutecontinuity}]
We consider the truncated weights $(W_i^m)_{i \ge 1}$. Let $\Gamma_{m}$ be an $N \times N$ matrix with entries given by
$$\Gamma_{m}(i,j) = \sqrt{\kappa_\sigma(W_i^m, W_j^m)}.$$

Given $\delta \in (0,1)$, define the function $g_{\delta,m}$ such that
\begin{equation*}
g_{\delta,m}(W_i^m, W_j^m)^2 = \left(\sqrt{\kappa_\sigma(W_i^m, W_j^m)} - \delta\right)^2 + 2\delta\left(\sqrt{\kappa_\sigma(W_i^m, W_j^m)} - \delta\right).
\end{equation*}
As a consequence
\begin{equation}\label{eq:varmatch}
g_{\delta,m}(W_i^m, W_j^m)^2 + \delta^2 = \kappa_\sigma(W_i^m, W_j^m)\,.
\end{equation}

Define the matrix $\Gamma_{g_{\delta,m}}(i,j) = g_{\delta,m}(W_i^m, W_j^m)$. Let $\{G_{i,j}\}_{1 \leq i,j \leq N}$ be i.i.d.~standard Gaussian random variables, independent of the sequence $(W_i)_{i \ge 1}$. Denote by $\mfG_N$ the matrix with entries
\[
\mfG_N(i,j) = \frac{1}{\sqrt{N}} G_{i \wedge j, i \vee j}\,.
\]
Define
\[
\B_{N,m}^{(1)} = \Gamma_{m} \circ \mfG_N\,.
\]
Similarly, define 
\[
\B_{N,m}^{(2)} = \Gamma_{g_{\delta,m}} \circ \mfG_N\,.
\] 
Lastly, consider a sequence of i.i.d.~standard Gaussian random variables $(G'_{i,j})_{1 \leq i,j \leq N}$, independent of the sigma field $\mathcal{F}$ generated by $(W_i)_{i \ge 1}, (G_{i,j})_{i,j \ge 1}$. Define a matrix $\B_{N,m}^{(3)}$ with entries
\[
\B_{N,m}^{(3)}(i,j) = \frac{1}{\sqrt{N}} G'_{i \wedge j, i \vee j}\,.
\]

We claim that, conditionally on $(W_i)_{i \in [N]}$, 
\begin{equation}\label{eq:B_N 1,2, and 3}
\B_{N,m}^{(1)} \stackrel{d}{=} \B_{N,m}^{(2)} + \delta \B_{N,m}^{(3)}.
\end{equation}
Indeed, conditionally on $(W_i)_{i \in [N]}$, the entries of $\B_{N,m}^{(1)}$, $\B_{N,m}^{(2)}$, and $\B_{N,m}^{(3)}$ are normally distributed. Thus, it is sufficient to compare the mean and variance of the entries. All the variables in question have mean zero and the variances match, too, due to \eqref{eq:varmatch}.
Following Proposition~\ref{prop:withoutalpha}, there exists a measure $\mu_{g_{\delta, m}}$ such that
\[
\lim_{N\to\infty}\frac{1}{N} \Tr\left((\B_{N,m}^{(2)})^k\right) =\int_{\mathbb{R}} x^k \, \mu_{g_{\delta, m}}(\De x)
\]
in probability.
In particular, we recall the expression for the even moments of $\mu_{g_{\delta, m}}$ given  in~\eqref{eq:limitingmoments}:
\[
M_{2k} = \sum_{\pi \in NC_2(2k)} \mathbb{E}\left[\prod_{(u,v) \in E(G_{\gamma\pi})} g_{\delta, m}^2(W_u^m, W_v^m)\right].
\]
Since $g_{\delta,m}^2(W_u^m, W_v^m) \le \kappa_\sigma(W_u^m, W_v^m)$, it follows that $\mu_{g_{\delta, m}}$ is uniquely determined by its moments, and is also compactly supported (Corollary~\ref{cor:general_kappa}). This verifies the first condition of Lemma \ref{lemma:Fact2.1}. Since $\B_{N,m}^{(3)}$ is a standard Wigner matrix, it follows from Lemma \ref{lemma:Fact2.1} that
\[
\lim_{N\to\infty}\frac{1}{N} \mathbb{E}_{\mathcal{F}}\left[\Tr\big((\B_{N,m}^{(2)} + \delta \B_{N,m}^{(3)})^k\big)\right] = \int_{\mathbb{R}} x^k \, (\mu_{g_{\delta,m}} \boxplus \mu_{sc,\delta})(\De x),
\]
where $\mu_{sc,\delta}$ is the semicircular law with variance $\delta^2$ and density
\[
\mu_{sc,\delta}(\dd x) = \frac{1}{2 \pi \delta} \sqrt{4 - \left(\frac{x}{\delta}\right)^2} \one_{|x| \leq 2 \delta} \, \dd x, \quad x \in \mathbb{R}.
\]
Since both $\mu_{g_{\delta,m}}$ and $\mu_{sc,\delta}$ are compactly supported, so is $\mu_{g_{\delta,m}} \boxplus \mu_{sc,\delta}$, and thus the measure is completely determined by its moments. 

From Proposition \ref{prop:GaussianESD} we have
\[
\lim_{N\to\infty}\mathbb{E}\left[\frac{1}{N} \mathbb{E}_{\mathcal{F}}[\Tr(\B_{N,m}^{(1)})^k]\right] = \int_{\mathbb{R}} x^k \, \mu_{\sigma,\tau,m}(\dd x)
\]
and
\[
\lim_{N\to\infty}\mathrm{Var}\left(\frac{1}{N} \mathbb{E}_{\mathcal{F}}[\Tr((\B_{N,m}^{(1)})^k)]\right) \le \lim_{N\to\infty}\mathrm{Var}\left(\frac{1}{N} \Tr((\B_{N,m}^{(1)})^k)\right) = 0.
\]
Thus,
$$\lim_{N\to\infty}\frac{1}{N} \mathbb{E}_{\mathcal{F}}\left[\Tr(\B_{N,m}^{(1)})^k\right] = \int_{\mathbb{R}} x^k \, \mu_{\sigma,\tau,m}(\dd x)$$
in probability.
Since the measures are uniquely determined by their moments, this shows that 
\begin{equation}\label{eq:help_local}   
\mu_{\sigma,\tau, m} =  \mu_{g_{\delta,m}} \boxplus \mu_{sc,\delta}. 
\end{equation}
We show that there exists $\mu_{g_{\delta}}$ such that
\begin{equation}\label{eq:dg}
\lim_{m\to\infty}d_L(\mu_{g_{\delta, m}}, \mu_{g_{\delta}}) = 0.
\end{equation}
If we can prove this, using \citet[Proposition 4.13]{bercovici1993free} it will follow that
\begin{equation}\label{eq:BerVoideltam}
\lim_{m\to\infty}d_L(\mu_{g_{\delta,m}} \boxplus \mu_{sc,\delta}, \mu_{g_{\delta}} \boxplus \mu_{sc,\delta}) \le \lim_{m\to\infty}d_L(\mu_{g_{\delta, m}}, \mu_{g_{\delta}}) = 0.
\end{equation}
To show \eqref{eq:dg}, we  employ Lemma \ref{lemma:slutsky}. Note that, from Remark~\ref{cor:general_kappa}, we get that for any fixed $m \ge 1$ one has
\[\lim_{N\to\infty}d_L\big(\mu_{\B_{N,m}^{(2)}}, \mu_{g_{\delta,m}}\big)=0\quad\text{ in $\mathbb{P}$-probability} \] 
where $\mu_{\B_{N,m}^{(2)}}$ is the empirical spectral distribution of $\B_{N,m}^{(2)}$.

This establishes condition \ref{item:1ar} of Lemma \ref{lemma:slutsky}. To complete the proof, we need to verify condition \ref{item:2ar}, namely,
\begin{equation}\label{eq:rank_trick}
\lim_{m \to \infty}\limsup_{N \to \infty} \mathbb{P}\big(d_L(\ESD(\B_{N,m}^{(2)}), \ESD(\B_{N}^{(2)})) > \varepsilon\big) = 0.
\end{equation}
 Here $\B_N^{(2)}$ is defined as $\B_{N,\infty}^{(2)}$ with $m = \infty$. 
 From Proposition \ref{HWinequality} we see that 
\begin{align*}
d_L\left(\ESD(\B_{N,m}^{(2)}), \ESD(\B_{N}^{(2)})\right)^3 &\le \frac{1}{N}\Tr\left( \big(\B_{N,m}^{(2)}- \B_N\big)^2\right)\\
&= \frac{1}{N^2}\sum_{i,j=1}^N \left(\Gamma_{g_{\delta,m}}(i,j)-\Gamma_{g_{\delta,\infty}}(i,j)\right)^2 G_{i\wedge j,i\vee j}^2\,.
\end{align*}
Hence we have
\begin{align*}
&\E\left[d_L\big(\ESD(\B_{N,m}^{(2)}), \ESD(\B_{N}^{(2)})\big)^3\right]\leq \frac1{N^2}\sum_{i\neq j=1}^N\ep[\left(\Gamma_{g_{\delta,m}}(i,j)-\Gamma_{g_{\delta,\infty}}(i,j)\right)^2]\\
&\qquad\leq \frac{2}{N^2}\sum_{i\neq j=1}^N\ep\left[ g_{\delta,\infty}(W_i,\,W_j)^2\left(\one_{W_j<m<W_i}+\one_{ W_i>W_j>m}\right)\right]\\
&\qquad\le \frac{2}{N^2}\sum_{i\neq j=1}^N\ep\left[ \kappa_{\sigma}(W_i, W_j)\left(\one_{W_j<m<W_i}+\one_{ W_i>W_j>m}\right)\right].
\end{align*}
Just as in the proof of \eqref{eq:conv2}, it follows that the last term is bounded by $Cm^{2-\tau}$. Thus, using Markov's inequality, condition \ref{item:2ar}  of Lemma \ref{lemma:slutsky} holds, too. 
In conclusion, we can show that there exists $\mu_{g_{\delta}}$ such that
\begin{align*}
\lim_{m\to\infty}d_L( \mu_{g_{\delta,m}} \boxplus \mu_{sc,\delta}, \mu_{\sigma,\tau}) &\stackrel{\eqref{eq:help_local}}{=}\lim_{m\to\infty} d_L(\mu_{\sigma,\tau,m}, \mu_{\sigma,\tau}) \stackrel{\eqref{eq:important_lim}}{=}0\\
&\stackrel{\eqref{eq:BerVoideltam}}{=}\lim_{m\to\infty}d_L(\mu_{g_{\delta,m}} \boxplus \mu_{sc,\delta}, \mu_{g_{\delta}} \boxplus \mu_{sc,\delta}).  
\end{align*}
Therefore it must be that $\mu_{\sigma,\tau} = \mu_{g_{\delta}} \boxplus \mu_{sc,\delta}$. 
The right-hand side is absolutely continuous, as shown by~\citet[Corollary 2]{Biane97}.

Finally, to show symmetry, we see that $\mu_{\sigma,\tau}$ does not give weight to singletons by absolute continuity. Therefore, in light of the weak convergence stated in~\eqref{eq:important_lim},
\[
\mu_{\sigma,\,\tau}(-\infty,\,-x)=\lim_{m\to\infty}\mu_{\sigma,\,\tau,\,m}(-\infty,\,-x)=\lim_{m\to\infty}\mu_{\sigma,\,\tau,\,m}(x,\,+\infty)=\mu_{\sigma,\,\tau}(x,\,+\infty)
\]
for all $x\geq 0$. This completes the proof. 
\qedhere
\end{proof}

\section{Stieltjes transform: proof of Theorem \ref{theorem:stieltjesfinal}}\label{sec:stieltjes}
To prove Theorem \ref{theorem:stieltjesfinal}, we first identify the Stieltjes transform for the measure $\mu_{\sigma,\tau,m}$. We then proceed to take the limit $m\to\infty$, which requires a functional analytic approach. 
Throughout this section, we fix $z\in\C^+$, given as $z=\xi+\ota\eta$ with $\eta>0$. If $\mu$ is a probability measure having all its moments $\{m_k\}_{k\geq 1}$, it follows from the definition of Stieltjes transform \eqref{def:ST}  that, for any $z\in\C^+$, 
\begin{equation}\label{eq:Stieltjesmomentsrelation}
\mathrm{S}_{\mu}(z) = -\sum_{k\geq0}\frac{m_k}{z^{k+1}},
\end{equation}
where the Laurent series on the right-hand side of \eqref{eq:Stieltjesmomentsrelation} converges for $|z|>R>0$, with $\mathrm{supp}(\mu) = [-R,R]$. 

\subsection{Stieltjes transform for $\mu_{\sigma,m,\tau}$}\label{subsec:Stieltjestruncated}
To derive a characterisation of the limiting measure $\mu_{\sigma,\tau}$, we need to first study the truncated version $\mu_{\sigma,\tau,m}$. We borrow ideas from the proof of \citet[Theorem 4.1]{Chakrabarty:Hazra:Sarkar:2015}. The main result of this subsection will be Proposition \ref{Stieltjes:finitem}, which requires a few technical lemmas to prove. The results in this subsection hold for the regime $\tau>2$ and $\sigma<\tau-1$, as before. 

We have that the (even) moments for the measure $\mu_{\sigma,\tau,m}$ are given by \eqref{eq:limitingmoments}. Using these, we derive a representation of $\mathrm{S}_{\mu_{\sigma, \tau,m}}(z)$. 

\begin{proposition}
    \label{Stieltjes:finitem}
For $\tau>2$ and $\sigma\in(0,\,\tau-1)$ there exists a function $a(z,x)=a_m(z,x)$ defined on $\C^+\times [1,\infty)$ such that
 \[\mathrm{S}_{\mu_{\sigma,\tau,m}}(z) = \int_1^\infty a(z,x) \mu_{W,m}(\dd{x})\,,\]
 where $\mu_{W,m}$ is the law of the truncated weights $(W_i^m)$.
 Moreover, $a(z,x)$ satisfies the following recursive equation:
\begin{equation}\label{eq: a(z,x) expansion}
   a(z,x)\left(z + \int_1^{\infty}a(z,y)\kappa_\sigma(x,y)\mu_{W,m}(\dd{y})\right) = -1\,.
\end{equation}
\end{proposition}

Before tackling the proof of the proposition, we lay the ground with two auxiliary results.
For any $k\geq1$ and $\pi\in NC_2(2k)$, recall the map $\Tpi$ of~\eqref{def:Tpi},
where $\gamma\pi = \{V_1,\ldots,V_{k+1}\}$. Consider the mapping $L_\pi:[1,\infty)^{k+1} \to \mathbb{R}$ defined as 
\begin{align}\label{eq:Lmapping}
    L_\pi(\mathbf{x}) =\kappa_{\sigma}^{1/2}(x_{\Tpi(1)},x_{\Tpi(2)})\kappa_{\sigma}^{1/2}(x_{\Tpi(2)},x_{\Tpi(3)})\ldots\kappa_{\sigma}^{1/2}(x_{\Tpi(2k)},x_{\Tpi(1)})
\end{align}
and the function $H_\pi:\mathbb R\to\mathbb R^{+}$ given as 
\begin{align}\label{eq:Hmapping}
    H_\pi(y) = \int_{[1,\infty)^k}L_\pi(y,\,x_2,\ldots,\,x_{k+1})\mu_{W,m}^{\otimes k}(\dd{\mathbf{x}'})\,,
\end{align}
where we are integrating over $\mathbf x'=(x_2,\dots,x_{k+1})\in[1,\infty)^k$.
\begin{lemma}\label{lemma:moments_Hpi}
Let $\{M_{2k}\}_{k\geq1}$ be as in \eqref{eq:limitingmoments}. Then
\begin{align*}
    M_{2k} = \sum_{\pi\in NC_2(2k)}\int_1^{\infty}H_\pi(y)\mu_{W,m}(\dd{y}).
\end{align*}
    
\end{lemma}
\begin{proof}[Proof of Lemma \ref{lemma:moments_Hpi}]
    We begin by evaluating the integral on the right-hand side. We have 
    \begin{align*}
        \int_1^{\infty} H_\pi(y)\mu_{W,m}(\dd{y}) 
        &= \int_1^{\infty}\int_{[1,\infty)^k}L_\pi(y,x_2,\ldots,x_{k+1})\mu_{W,m}^{\otimes k}(\dd{\mathbf{x}'})\mu_{W,m}(\dd{y})\\
        &= \int_{[1,\infty)^{k+1}}\kappa_{\sigma}^{1/2}(x_{\Tpi(1)},x_{\Tpi(2)})\cdots\kappa_{\sigma}^{1/2}(x_{\Tpi(2k)},x_{\Tpi(1)})\mu_{W,m}^{\otimes k+1}(\dd{\mathbf{x}}).
    \end{align*}
    We know that, for $\pi\in NC_2(2k)$, $\#\gamma\pi=k+1$ and so the graph $G_{\gamma\pi}$ has $k+1$ vertices. Furthermore, when we perform a closed walk of the form $1\to2\to\ldots\to2k\to1$ on the (unoriented) graph $G_{\gamma\pi}$, we traverse each edge \emph{exactly} twice. In particular, the product $\kappa_{\sigma}^{1/2}(x_{\Tpi(1)},x_{\Tpi(2)})\cdots\kappa_{\sigma}^{1/2}(x_{\Tpi(2k)},x_{\Tpi(1)})$ has $2k$ terms with $k$ matchings, and so
    \[
    \kappa_{\sigma}^{1/2}(x_{\Tpi(1)},x_{\Tpi(2)})\cdots\kappa_{\sigma}^{1/2}(x_{\Tpi(2k)},x_{\Tpi(1)}) = \prod_{(u,v)\in E(G_{\gamma\pi})} \kappa_{\sigma}(x_u,x_v)\,.
    \]
    We then have that 
    \begin{align*}
    \int_1^{\infty} H_\pi(y)\mu_{W,m}(\dd{y}) &= \int_{[1,\infty)^{k+1}}\prod_{(u,v)\in E(G_{\gamma\pi})} \kappa_{\sigma}(x_u,x_v) \mu_{W,m}^{\otimes k+1}(\dd{\mathbf{x}}) \\
    &= \E\Bigg[\prod_{(u,v)\in E(G_{\gamma\pi})} \kappa_{\sigma}(W_u^m,W_v^m)  \Bigg]\,,
    \end{align*}
    which concludes the proof.
\end{proof}
We  show now some properties of $H_\pi$ that will help us in the upcoming computations. 
\begin{lemma}\label{lemma:Hproperties}
    Let $k\geq1$ and let $H_\pi$ be as defined in \eqref{eq:Hmapping}. Let $\pi\in NC_2(2k)$. Then, 
    \begin{enumerate}[label=(\arabic*),ref=(\arabic*)]
\item\label{item:H_(1,2k)prop} If $\pi=(1,2k)\cup \pi_1$, where $\pi_1$ is a non-crossing pair partition of $\{2,\ldots,2k-1\}$, then, 
\begin{align}\label{eq:H_property1}
    H_\pi(y) = \int_1^{\infty}H_{\pi_1}(x)\kappa_\sigma(x,y)\mu_{W,m}(\dd{x}).
\end{align}
\item\label{item:H_pi1pi2prop}
If $\pi = \pi_1\cup\pi_2$, then $H_\pi(\cdot) = H_{\pi_1}(\cdot)H_{\pi_2}(\cdot)$. 
\end{enumerate}
\end{lemma}
\begin{proof}[Proof of Lemma \ref{lemma:Hproperties}]
    We first prove property \ref{item:H_(1,2k)prop}. Let $\pi=(1,2k)\cup\pi_1$. Then, $\gamma\pi = \{(
    1),V_2,\ldots,V_{k+1}\}$. We know that $2\in V_2$ and then $\gamma\pi(2k)=2\in V_2$.  
    Now, fix $x_1=y$. Then 
    \begin{align*}
        H_\pi(y) &= \int_{[1,\infty)^k}L_\pi(y,x_2,\ldots,x_{k+1})\mu_{W,m}^{\otimes k}(\dd{\mathbf{x}')}\\
        &=\int_{[1,\infty)^k}\kappa_{\sigma}^{1/2}(y,x_2)\kappa_{\sigma}^{1/2}(x_2,x_{\Tpi(3)})\ldots\kappa_{\sigma}^{1/2}(x_{\Tpi(2k-1)},x_2)\kappa_{\sigma}^{1/2}(x_2,y)\mu_{W,m}^{\otimes k}(\dd{\mathbf{x}')}\\
        &= \int_1^{\infty}\kappa_\sigma(y,x_2)\int_{[1,\infty)^{k-1}}\kappa_{\sigma}^{1/2}(x_2,x_{\Tpi(3)})\ldots\kappa_{\sigma}^{1/2}(x_{\Tpi(2k-1)},x_2)\mu_{W,m}^{\otimes k-1}(\dd (x_3,\dots,x_{k+1}))\mu_{W,m}(\dd x_2)
        \\
        &= \int_1^{\infty}\kappa_\sigma(y,x_2)H_{\pi_1}(x_2)\mu_{W,m}(\dd{x_2}),
    \end{align*}
    which is what we desired. 
    
    For property \ref{item:H_pi1pi2prop}, let $\pi=\pi_1\cup\pi_2$, with $\pi_1\in NC_2(\{1,2,\ldots,2r\})$ and $\pi_2\in NC_2(\{2r+1,\ldots,2k\})$ and let us consider the function $H_{\pi}(y)$ with $y= x_{1}= x_{\Tpi(1)}$. Then, 
    \begin{align*}
        H_\pi(y) = \int_{[1,\infty)^k}\kappa_\sigma^{1/2}(y,x_{\Tpi(2)})\ldots \kappa_\sigma^{1/2}(x_{\Tpi(2r)},x_{\Tpi(2r+1)})\ldots \kappa_\sigma^{1/2}(x_{\Tpi(2k)},\, y)\mu_{W,m}^{\otimes k}(\dd{\mathbf{x}}').
    \end{align*}
  We now claim that this integral can be split up into two integrals. First, consider the element $x_{\Tpi(1)}$. Since we assume that `1' maps to $V_1\in\gamma\pi$, all elements of $V_1$ are mapped to $y$. 
     To understand where other elements are mapped, we will state a claim and see its consequences to this proof, and then prove it on~\pageref{proof:gammaupi}.
     \begin{claim}\label{claim:gammapi on pi1Upi2}
       Under $\gamma\pi$, the elements $\{2,\ldots,2r\}$ are mapped to the blocks $V_1\cup\{V_2,\ldots,V_{r'}\}\subset \gamma\pi$, and the elements $\{2r+1,\ldots,2k\}$ are mapped to the blocks $V_1\cup \{V_{r'+1},\ldots,V_{k+1}\}\subset\gamma\pi$, where $r'<k+1$ is some index. In particular $\gamma\pi(2r+1)\in V_1$.  
     \end{claim}

    From this Claim we have that 
   \begin{align*} 
       H_\pi(y)&=\int_{[1,\infty)^k}\kappa_\sigma^{1/2}(y,x_{\Tpi(2)})\ldots \kappa_\sigma^{1/2}(x_{\Tpi(2r)},x_{\Tpi(2r+1)})\ldots \kappa_\sigma^{1/2}(x_{\Tpi(2k)},y)\mu_{W,m}^{\otimes k}(\dd{\mathbf{x}}')\\
       &=\int_{[1,\infty)^k}\kappa_\sigma^{1/2}(y,x_{\Tpi(2)})\ldots \kappa_\sigma^{1/2}(x_{\Tpi(2r)},y)\ldots \kappa_\sigma^{1/2}(x_{\Tpi(2k)},y)\mu_{W,m}^{\otimes k}(\dd{\mathbf{x}}')\\
       &= \int_{[1,\infty)^{r'}}\kappa_\sigma^{1/2}(y,x_{\Tpi(2)})\ldots \kappa_\sigma^{1/2}(x_{\Tpi(2r)},y)\mu_{W,m}^{\otimes r'}(\dd{\mathbf{x}^{(r')}})\\
       &\times\int_{[1,\infty)^{k-r'}}\kappa_\sigma^{1/2}(y,x_{\Tpi(2r+2)})\ldots \kappa_\sigma^{1/2}(x_{\Tpi(2k)},y)\mu_{W,m}^{\otimes (k-r')}(\dd{\mathbf{x}^{(k-r')}})\\
       &= H_{\pi_1}(y)H_{\pi_2}(y).
   \end{align*}
   This concludes the proof.
\end{proof}

\begin{proof}[Proof of Claim \ref{claim:gammapi on pi1Upi2}]\label{proof:gammaupi}
   Let $\gamma_1$ resp. $\gamma_2$ be the shift by one on $[2r]$ resp. $\{2r+1,\, \ldots, \,2k\}$. To prove this claim, it suffices to analyse the special indices $\{1,2r,2r+1,2k\}$, since $\gamma_1$ and $\gamma_2$ are cyclic permutations on $[2r]$ and $\{2r+1,\, \ldots, \,2k\}$, respectively. We will be using the fact that all elements in a block of $\gamma\pi$ must be either all odd or all even \citep[Property 1]{avena}, and that any pairing in $\pi$ must have one element odd and the other even \citep[Property 2]{avena}. \begin{enumerate}
        \item We already have $1\in V_1$. Now, let $(o_1,2k)\in\pi_2$, for some $o_1$ such that $o_1\geq 2r+1$. Then, $o_1$ must be odd. Now, $o_1+1$ is even, and cannot belong to $V_1$. Thus $\gamma\pi(2k)=o_1+1\in\{V_{r'+1},\ldots,V_{2k}\}$. This takes care of the index $2k$.
        \item Let us continue with $(o_2,2r)\in\pi_1$ for some $o_2$. We know that $o_2$ must be odd. Thus, $\gamma\pi(2r)=o_2+1\in\{V_2,\ldots,V_{r'}\} =:\gamma_1\pi_1\setminus V_1$. This resolves the case of $2r$. 
        \item Lastly, by construction, $\gamma\pi(o_2)=2r+1$, which brings us to the last special element. Since $o_2$ and $2r+1$ belong to the same block in $\gamma\pi$, it suffices to show that this block is $V_1$, that is, the block to which element 1 belongs. Now, if $(1,o_2-1)\in\pi_1$, we are done, since $\gamma\pi(1)=o_2$. Suppose not, and let $(1,e_1)\in\pi_1$ for some even integer $e_1$. Similarly as before, if now $(e_1+1,o_2-1)\in\pi_1$, we are done. Since $\pi_1$ and $\pi_2$ act on the first $2r$ elements and the remaining $2k-2r$ elements respectively, then, by the non-crossing nature, there is a sequence of even integers $\{e_i\}_{i=1}^t$ such that $(1,e_1),(e_1+1,e_2),\ldots,(e_t+1,o_2-1)\in\pi_1$. Computing $\gamma\pi$ recursively gives us that $\gamma\pi(1)=o_2$, and so $\gamma\pi(2r+1)\in V_1$. 
        \end{enumerate}
        This proves the claim. 
\end{proof}

We are now ready to prove Proposition \ref{Stieltjes:finitem}. \begin{proof}[Proof of Proposition~\ref{Stieltjes:finitem}]
We now derive the Stieltjes transform of the measure $\mu_{\sigma,\tau,m}$. Using \eqref{eq:Stieltjesmomentsrelation} and Proposition \ref{prop:GaussianESD}, we have that 
\[
\mathrm{S}_{\mu_{\sigma,\tau,m}}(z) = -\sum_{k\geq0}\frac{M_{2k}}{z^{2k+1}}.
\]
Using Lemma \ref{lemma:moments_Hpi} we substitute the expression for $M_{2k}$. We have
\begin{align}
\mathrm{S}_{\mu_{\sigma,\tau,m}}(z) &= -\sum_{k\geq0}\frac{1}{z^{2k+1}}\int_1^{\infty}\sum_{\pi\in NC_2(2k)} H_\pi(x)\mu_{W,m}(\dd{x}) \nonumber\\
&=-\int_1^\infty \sum_{k\geq0}\sum_{\pi\in NC_2(2k)}\frac{H_\pi(x)}{z^{2k+1}}\mu_{W,m}(\dd{x}) \label{eq:st100},
\end{align}
where we could interchange the integral and the sum by  Fubini's theorem. Now, we define the function $a(z,x)$ as 
\begin{equation}\label{eq:a(z,x)}
    a(z,x):= -\sum_{k\geq0}\sum_{\pi\in NC_2(2k)}\frac{H_\pi(x)}{z^{2k+1}}.
\end{equation}
Then using \eqref{eq:st100} we have \[\mathrm{S}_{\mu_{\sigma,\tau,m}}(z) = \int_1^\infty a(z,x) \mu_{W,m}(\dd{x}).\]
We now state some properties of the function $a(z,x)$.  Firstly, for any $z\in\C^+$ the map $x\mapsto a(z,x)$ is in $L^{\infty}([1,\infty),\mu_{W,m})$ as $H_\pi$ is bounded .  Secondly, for any $x\in[1,\infty)$, the map $z\mapsto a(z,x)$ is analytic in $\C$, which follows from the Laurent series expansion. Finally we see that $a(z,x)$ lies in $\C^+$, for any $z\in\C^+$ and $x>1$.  Indeed, for any $\Im(z)>0$, the expansion on the right-hand side of \eqref{eq:a(z,x)} will always have a non-trivial imaginary part. Thus, since $a(\cdot,\cdot)$ is analytic, it will either lie completely in $\C^-$ or $\C^+$, since  it can never take values in $\Rr$. However, $\mathrm{S}_{\mu_{W,m}}(z) \in \C^+$, and thus, $a(z,x)\in\C^+$ for any $z\in\C^+$ and $x>1$.

To write down a functional recursion for $a(\cdot,\cdot)$ it is convenient to use the notion of words. Any partition $\pi$ can be associated to a \emph{word} $w$, with any elements in $i,j\in[2k]$ being associated with the same letter in $w$ if $i,j$ are in the same block of $\pi$. For example, $\pi=\{\{1,2\}, \{3,4\}\}$ can be written as $w=aabb$. In particular, any partition $\pi\in NC_2(2k)$ can be associated to a word $w$ of the form $w=aw_1aw_2$, where $w_1, w_2$ are words that can be empty. For any word $w$ associated to a partition $\pi$, let $H_\pi=H_w$. Furthermore, for $w\in NC_2(2k)$ we mean a word $w$ whose associated partition $\pi$ is in $NC_2(2k)$. Then we have, using Lemma \ref{lemma:Hproperties} in the third equality, 
\begin{align}
    a(z,x) &= -\sum_{k\ge 0}\sum_{w\in NC_2(2k)} \frac{H_w(x)}{z^{2k+1}} \nonumber \\
    &= -\frac{1}{z} -\sum_{k\geq 1}\sum_{\substack{w\in NC_2(2k)\\ w = aw_1aw_2}} \frac{H_{aw_1aw_2}(x)}{z^{2k+1}} \nonumber\\
    &= -\frac{1}{z} - \sum_{k\geq 1}\sum_{\substack{w\in NC_2(2k)\\ w = aw_1aw_2}}  \frac{H_{aw_1a}(x)H_{w_2}(x)}{z^{2k+1}} \nonumber\\
    &=-\frac{1}{z} -\frac{1}{z}\sum_{k\geq 1}\sum_{\ell=1}^k \sum_{w_1\in NC_2(2\ell-2)}\frac{H_{aw_1a}(x)}{z^{2\ell-2+1}}\sum_{w_2\in NC_2(2k-2\ell)} \frac{H_{w_2}(x)}{z^{2k-2\ell +1}}.
\end{align}
One can see that the word $aw_1a$ has as corresponding partition $(1,2\ell)\cup\pi_1$, with $\pi_1\in NC_2(2\ell-2)$. Using \eqref{eq:H_property1} from Lemma \ref{lemma:Hproperties}, we have 
\begin{align*}
    a(z,x) &= -\frac{1}{z} -\frac{1}{z}\sum_{k\geq1}\sum_{\ell=1}^k\sum_{w_1\in NC_2(2\ell - 2)}\frac{1}{z^{2\ell-1}}\int_1^{\infty}H_{w_1}(y)\kappa_{\sigma}(x,y)\mu_{W,m}(\dd{y})\sum_{w_2\in NC_2(2k-2\ell)} \frac{H_{w_2}(x)}{z^{2k-2\ell +1}} \nonumber \\
    &=-\frac{1}{z}-\frac{1}{z}\sum_{k\geq1}\sum_{\ell=1}^k\sum_{\pi_2\in NC_2(2k-2\ell)}\frac{H_{\pi_2}(x)}{z^{2k-2\ell +1}}\sum_{\pi_1\in NC_2(2\ell-2)}\frac{1}{z^{2\ell-1}}\int_1^{\infty}H_{\pi_1}(y)\kappa_{\sigma}(x,y)\mu_{W,m}(\dd{y}) \nonumber\\
    &= -\frac{1}{z} - \frac{a(z,x)}{z}\int_1^{\infty}a(z,y)\kappa_\sigma(x,y)\mu_{W,m}(\dd{y}).
\end{align*}
Thus, we have \eqref{eq: a(z,x) expansion}, which completes the proof of Proposition \ref{Stieltjes:finitem}.
\end{proof}

\begin{remark}
Equation \eqref{eq: a(z,x) expansion} gives an analytic description of $a$ in terms of the recursive equation. Now, for any $z\in \C^+$, we have that 
\begin{equation}\label{eq:complex identity}
z = \ota\int_0^{\infty} \e^{-\ota tz^{-1}}\dd{t}.
\end{equation}
Since $a(z,x)\in \C^+$ for any fixed $x\in[1,\infty)$, applying \eqref{eq:complex identity} to $a(z,x)$ and using \eqref{eq: a(z,x) expansion} gives us that 
\begin{align}\label{eq:a(z,x) complex identity}
    a(z,x) = \ota\int_0^{\infty}\e^{\ota tz}\exp\left\{ \ota t\int_1^{\infty}a(z,y)\kappa_\sigma(x,y)\mu_{W,m}(\dd{y})\right\}\dd{t}.
\end{align}
An immediate consequence of \eqref{eq:a(z,x) complex identity} is that $a(z,x)$ is uniformly bounded in $x$ and $m$. Indeed, if we take $z=\xi + \ota\eta$ with $\eta>0$, we have that 
\begin{align}\label{eq:a(z,x) bound}
    |a(z,x)| &\leq \int_0^{\infty} \e^{-\eta t} \left| \exp\left\{ \ota t\int_1^{\infty}a(z,y)\kappa_\sigma(x,y)\mu_{W,m}(\dd{y})\right\}  \right| \dd{t}\nonumber \\
    &\leq \int_0^{\infty} \e^{-\eta t}\dd{t} = \frac{1}{\eta}.
\end{align}
The bound in the second line holds since $a(z,x)\in\C^+$, and so $\int_1^{\infty}a(z,y)\kappa_\sigma(x,y)\mu_{W,m}(\dd{y}) \in \C^+$ as $\kappa_\sigma\geq 1$. 

\end{remark}

\subsection{Limiting Stieltjes transform} We now set up the framework required to prove Theorem \ref{theorem:stieltjesfinal}. For the remainder of this section, denote $a_z(x):=a(z,x)$, which implicitly depends on $m$. We wish to extend Proposition \ref{Stieltjes:finitem} to the measure $\mu_{\sigma,\tau}$ by passing to the limit $m\to\infty$. 
We have a natural candidate for the function $a^*$ in Theorem \ref{theorem:stieltjesfinal}, which should be the limit of $a(\cdot,\cdot)$ as $m$ tends to infinity. We now formalise this idea through a series of lemmas.  

Since our goal now is to show Theorem~\ref{theorem:stieltjesfinal} we are going to work for the remainder of this section with the following parameters:
\begin{enumerate}
    \item $\tau>3$, 
    \item $\sigma<\tau -2$, and 
    \item a parameter $\beta$ such that $2\vee1+\sigma<\beta<\tau -1$.
\end{enumerate}
Let $\overline{\C}^+ = \C^+\cup \Rr$ be the closure of $\C^+$, and let $\nu$ be the measure defined as 
\begin{equation}\label{eq:nu(dx)}
\nu(\dd{x}) = x^{-\beta}\dd{x}.
\end{equation} 
Consider the space $L^1([1,\infty),\nu)$ of all functions $f:[1,\infty)\to \overline{\C}^+$ that are $L^1-$integrable with respect to $\nu$.
\begin{definition}\label{def:Banach space}
Let $\mathcal{B}$ denote the Banach space $\mathcal{B} := (L^1([1,\infty),\nu), \|\cdot\|_1)$, where the norm $\|\cdot\|_1$ is the $L^1$ norm with respect to $\nu$ as in \eqref{eq:nu(dx)}, which is defined for $f\in L^1([1,\infty),\nu)$ as
\begin{equation}\label{eq:norm}
\|f\|_1 := \int_1^{\infty} |f(x)|x^{-\beta}\dd{x}. 
\end{equation} 
\end{definition}
Recall that $\mu_{W,m}$ denotes the law of the truncated weights $(W_x^m)_x$, given as
\begin{align*}
    \mu_{W,m}(\cdot) = c_m^{-1}\mu_w(\cdot)\mathbbm{1}_{\{\cdot \leq m\}},
\end{align*}
where $c_m=1-m^{-(\tau-1)}$ is a normalizing constant converging to $1$ as $m$ tends to infinity, and $\mu_W$ is the Pareto law defined in~\eqref{eq:paretolaw}. For $z\in \C^+$, let $T_z$ denote the map \begin{align}\label{eq:Tz map}
    T_zf(\cdot) = \ota\int_0^{\infty}\e^{\ota tz}\exp\left\{ \ota t\int_1^{\infty} f(y)\kappa_\sigma(\cdot,\,y)\mu_W(\dd{y})   \right\}\dd{t}.
\end{align}
Then, we have the following result. 
\begin{lemma}\label{lemma:Tz contraction}
There exists a constant $\tilde c=\tilde c(\tau, \sigma, \beta)$ such that, for all $z\in \C^+$ with $\Im(z)^2=\eta^2>\tilde c$   $T_z:\mathcal{B}\to\mathcal{B}$ is a contraction mapping, with a contraction constant $\tilde{c}\eta^{-2}$.
\end{lemma}
\begin{proof}[Proof of Lemma \ref{lemma:Tz contraction}]
    We first need to show that, for any $f\in\Ba$, one has $T_zf\in\Ba$. Indeed, for $x\geq 1$ it holds that \begin{align*}
        \big|T_zf(x)\big| &
        \leq \int_0^{\infty}\e^{-\eta t}\left|\exp\left\{ \ota t\int_1^{\infty} f(y)\kappa_\sigma(x,y)\mu_W(\dd{y})   \right\} \right|\dd{t} \leq \frac{1}{\eta},
    \end{align*}
    where the last inequality holds as $f(y)\in\overline{\C^+}$ for any $y\geq 1$, and thus the second complex exponential is bounded by 1. Since $|T_zf(\cdot)|$ is uniformly bounded, it is $L^1-$integrable with respect to $\nu$, and so $T_z(\Ba)\subseteq \Ba$. 
    
    Now, we wish to show $T_z$ is a contraction. Let us take $f_1,f_2\in \Ba$. 
    Recall that for any $z_1,z_2\in\overline{\C^+}$ and $t>0$, we have 
    \begin{equation}\label{eq:complex difference identity}
        |\e^{\ota tz_1} - \e^{\ota tz_2}| \leq t|z_1-z_2|.
        \end{equation} Then, for any $x\in [1,\infty)$ we have  that 
    \begin{align}\label{eq:contraction step 1}
        |T_zf_1(x) - T_zf_2(x)| &= \left|  \ota\int_0^{\infty}\e^{\ota tz}\left(\e^{ \ota t\int_1^{\infty} f_1(y)\kappa_\sigma(x,y)\mu_W(\dd{y}) } - \e^{ \ota t\int_1^{\infty} f_2(y)\kappa_\sigma(x,y)\mu_W(\dd{y})   }\right)\dd{t}    \right| \nonumber \\
        &\leq \int_0^{\infty}\e^{-\eta t}\left|\e^{ \ota t\int_1^{\infty} f_1(y)\kappa_\sigma(x,y)\mu_W(\dd{y}) } - \e^{ \ota t\int_1^{\infty} f_2(y)\kappa_\sigma(x,y)\mu_W(\dd{y})   } \right| \dd{t} \nonumber\\
        &\leq \int_0^{\infty}\e^{-\eta t}t\left|\int_1^{\infty} \left(f_1(y)-f_2(y)\right)\kappa_\sigma(x,y)\mu_{W}(\dd{y})    \right| \dd{t},
    \end{align}
    where in \eqref{eq:contraction step 1} we use \eqref{eq:complex difference identity}. Now, evaluating the integral over $t$ in \eqref{eq:contraction step 1}, we obtain 
    \begin{equation}\label{eq:contraction step 2}
        |T_zf_1(x) - T_zf_2(x)| \leq \frac{(\tau-1)}{\eta^2}\int_1^{\infty}|f_1(y)-f_2(y)|\kappa_\sigma(x,y)y^{-\tau}\dd{y},
    \end{equation}
    where we explicitly write down the Pareto law $\mu_W(\dd{y}) := (\tau-1)y^{-\tau}\dd{y}.$ Recall that $\kappa_\sigma(x,y)=(x\wedge y)(x\vee y)^{\sigma}$. Thus, \eqref{eq:contraction step 2} becomes
    \begin{align*}
        |T_zf_1(x) - T_zf_2(x)| \leq & \frac{\tau-1}{\eta^2}\Big(\int_1^{x}|f_1(y)-f_2(y)|xy^{\sigma-\tau}\dd{y}+ \int_x^{\infty}|f_1(y)-f_2(y)|x^{\sigma}y^{1-\tau}\dd{y}\Big).
    \end{align*}
    Integrating with respect to $\nu$ gives us 
    \begin{align}\label{eq:contraction step 3}
     \|  T_zf_1 - & T_zf_2 \|_1 \nonumber\\
     \leq &\frac{\tau-1}{\eta^2}\int_1^{\infty} \left( x\int_1^x |f_1(y)-f_2(y)| y^{\sigma-\tau}\dd{y} + x^{\sigma}\int_x^{\infty}|f_1(y)-f_2(y)|y^{1-\tau}\dd{y}  \right)x^{-\beta}\dd{x} \nonumber\\
     =&\frac{\tau-1}{\eta^2}\left( \int_1^{\infty}|f_1(y)-f_2(y)|y^{\sigma-\tau}\int_y^{\infty}x^{1-\beta}\dd{x}\dd{y} \right.\nonumber
     \\&\left.+ \int_1^{\infty}|f_1(y)-f_2(y)|y^{1-\tau}\int_1^y x^{\sigma-\beta}\dd{x}\dd{y} \right)\,.
    \end{align}
Using $\beta>2$, the first integral in \eqref{eq:contraction step 3} can be bounded by
\begin{align}\label{eq:contraction step 3.1}
    \int_1^{\infty}|f_1(y)-f_2(y)|y^{\sigma-\tau}\int_y^{\infty}x^{1-\beta}\dd{x}\dd{y} 
    &= c_1\int_1^{\infty} |f_1(y)-f_2(y)|y^{-\beta}y^{2+\sigma-\tau}\dd{y} \nonumber\\
    &\leq c_1\|f_1 - f_2\|_1\,, 
\end{align}
since $y^{2+\sigma-\tau}\leq 1$ and $c_1=1/(\beta-2)$. Similarly, the second integral in \eqref{eq:contraction step 3} gives us 
\begin{align}\label{eq:contraction step 3.2}
   \int_1^{\infty}|f_1(y)-f_2(y)|y^{1-\tau}\int_1^y x^{\sigma-\beta}\dd{x}\dd{y} 
   &\leq  c_2 \int_1^{\infty}|f_1(y)-f_2(y)|y^{1-\tau}\dd{y}\nonumber\\
   &\leq c_2\|f_1-f_2\|_1,
\end{align}
with $c_2= 1/(\beta-1-\sigma)$, where for the last line we have used $1-\tau<-\beta$. Combining \eqref{eq:contraction step 3.1} and \eqref{eq:contraction step 3.2} in \eqref{eq:contraction step 3} gives us that 
\begin{equation}\label{eq:contraction inequality}
    \|T_zf_1 - T_zf_2\|_1 \leq \frac{\tilde c}{\eta^2}\|f_1-f_2\|_1\,,
\end{equation}
where $\tilde{c}$ is a constant depending on $\tau,\,\sigma$ and $\beta$. Thus, taking $\eta>0$ to be sufficiently large such that $\eta>\sqrt{\tilde{c}}$ gives us that $T_z$ is a contraction mapping on $\Ba$, hence proving the result. 
\end{proof}
The following corollary is immediate from the Banach fixed-point theorem for contraction mappings. 
\begin{corollary}\label{corollary:Fixed point corollary}
    Let $T_z:\Ba\to\Ba$ be the contraction map given in \eqref{eq:Tz map}. Then, there exists a unique analytic function $a_z^*\in\Ba$ such that $
    T_z(a_z^*)=a_z^*.$ 
\end{corollary}

We know from \eqref{eq:a(z,x) complex identity} that
\begin{equation}\label{eq:a(z,x) fixed point}
    a_z(x) = \ota\int_0^{\infty}\e^{\ota tz}\exp\left\{ \ota t\int_1^{\infty} c_m^{-1}a_z(y)\kappa_\sigma(x,y)\mathbbm{1}_{\{y\leq m\}}\mu_W(\dd{y}) \right\}\dd{t}\,.
\end{equation}
Define $\tilde{a}_z$ as 
\begin{equation}\label{eq:a_tilde}
    \tilde{a}_z(x) = \ota\int_0^{\infty}\e^{\ota tz}\exp\left\{ \ota t\int_1^{\infty} c_m^{-1}a_z(y)\kappa_\sigma(x,y)\mu_W(\dd{y}) \right\}\dd{t}\,.
\end{equation}
Then, $\tilde{a}_z = T_z(c_m^{-1}a_z)$. We now have the following lemma. 
\begin{lemma}\label{lemma:a - a_tilde}
    Let $a_z$ and $\tilde{a}_z$ be as in \eqref{eq:a(z,x) fixed point} and \eqref{eq:a_tilde}, respectively. Then, 
    \[
    \|a_z - \tilde{a}_z\|_1 \leq \frac{C(m)}{\eta^3}\,,
    \]
    where $C(m)$ is a constant depending on $m$ such that $\lim_{m\to\infty}C(m)=0$.
\end{lemma}

\begin{proof}[Proof of Lemma \ref{lemma:a - a_tilde}]
    Since $a_z\in\Ba$, we again use \eqref{eq:complex difference identity} to get 
    \begin{align}\label{eq:a-atilde step 1}
        |a_z(x) - \tilde{a}_z(x)| &\leq \int_0^{\infty} \e^{-\eta t}t\left| \int_m^{\infty} c_m^{-1}a_z(y)\kappa_\sigma(x,y)\mu_W(\dd{y}) \right|\dd{t} \nonumber\\
        &\leq \frac{\tau-1}{c_m\eta^2}\int_m^{\infty}|a_z(y)|\kappa_\sigma(x,y)y^{-\tau}\dd{y}\,,
    \end{align}
    where we evaluate the integral over $t$ to get the factor of $\eta^{-2}$ in \eqref{eq:a-atilde step 1}. Recall that $c_m=1-m^{-(\tau-1)}$. Using \eqref{eq:a(z,x) bound}, we have that 
    \begin{equation}\label{eq:a-atilde step 2}
        |a_z(x) - \tilde{a}_z(x)| \leq \frac{\tau-1}{c_m\eta^3}\int_m^{\infty}\kappa_\sigma(x,y)y^{-\tau}\dd{y}.
    \end{equation}
    Since $\kappa_\sigma(x,y)\leq (xy)^{1\vee\sigma}$, we have
    \begin{equation}\label{eq:a-atilde step 3}
        |a_z(x) - \tilde{a}_z(x)| \leq \frac{\tau-1}{c_m\eta^3}x^{1\vee\sigma}\int_m^{\infty} y^{(1\vee\sigma)-\tau}\dd{y} = \frac{(\tau-1) m^{(1\vee\sigma)-(\tau-1)}}{c_m((\tau-1)-(1\vee \sigma))\eta^3}x^{1\vee\sigma},
    \end{equation}
    where we use the fact that $\tau>\max(2,1+\sigma)$, and so the integral evaluated in \eqref{eq:a-atilde step 3} is finite. Define 
    $$c(m) := \frac{(\tau-1) c_m^{-1}m^{(1\vee\sigma)-(\tau-1)}}{(\tau-1)-(1\vee\sigma)}.$$ Since $c_m $ tends to one, and $m^{(1\vee\sigma)-(\tau-1)} $ tends to zero we have $c(m) = o_m(1)$. Now, integrating both sides of \eqref{eq:a-atilde step 3} against $x^{-\beta}\dd{x}$ gives us 
    \begin{equation}
        \|a_z - \tilde{a}_z\|_1 \leq \frac{c(m)}{\eta^3}\int_1^{\infty}x^{1\vee\sigma - \beta}\dd{x} = \frac{C(m)}{\eta^{3}}\,,
    \end{equation}
    since $\beta>2\vee1+\sigma$, and where $C(m)=o_m(1)$, completing the proof. 
\end{proof}

We are now at the penultimate step, where we have the necessary tools to show the convergence of $a_z$ to $a_z^*$ in the space $\Ba$. 
\begin{lemma}\label{lemma:a_z convergence}
    Let $a^*_z$ be the unique fixed point of the contraction map $T_z$ defined in \eqref{eq:Tz map}. Then, we have that 
    \begin{equation}
       \lim_{m\to\infty} \|a_z-a_z^*\|_1 = 0\, .
    \end{equation}
\end{lemma}

\begin{proof}[Proof of Lemma \ref{lemma:a_z convergence}]
    We have, using Lemma \ref{lemma:a - a_tilde} and the fact that $T_z$ is a contraction,  that 
    \begin{align*}
        \|a_z - a_z^*\|_1 &\leq \|a_z - \tilde{a}_z\|_1 + \|\tilde{a}_z - a^*_z\|_1  \nonumber\\
        &\leq C(m)\eta^{-3} + \|T_z(c_m^{-1}a_z) - T_z(a^*_z)\|_1 \nonumber\\
        &\leq C(m)\eta^{-3} + \tilde{c}\eta^{-2}\|c_m^{-1}a_z - a^*_z\|_1 \nonumber \\
        &\leq C(m)\eta^{-3} +\tilde{c}\eta^{-2}c_m^{-1}\|a_z-a^*_z\|_1 + \tilde{c}\eta^{-2}\|a^*_z\|_1|c_m^{-1}-1|\,.
    \end{align*}
    Thus, choosing $\eta>0$ such that $0<1-\tilde{c} c_m^{-1}\eta^{-2} <1$, we have that 
    \begin{align}\label{eq:a-a*}
        \|a_z - a^*_z\|_1 \leq \frac{1}{1-C_\tau c_m^{-1}\eta^{-2}}\left(C(m)\eta^{-3} + C_\tau\eta^{-2}\|a_z^*\|_1|c_m^{-1}-1|\right)\,.
    \end{align}
    Now, as $m\to\infty$, we have that $C(m)\to0$, and $c_m\to 1$. Since $\|a_z^*\|<\infty$, we have that the right-hand side of \eqref{eq:a-a*} goes to 0 as $m\to\infty$. Thus, $\|a_z-a^*_z\|_1\to0$ as $m\to\infty$ for $z$ in an appropriate domain $D_\eta\subset \C^{+}$. However, in the complex variable $z$, the domains of $a_z$ and $a^*_z$ are $\C^+$. Since the convergence holds for an open set of this domain (that is, in $D_\eta\subset\C^+$), by the identity theorem of complex analysis, the convergence holds everywhere in $\C^+$, that is, for each $z\in\C^+$.
\end{proof}

We now proceed towards a proof of Theorem \ref{theorem:stieltjesfinal}, and to achieve this we wish to take the limit $m\to\infty$ to characterise $\mathrm{S}_{\mu_{\sigma,\tau}}$. We know that since $\lim_{m\to\infty}\mu_{\sigma,\tau,m}= \mu_{\sigma,\tau}$, then for each $z\in\C^+$, $\lim_{m\to\infty}\St_{\mu_{\sigma,\tau,m}}(z)=\St_{\mu_{\sigma,\tau}}(z)$. 

\begin{proof}[Proof of Theorem \ref{theorem:stieltjesfinal}]
    Let $a_z^*$ be the unique fixed point of the contraction mapping $T_z$ as in Corollary \ref{corollary:Fixed point corollary}, and let $\St_{\mu_{\sigma,\tau}}(z)$ be the Stieltjes transform of $\mu_{\sigma,\tau}$ for any $z\in\C^+$. We wish to show that  
    \[
    \St_{\mu_{\sigma,\tau}}(z) = \int_1^{\infty}a_z^*(x)\mu_W{(\dd{x})}.
    \] 
    We have that 
    \begin{align}\label{eq:Stieltjes convergence step 1}
        &\left| \int_1^{\infty}a_z(x)\mu_{W,m}({\dd{x}}) - \int_1^{\infty}a_z^*(x)\mu_W({\dd{x}}) \right| \nonumber\\
        &\leq \left| \int_1^{\infty}a_z(x)\mu_{W,m}({\dd{x}}) - \int_1^{\infty}a_z^*(x)\mu_{W,m}({\dd{x}})     \right| + \left|  \int_1^{\infty}a_z^*(x)\mu_{W,m}({\dd{x}}) - \int_1^{\infty}a_z^*(x)\mu_W({\dd{x}})      \right|.
    \end{align}
    The first term in \eqref{eq:Stieltjes convergence step 1} can be evaluated as 
    \begin{align}\label{eq:Stieltjes convergence step 2.1}
        \left| \int_1^{\infty}a_z(x)\mu_{W,m}({\dd{x}}) - \int_1^{\infty}a_z^*(x)\mu_{W,m}({\dd{x}})     \right|  &\leq (\tau-1) c_m^{-1}\int_1^m|a_z(x)-a_z^*(x)| x^{-\tau}\dd{x} \nonumber\\
        &\leq (\tau-1) c_m^{-1}\int_1^{\infty}|a_z(x)-a_z^*(x)|x^{-\beta}x^{\beta-\tau}\dd{x} \nonumber\\
        &\leq (\tau-1) c_m^{-1}\|a_z-a_z^*\|_1={o}_m(1),
    \end{align}
    as $x^{\beta-\tau}\leq 1$, and $\|a_z-a^*_z\|_1=o_m(1)$ from Lemma \ref{lemma:a_z convergence}. The second term of \eqref{eq:Stieltjes convergence step 1} can be evaluated as 
    \begin{align}\label{eq:Stieltjes convergence step 2.2}
        &\left|  \int_1^{\infty}a_z^*(x)\mu_{W,m}({\dd{x}}) - \int_1^{\infty}a_z^*(x)\mu_W({\dd{x}})      \right|\nonumber\\
        &\leq c_m^{-1}\left|  \int_1^{m}a_z^*(x)\mu_{W}({\dd{x}}) - \int_1^{\infty}a_z^*(x)\mu_W({\dd{x}})      \right| + \left| \int_1^{\infty}a_z^*(x)\mu_W(\dd{x})   \right||c_m^{-1}-1| \nonumber\\
        &\leq \frac{(\tau-1)}{c_m\eta}\int_m^{\infty}x^{-\tau}\dd{x} + \frac{|c_m^{-1}-1|}{\eta} = \frac{(\tau-1) m^{1-\tau}}{c_m\eta} + \frac{|c_m^{-1}-1|}{\eta}= {o}_m(1),
    \end{align}
    since $|a_z^*|\leq \eta^{-1}$. 
    Combining \eqref{eq:Stieltjes convergence step 2.1} and \eqref{eq:Stieltjes convergence step 2.2} completes the proof of the theorem. 
\end{proof}
\section*{Acknowledgements} 
\begin{wrapfigure}{l}{0.07\textwidth} 
  \vspace{-\intextsep}   \includegraphics[width=0.07\textwidth]{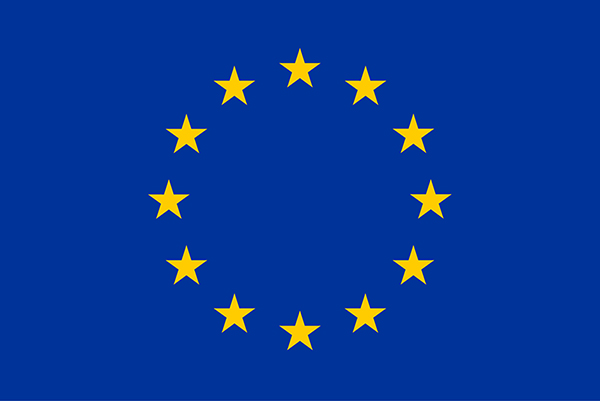}
\end{wrapfigure}
\noindent A.C.~acknowledges the hospitality of Leiden University where part of this work was carried out, and of the NETWORKS program for funding. The work of R.S.H.~and N.M.~is supported in part by the Netherlands Organisation for Scientific Research (NWO) through the Gravitation NETWORKS grant 024.002.003. The work of N.M. is further supported by the European Union’s Horizon 2020 research and innovation programme under the Marie Skłodowska-Curie grant agreement no. 945045. M.S.~is supported by
the MUR Excellence Department Project MatMod@TOV awarded to the Department of
Mathematics, University of Rome Tor Vergata, CUP E83C18000100006, and by the MUR 2022 PRIN project GRAFIA, project code 202284Z9E4. M.S. is also part of the INdAM group GNAMPA.
The authors thank Joost Jorritsma for pointing out an imprecision in a previous version of the draft.

\bibliographystyle{abbrvnat}
\bibliography{reference.bib}
\end{document}